\documentclass[a4paper]{article}

\newif\ifdraft
\draftfalse

\usepackage[T1]{fontenc}
\usepackage[utf8]{inputenc}  

\usepackage{amsthm}
\usepackage{amsfonts}
\usepackage{amsmath}
\usepackage{amssymb}

\usepackage{xypic} 
\usepackage{booktabs} 
\usepackage{mathtools}

\usepackage[nottoc,numbib]{tocbibind}
\usepackage[round,authoryear]{natbib}

\ifdraft
  \usepackage{calc}
  \newdimen{\tightmargin} 
  \usepackage[
    paperwidth=\textwidth+2\tightmargin,
    paperheight=\textheight+\footskip+2\tightmargin,
    margin=\tightmargin,
    bottom=\tightmargin+\footskip
  ]{geometry}
\fi
  
\usepackage{times}
\usepackage{upgreek} 

\usepackage{xcolor} 
\definecolor{darkgreen}{rgb}{0,0.2,0}
\definecolor{darkred}{rgb}{0.25,0,0}
\definecolor{darkblue}{rgb}{0,0,0.3}
\PassOptionsToPackage{obeyspaces}{url}
\usepackage[breaklinks,colorlinks,citecolor=darkgreen,linkcolor=darkblue,urlcolor=darkblue]{hyperref}
\usepackage{enumitem} 
\setlist{itemsep=0.25ex}

\usepackage[capitalise,nameinlink,sort]{cleveref} 

\usepackage{ifmtarg}
\usepackage{mathpartir} 

{
\theoremstyle{plain}
\newtheorem{theorem}{Theorem}[section]
\newtheorem{proposition}[theorem]{Proposition}
\newtheorem{lemma}[theorem]{Lemma}
\newtheorem{corollary}[theorem]{Corollary}
}
{
\theoremstyle{definition}
\newtheorem{definition}[theorem]{Definition}
\newtheorem{example}[theorem]{Example}

}

\numberwithin{equation}{section}

\makeatletter
\newenvironment{propositionwithqed}[1][]
  {\@ifmtarg{#1}{\begin{proposition}}{\begin{proposition}[#1]}\pushQED{\qed}}
  {\popQED\end{proposition}}
\makeatother

\newcommand{\extrafootertext}[1]{%
    \bgroup
    \renewcommand\thefootnote{\fnsymbol{footnote}}%
    \renewcommand\thempfootnote{\fnsymbol{mpfootnote}}%
    \footnotetext[0]{#1}%
    \egroup
}

\newcounter{saveenumi}

\newif\ifshow
\showtrue
\usepackage{comment}
\ifshow
  \newcommand{\todo}[1]{\textcolor{darkred}{#1}}
  \newcommand{\placeholder}[1]{\textcolor{purple}{#1}}
  \newcommand{\internal}[1]{\textcolor{darkgreen}{#1}}
  \newcommand{\obsolete}[1]{\textcolor{gray}{#1}}
\else
  \newcommand{\todo}[1]{}
  \newcommand{\placeholder}[1]{}
  \newcommand{\internal}[1]{}
  \newcommand{\obsolete}[1]{}
\fi

\ifshow

\else
  \excludecomment{longplaceholder}
  \excludecomment{longtodo}
  \excludecomment{longinternal}
  \excludecomment{longobsolete}
\fi

\newcommand{\defemph}[1]{\emph{\textbf{#1}}}
\newcommand{\defeq}{\coloneqq}

\newcommand{\N}{\mathbb{N}}

\newcommand{\injto}{\hookrightarrow} 
\makeatletter
\newcommand{\idmap}[1][]{\@ifmtarg{#1}{\operatorname{id}}{\operatorname{id}_{#1}}} 
\makeatother

\newcommand{\tuple}[2]{\langle #1 \rangle_{#2}} 

\newcommand{\set}[1]{\{#1\}}
\newcommand{\such}{\mid}

\newcommand{\initialSegment}[2][]{{\downarrow}_{#1}#2} 

\newcommand{\all}[1]{\forall #1 \,.\,}

\newcommand{\lthen}{\Rightarrow}

\DeclareMathOperator{\args}{arg} 
\newcommand{\argclass}[2]{\operatorname{cl}_{#1} #2} 
\newcommand{\argbinder}[2]{\operatorname{bind}_{#1} #2} 
\DeclareMathOperator{\simplearity}{simp} 

\newcommand{\class}[1]{\operatorname{c}_{#1}} 
\newcommand{\arity}[1]{\upalpha_{#1}} 

\newcommand{\mto}{\colon} 

\makeatletter
\newcommand{\Expr}[3]{\operatorname{Expr}^{#1}_{#2}\@ifmtarg{#3}{}{(#3)}} 
\makeatother
\newcommand{\ExprTy}[2]{\Expr{\Ty}{#1}{#2}} 
\newcommand{\ExprTm}[2]{\Expr{\Tm}{#1}{#2}} 

\newcommand{\Context}[1]{\operatorname{Cxt}{#1}} 
\newcommand{\Judg}[1]{\operatorname{Judg}{#1}} 
\newcommand{\RawRule}[1]{\operatorname{RawRule}{#1}} 

\newcommand{\ctxextend}[2]{#1 \mathbin{.} #2}  

\newcommand{\SequentialRule}[2]{#1 \Longrightarrow #2} 
\newcommand{\RuleBoundary}[2]{#1 \Longrightarrow #2} 

\newcommand{\wfreplace}[1]{#1^{\mathrm{wf}}} 

\newcommand{\mvextend}[2]{#1 + #2}  
\newcommand{\Inst}[3]{\operatorname{Inst}_{{#1},{#2}}(#3)} 

\newcommand{\mvpromote}[1]{\widetilde{#1}} 

\newcommand{\Ty}{\mathsf{ty}} 
\newcommand{\Tm}{\mathsf{tm}} 
\newcommand{\TyEq}{\mathsf{tyeq}} 
\newcommand{\TmEq}{\mathsf{tmeq}} 

\DeclareMathOperator{\ob}{ob} 

\newcommand{\Presup}[1]{\operatorname{Presup}{#1}} 

\newcommand{\congrule}[1]{{#1}_{\equiv}} 

\newcommand{\types}{\mathrel{\vdash}} 
\newcommand{\of}[1][1]{\mspace{#1 mu plus 1.0mu}\mathord{:}\mspace{#1 mu plus 1mu}} 
\newcommand{\type}{\;\mathsf{type}} 
\newcommand{\cxt}{\mathrel{\textsf{cxt}}} 
\newcommand{\emptycxt}{[\,]} 
\newcommand{\judgeq}{\equiv}

\newcommand{\typesjudgement}{\types}
\newcommand{\istype}[2]{#1 \typesjudgement #2 \type} 
\newcommand{\isterm}[3]{#1 \typesjudgement #2 : #3} 
\newcommand{\eqtype}[3]{#1 \typesjudgement #2 \judgeq #3} 
\newcommand{\eqterm}[4]{#1 \typesjudgement #2 \judgeq #3 : #4} 

\newcommand{\typesboundary}{\types} 
\newcommand{\bdryhead}{\Box} 
\newcommand{\qjudgeq}{\judgeq^{\scriptscriptstyle ?}}
\newcommand{\istypebdry}[1]{#1 \typesboundary \bdryhead \type} 

\newcommand{\iscxt}[1]{\types #1 \cxt} 

\newcommand{\plug}[2]{#1[#2]}

\newcommand{\cat}[1]{\mathcal{#1}} 
\newcommand{\Set}{\mathrm{Set}} 
\newcommand{\RawTh}{\mathrm{RawTh}} 
\newcommand{\Sig}{\mathrm{Sig}} 

\newcommand{\Scope}{\mathrm{Scope}} 

\newcommand{\act}[1]{#1_{*}} 
\newcommand{\tca}[1]{#1^{*}} 

\newcommand{\rename}[1]{\act{#1}} 

\DeclareMathOperator{\clos}{cl} 

\newcommand{\closureSystem}[1]{\mathcal{#1}} 
\newcommand{\csS}{\closureSystem S} 
\newcommand{\csT}{\closureSystem T} 
\DeclareMathOperator{\premises}{prems} 
\DeclareMathOperator{\conclusion}{concl} 
\newcommand{\derivation}[3]{\operatorname{Der}_{#1}(#2,#3)} 
\DeclareMathOperator{\hyp}{hyp} 
\newcommand{\der}[2]{\operatorname{der}(#1,#2)} 

\DeclareMathOperator{\ClosureRule}{Rule} 
\DeclareMathOperator{\ClosureSystem}{Clos} 

\newcommand{\StructuralRules}{\operatorname{Struct}} 

\newcommand{\position}[1]{|#1|} 
\newcommand{\emptyscope}{\mathsf{0}} 
\newcommand{\sumscope}[2]{#1 \oplus #2} 
\DeclareMathOperator{\inlscope}{inl} 
\DeclareMathOperator{\inrscope}{inr} 
\newcommand{\singletonscope}{\mathsf{1}}
\newcommand{\numscope}[1]{[{#1}]}
\newcommand{\numscopemap}[2]{w_{{#1},{#2}}}

\DeclareMathOperator{\Fam}{Fam} 
\DeclareMathOperator{\famindex}{ind} 
\newcommand{\famev}[1]{\operatorname{ev}_{#1}} 

\newcommand{\family}[1]{\langle #1 \rangle} 
\newcommand{\famtuple}[2]{\family{#1 \mid #2}} 
\newcommand{\fammap}[2]{\family{#1}_{#2}} 
\newcommand{\singletonfamily}[1]{\family{#1}} 
\newcommand{\emptyfam}{\langle \, \rangle} 

\newcommand{\inl}{\iota_0}
\newcommand{\inr}{\iota_1}

\newcommand{\synvar}[1]{\mathsf{var}_{#1}} 
\newcommand{\synmeta}[1]{\mathsf{meta}_{#1}} 

\newcommand{\symb}[1]{\mathsf{#1}} 
\newcommand{\symPi}{\Uppi} 
\newcommand{\symSigma}{\Upsigma} 
\newcommand{\symlambda}{\uplambda} 
\newcommand{\symapp}{\symb{app}} 

\newcommand{\synPi}[1][]{\symPi(#1)} 

\newcommand{\symA}{\symb{A}}
\newcommand{\symB}{\symb{B}}
\newcommand{\symC}{\symb{C}}
\newcommand{\symS}{\symb{S}}

\newcommand{\syma}{\symb{a}}
\newcommand{\symf}{\symb{f}}

\newcommand{\genapp}[1]{\widehat{#1}} 

\newcommand{\natty}[2]{\tau_{#1}(#2)} 

\newcommand{\prd}[1]{\symPi (#1) \,.\,} 

\newcommand{\coqident}[1]{#1} 

\begin{document}

\title{A general definition of dependent type theories}

\author{Andrej Bauer
  \and Philipp G.~Haselwarter
  \and Peter LeFanu Lumsdaine}

\date{September 11, 2020}

\maketitle

\extrafootertext{The authors benefited greatly from visits funded by the COST Action \href{https://eutypes.cs.ru.nl}{\emph{EUTypes}} CA15123.}
\extrafootertext{This material is based upon work supported by the U.S.~Air Force Office of Scientific Research under award number FA9550-17-1-0326, grant number 12595060.}

\begin{abstract}
  We define a general class of dependent type theories, encompassing Martin-Löf’s intuitionistic type theories and variants and extensions.
  The primary aim is pragmatic: to unify and organise their study, allowing results and constructions to be given in reasonable generality, rather than just for specific theories.
  Compared to other approaches, our definition stays closer to the direct or naïve reading of syntax, yielding the traditional presentations of specific theories as closely as possible.
  
  Specifically, we give three main definitions: \emph{raw type theories}, a minimal setup for discussing dependently typed derivability; \emph{acceptable type theories}, including extra conditions ensuring well-behavedness; and \emph{well-presented type theories}, generalising how in traditional presentations, the well-behavedness of a type theory is established step by step as the type theory is built up.
  Following these, we show that various fundamental fitness-for-purpose metatheorems hold in this generality.

  Much of the present work has been formalised in the proof assistant Coq.
\end{abstract}

\tableofcontents

\newpage

\section{Introduction}

\subsection{Overview}

  We give a general definition of dependent type theories,
  encompassing for example Martin-Löf’s intuitionistic type theories and many variants and extensions.

  The primary aim is to give a setting for formal general versions of various constructions and results,
  which in the literature have been given for specific theories but are heuristically understood to hold for a wide class of theories:
  for instance, the conservativity theorem of \citep{hofmann:syntax-and-semantics}, or the coherence theorems of \citep{hofmann:lcccs,lumsdaine-warren:local-universes}.
  
  This has been a sorely felt gap in the literature until quite recently;
  the present work is one of several recent approaches to filling it
  \citep{isaev17:_algeb,uemura19:_gener_framew_seman_type_theor,brunerie:_agda}.
  
  A secondary aim is to stick very closely to an elementary understanding of syntax.
  Established general approaches --- for instance, logical frameworks and categorical semantics --- give, in examples, not the original syntax of the example theories, but an embedded or abstracted version, typically then connected to the original syntax by adequacy or initiality theorems.
  Our approach directly recovers quite conventional presentations of the example theories themselves.
  
  As a corollary of this goal, we must confront the bureaucratic design decisions of syntax: the selection of structural rules, and so on.
  These are often swept under the rug in specific type theories as “routine”; to initiates they are indeed standard, but newcomers to the field often report finding this lack of detail difficult.
  We therefore elide nothing, and set out a precise choice of all such decisions, carefully chosen and proven to work well in reasonable generality, which we hope will be of value to readers.

  In pursuit of the above aims, we offer not one main definition of type theory, but three, at increasing levels of refinement.
  
  Firstly, we define \emph{raw type theories}, as a conceptually minimal description of traditional presentations of type theories by symbols and rules,
  sufficient to define the derivability relation, but not yet incorporating any well-formedness constraints on the rules.
  
  Secondly, we give sufficient conditions on a raw type theory to imply that derivability over it is well-behaved in various standard ways.
  Specifically, we isolate simple syntactic checks that suffice to imply core fitness-for-purpose properties, and package these into the notion of an \emph{acceptable type theory}.
  
  Thirdly, we analyse the well-founded nature of traditional presentations,
  involved in more elaborate constructions such as the categorical semantics,
  as well as (arguably) the intuitive assignment of meaning to a theory.  
  This leads us to the notion of \emph{well-presented type theories}, which we hope can serve as a full-fledged proposal fulfilling our primary aim.
  
\subsection{Specifics}

We aim, as far as possible, not to argue for any novel approach to setting up type theories, but simply to give a careful analysis of how type theories are traditionally presented, in order to lay out a generality in which such presentations can be situated.
As such, the first few components of our definition are the expected ones.

We begin with an appropriate notion of \emph{signature}, for untyped syntax with variable-binding, and develop the standard notions of ``raw'' \emph{syntactic expressions} over such signatures, including substitution, translation along signature morphisms, and so on.

With the syntax of types and terms properly set up, a type theory is traditionally presented by giving a collection of \emph{rules}.
Type theorists are very accustomed to reading these --- but as anyone who has tried to explain type theory to a non-initiate knows, there is a lot to unpack here.
The core of our definition is a detailed study of the situation: what is really going on when we write and read inference rules, check that they are meaningful, and interpret them as a presentation of type theory?

Take the formation rule for $\symPi$-types:
\[
  \inferrule* {
    { \istype \Gamma A }
    \and
    { \istype {\Gamma, x \of A} B }
  }
  { \istype \Gamma {\prd {x \of A} B} }
\]
When pressed to explain this, most type theorists will say that the rule represents \emph{inductive clauses} for constructing derivations, or \emph{closure conditions} for the derivability predicate:
given derivations of the judgements above the line, a derivation of the judgement below the line is constructed. In particular, if $\Gamma$, $A$ and~$B$ are syntactically valid representations of a context and types, and the judgements
\[
 \istype{\Gamma}{A},
 \qquad\text{and}\qquad
 \istype{\Gamma, x \of A}{B}
\]
are both derivable, then so is the judgement
\[
  \istype{\Gamma}{\prd {x \of A} B}.
\]
This understanding of rules is sufficient for explaining the definition of a specific type theory, and defining derivability of judgements.

However, it is in general too permissive: to be well-behaved, type theories should not be given by arbitrary closure conditions, but only by those that can be specified syntactically by rules looking something like the traditional ones.
In other words, we want to make explicit the idea of a rule as a syntactic entity in its own right, accompanied with a mathematically precise explanation of what makes it type-theoretically acceptable, and how it gives rise to a quantified family of closure conditions.

So to a first approximation, we say a rule consists of a collection of judgements --- its \emph{premises} --- and another judgement, its \emph{conclusion}.
However, a subtlety lurks: what are $\Gamma$, $A$, and $B$ in the above $\symPi$-formulation rule?

The symbol $\Gamma$ is easy: we can dispense with it entirely.
We prescribe (as type theorists often do, heuristically) that all rules should be valid over arbitrary contexts, and so since the arbitrary context is always present in the interpretation of a rule as a family of closure conditions, it never needs to be included in the syntactic specification of the rule.
This precisely justifies a common “abuse of notation” in presenting type theories: the context $\Gamma$ is omitted when writing down the rules, and one mentions apologetically somewhere that all rules should be understood as over an arbitrary ambient context.

Explaining $A$ and $B$ is more interesting.
They are generally called “metavariables”, and in the family of closure conditions they are indeed that --- quantified variables of the meta-theory, ranging over syntactic entities.
However, if the rule is to be considered as formal syntactic entity, $A$ and $B$ must themselves be part of that syntax.
We therefore take the premises and conclusion of rules as formed over the ambient signature of the type theory \emph{extended} with extra symbols $\symA$ and $\symB$ to represent the metavariables of the rule.

These considerations result in the notion of a \emph{raw rule}, whose premises and conclusion are raw judgements over a signature extended with metavariables. A \emph{raw type theory} is then just a family of raw rules. It holds enough information to be used, but still permits arbitrariness that must be dispensed with.
At the very least, the type and term expressions appearing in the rule ought to represent derivable types and terms, respectively.
Thus in the next stage of our definition we ask that every rule be accompanied with derivations showing that the \emph{presuppositions} hold, namely, that its type expressions are derivable types and that its terms have derivable types.
Another condition that we impose on rules is \emph{tightness}, which roughly requires that the metavariables symbols be properly typed by the premises, and for rules that build term or a type judgements, that they do so in the most general form.

Even though every rule of a raw type theory may be presuppositive and tight, the theory as a whole may be deficient, for example, if one of the symbols has no corresponding formation rule, or several of them. Overall we call a type theory \emph{tight} when there is bijective correspondence between its symbols and formation rules, which are tight themselves. And to make sure equality is well behaved, we also require that for each symbol there is a suitable congruence rule ensuring that the symbol commutes with equality. When a raw type theory has all these features, we call it an \emph{acceptable} type theory.

Because the derivations of presuppositions appeal to the very rules they certify, an unsettling possibility of circular reasoning arises.
We resolve the matter in two ways. 
First, we add one last stage to the definition of type theories and ask that all the rules, as well as the premises within each rule, be ordered in a well-founded manner.
Second, we show that for acceptable type theories whose contexts and premises are well-founded as finite sequences, circularities can always be avoided by passing to the \emph{well-founded replacement} of the theory (\cref{sec:well-founded-replacement}).
Apart from expelling the daemons of circularity, the well-founded order supplies a useful induction principle.

In the end, the definition of a general type theory has roughly five stages:
\begin{enumerate}
\item the \emph{signature} (\cref{def:signature}) describes the arities of primitive type and term symbols that form the \emph{raw syntax} (\cref{def:raw-syntax}),
\item \emph{raw rules} (\cref{def:raw-rule}) constitute a \emph{raw type theory} (\cref{def:raw-type-theory}),
\item the raw rules are verified to be \emph{tight} and \emph{presuppositive} (\cref{def:tight-rule,def:weakly-presuppositive-rule}, and therefore \emph{acceptable} (\cref{def:acceptable-rule}),
\item the raw type theory is verified to consist of acceptable symbol rules (\cref{def:symbol-rule}) and equations, that it is \emph{tight} and \emph{congruous}, and therefore \emph{acceptable} (\cref{def:theory-good-properties}).
\item finally, an acceptable type theory may be \emph{well-presented} (\cref{def:well-presented-type-theory}) and hence \emph{well-founded} (\cref{def:well-founded-theory}), or we may pass to its \emph{well-founded replacement} (\cref{thm:wf-replacement-equivalence}).
\end{enumerate}

We readily acknowledge that there are many alternative ways of setting up type theories, each serving a useful purpose. It is simply our desire to actually give \emph{one} mathematically complete description of what type theorie\emph{s} are in general.

Once the definition is complete, we should provide evidence of its scope and utility. We do so in \cref{sec:well-behavedness} by proving fundamental meta-theorems, among which are:
\begin{enumerate}
\item Derivability of presuppositions, \cref{thm:presuppositions}, stating that the presuppositions of a derivable judgement are themselves derivable.
\item Elimination of substitution, \cref{thm:elimination-substitution}, stating that anything that can be derived using the substitution rules can also be derived without them.
\item Uniqueness of typing, \cref{thm:tight-uniqueness-of-typing}, stating that a term has at most one type, up to judgmental equality.
\item An inversion principle, \cref{thm:inversion-principle}, that reconstructs the proof-relevant part of the derivation of a derivable judgement from the information given in the judgement.
\end{enumerate}

Our definitions are set up to support a meta-theoretic analysis of type theories, but deviate from how type theories are presented in practice.
First, one almost always encounters only finitary syntax in which contexts and premises are presented as finite sequences -- we call these \emph{sequential contexts and premises} and treat them in \cref{sec:sequential-contexts,sec:sequential-rules}.
Second, theories are not constructed in five stages, but presented through rules that are manifestly acceptable and free of circularities, while symbols are introduced simultaneously with their formation rules. In \cref{sec:well-presented-rules,sec:well-presented-theories} we make these notions precise by defining \emph{well-presented rules and theories}, and their realisations as raw type theories.

Having heard tales about minor but insidious mistakes in the literature on the meta-theory of syntax, we decided to protect ourselves from them by formalising parts of our paper in the Coq proof assistant~\citep{coq}. An overview of the formalisation is given in \cref{sec:formalisation-coq}, including comments about the meta-mathematical foundations sufficient for carrying out our work.
The formalisation allows us to claim a high level of confidence and omission of routine syntactic arguments. Nevertheless, we still strove to make the paper self-contained by following the established standards of informal rigour.

\subsection{Disclaimers}

Having said what this paper is about, it is worth saying a little about what it is \emph{not}.

It is most certainly not intended as a prescriptive definition of what all dependent type theories should be.
Many important type theories in the literature are not covered by our definition, and we do not mean to reject them.
The aim of this work is simply pragmatic: to encompass \emph{some large class} of theories of interest, in order to better organise and unify their study.
We very much hope our approach may be extended to wider generalities.

We do not claim or aim to supersede other general approaches to studying type theories, such as those based on logical frameworks.
Such approaches are well-developed, powerful for many applications, and sidestep some complications of the present approach.
However, all such approaches (that we are aware of) work by using a somewhat modified syntax (e.g.~embedded in a larger system) --- the syntax they yield is not obviously the same as the syntax given by a “direct” or “naïve” reading of presentations of theories.
They are typically accompanied by adequacy theorems, or similar, showing equivalence between the modified syntax and the naïve, for the specific type theory under consideration.

By contrast, we aim to directly study and generalise the naïve approach itself, which (to our knowledge) has not been done previously in such generality.
Our motivations are therefore largely complementary to such approaches.
A more detailed comparison is given in \cref{sec:discussion-and-related-work}.


\section{Preliminaries}
\label{sec:preliminaries}

We begin by setting up definitions and terminology of a general mathematical nature that we will use throughout.

\subsection{Families}

For several reasons, we work with \emph{families} in places where classical treatments would use either \emph{subsets} of, or \emph{lists} from, a given set.
While the term is standard (e.g. ``the product of a family of rings''), we make rather more central use of it than is usual, so we establish some notations and terminology.

\begin{definition}
  \label{def:family}%
  Given a set $X$, a \defemph{family $K$ of elements of $X$} (or briefly, a \defemph{family on~$X$}) consists of an index set $\famindex K$ and a map $\famev{K} : \famindex{K} \to X$.
  We let $\Fam X$ denote the collection of all families on~$X$, and use the \defemph{family comprehension} notation $\famtuple{e_i \in X}{i \in I}$ for the family indexed by~$I$ that maps~$i$ to~$e_i$. A family may be explicitly described by displaying the association of indices to values. For example, we may write $\family{0 \mto e_0, 1 \mto e_1, 2 \mto e_2}$ for the family $\famtuple{e_i}{i \in \set{0,1,2}}$.
\end{definition}

\begin{example}
  Any \emph{subset} $A \subseteq X$ can be viewed as a family $\famtuple{i \in X}{i \in A}$, with $\famindex A$ as $A$ itself and $\famev{A}$ the inclusion $A \injto X$.
  Motivated by this, we will often speak of a family~$K$ as if it were a subset, writing $x \in K$ rather than $x \in \famindex K$, and treating such $x$ itself as an element of $X$ rather than explicitly writing $\famev{K}(x)$.
\end{example}

\begin{example}
  Any \emph{list} $\ell = [x_0,\ldots,x_n]$ of elements of $X$ can be viewed as a family, with $\famindex \ell = \{ 0, \ldots, n \}$ and $\famev{\ell}(i) = x_i$, or equivalently $\ell = \family{0 \mto x_0, \ldots, n \mto x_n}$.
  We will often use list notation to present concrete examples of families.
\end{example}

Working constructively, it is quite important to keep the distinction between families and subsets where classical treatments would confound them.
For instance, a propositional theory is usually classically defined as a \emph{set} of propositions; we would instead use a \emph{family} of propositions.
In a derivation over the theory, uses of axioms therefore end up “tagged” with elements of the index set of the theory, typically explaining how a certain proposition arises as an axiom (since the same proposition might occur as an instance of axiom schemes in multiple ways).
These record constructive content which may be needed for, say, interpreting axioms according to a proof by cases over the axiom schemes of the theory.

Our use of families where most traditional treatments use lists --- e.g.\ for specifying the argument types of a constructor --- is less mathematically significant.
It is partly to avoid baking in assumptions of finiteness or ordering where they are not required; but it is mostly motivated just by the formalisation, where families provide a more appropriate abstraction.

\begin{definition}
  \label{def:family-map}%
  A \defemph{map of families} $f : K \to L$ between families~$K$ and~$L$ on~$X$ is a map
  $f : \famindex K \to \famindex L$ such that $\famev{L} \circ f = \famev{K}$.
\end{definition}

We shall notate such a map as $\fammap{f(x)}{x \in \famindex K}$. Indeed, the notation $\fammap{f(x)}{x \in A}$ works for \emph{any} maps $f : A \to B$, as it is just an alternative way of writing $\lambda$-abstractions.

Families and their maps form a category $\Fam X$, which is precisely the slice category $\Set/X$.
A map $r : X \to Y$ yields a functorial action $\act{r} : \Fam X \to \Fam Y$ which takes $K \in \Fam X$ to the family $\act{r} K$ with $\famindex (\act{r} K) = \famindex K$ and $\famev{\act{r} K} = r \circ \famev{K}$. It is perhaps clearer to write down the action in terms of family comprehension: $\act{r} \famtuple{e_i}{i \in I} = \famtuple{r(e_i)}{i \in I}$.

\begin{definition}
  \label{def:family-map-over}%
  Given a function $r : X \to Y$ and families $K$, $L$ on $X$, $Y$ respectively, a \defemph{map $f : K \to L$ over $r$} is a map $f : \act{r} K \to L$; equivalently, a map $f : \famindex K \to \famindex L$ forming a commutative square over $r$.
\end{definition}

\subsection{Closure systems}

The general machinery of derivations as closure systems occurs throughout logic, and is independent of the specific syntax or judgements of the logical systems involved.

\begin{definition}
  \label{def:closure-rule}\label{def:closure-system}%
  A \defemph{closure rule $(P, c)$} on a set~$X$ consists of a family $P$ of elements in~$X$, its \defemph{premises}, and a \defemph{conclusion} $c \in X$.
  A \defemph{closure system~$\csS$} on a set~$X$ is a family of closure rules on~$X$, where we respectively write $\premises R$ and $\conclusion R$ for the premises and the conclusion corresponding to a rule $R \in \csS$.
  We write $\ClosureRule X$ and $\ClosureSystem X$ for the collections of closure rules and closure systems on~$X$, respectively.
\end{definition}

As is tradition, we display a closure rule with premises $[p_1,\ldots,p_n]$ and conclusion~$c$ as
\[
  \inferrule{p_1 \quad \cdots \quad p_n}{c}
\]

The constructions of closure rules and closure systems are evidently functorial in the ambient set.
A map $f : X \to Y$ sends a rule $R \in \ClosureRule X$ to the rule $\act{f} R \in \ClosureRule Y$ with $\premises (\act{f} R) \defeq \famtuple{f(p)}{p \in \premises R}$ and $\conclusion (\act{f} R) \defeq f(\conclusion R)$.
Similarly, a closure system $\csS$ on $X$ is taken to the closure system $\act{f} \csS$ on $Y$, defined by $\act{f} \csS \defeq \famtuple{\act{f} r}{r \in \csS}$.

\begin{definition}
  A \defemph{simple map} $\csS \to \csT$ between closure systems $\csS$ and $\csT$ on~$X$ is just a map between them as families.
  More generally, a \defemph{simple map $\bar{f} : \csS \to \csT$ over $f : X \to Y$} from $\csS \in \ClosureSystem X$ to $\csT \in \ClosureSystem Y$
  is just a simple map $\bar{f} : \act{f}\csS \to \csT$, or equivalently a family map~$\bar{f}$ over $\act{f} : \ClosureRule X \to \ClosureRule Y$.
\end{definition}

A closure system yields a notion of derivation:

\begin{definition}
  \label{def:closure-system-derivation}%
  Given a closure system $\csS$ on $X$, a family $H$ of elements in $X$, and an element $c \in X$, the
  \defemph{derivations~$\derivation{\csS}{H}{c}$ of~$c$ from hypotheses~$H$} are inductively generated
  by:
  \begin{enumerate}
  \item for every $h \in H$, there is a corresponding derivation $\hyp h \in \derivation{\csS}{H}{h}$,
  \item for every rule $R \in \csS$ and a map $D \in \prod_{p \in \premises R} \derivation{\csS}{H}{p}$ there is a derivation $\der{R}{D} \in \derivation{\csS}{H}{\conclusion R}$.
  \end{enumerate}
\end{definition}

In the second clause above $D$ is a dependent map, i.e., for each $p \in \premises R$ we have $D_p \in \derivation{\csS}{H}{p}$.
We do not shy away from using products of families and dependent maps when the situation demands them.

The elements of $\derivation{\csS}{H}{c}$ may be seen as well-founded trees with edges and nodes suitably labelled from $X$, $\csS$, and $H$.
We take such inductively generated families of sets as primitive;
their existence may be secured one way or another, depending on the ambient mathematical foundations.
The essential feature of derivations, which we rely on, is the structural induction principle they provide.

It is easy to check that derivations are functorial in simple maps of closure systems, in a suitable sense:
\begin{propositionwithqed}
  A simple map $\bar{f} : \csS_X \to \csS_Y$ of closure systems over $f : X \to Y$ acts on derivations as $\act{\bar{f}} : \derivation{\csS_X}{H}{c} \to \derivation{\csS_Y}{\act{f}H}{f(c)}$ for each $H$ and $c$.
  The action is moreover functorial, in that $\act{\idmap} = \idmap$ and $\act{(\bar{f} \circ \bar{g})} = \act{\bar{f}} \circ \act{\bar{g}}$.
\end{propositionwithqed}

Often, one wants a more general notion of map, sending each rule of the source system not necessarily to a single rule of the target system, but instead to a \emph{derived} rule:

\begin{definition}
  \label{def:derivation-grafting}%
  A \defemph{derivation of a rule} $R$ over a closure system $\csS$ is a derivation of $\conclusion R$ from $\premises R$ over $\csS$.
  Given such a derivation, we call $R$ a \defemph{derived rule} of $\csS$, or say $R$ is \defemph{derivable over $\csS$}.
  A \defemph{map of closure systems} $\bar{f} : \csS \to \csT$ over $f : X \to Y$ is a function giving, for each rule $R$ of $C$, a derivation of $\act{f}R$ in $\csT$.
\end{definition}

To show that maps of closure systems preserve derivability, we need a \emph{grafting} operation on derivations.

\begin{lemma}
  \label{lem:hypotheses-grafting}
  Given an ambient closure system $\csS$, suppose $D$ is a derivation of $c$ from hypotheses~$H$ over $\csS$, and for each $h \in H$, $D_h$ is a derivation of $h$ from $H'$.
  Then there is a derivation of $c$ from $H'$ over $\csS$.
\end{lemma}

\begin{proof}
  The derivation of~$c$ from~$H'$ is constructed inductively from a derivation of~$c$ from~$H$:
  \begin{enumerate}
  \item if $c$ is derived as one of the hypotheses $h \in H$, then $D_h$ derives $c$ from $H'$,
  \item if $\der{R}{D'}$ derives $c$ from $H$, then for each $p \in \premises R$ we inductively obtain a derivation $D''_p$ of $p$ from $H'$ from the corresponding derivation~$D'_p$ of $p$ from $H$, and assemble these into the derivation $\der{R}{D''}$ of~$c$ from~$H'$. \qedhere
  \end{enumerate}
\end{proof}

\begin{definition}
  A closure system map $\bar{f} : \csS \to \csT$ over $f : X \to Y$ \defemph{acts on derivations}: if $D$ is a derivation of $c$ from $H$ over $\csS$, there is a derivation $\act{\bar{f}}D$ of $f(c)$ from $\act{f}H$ over $\csT$.
\end{definition}

\begin{proof}
  $\act{\bar{f}}D$ is defined by recursion on $D$.
  Wherever $D$ uses a rule $R$ of $\csS$, with derivations $D_h$ of the premises, $\act{\bar{f}}D$ uses the given derivation $\bar{f}(R)$ of $\act{f}R$, with the derivations $\act{\bar{f}}D_h$ grafted in at the hypotheses.
\end{proof}

Categorically, grafting can be recognised as the multiplication operation of a monad structure on derivations, and our maps of closure systems can be seen as Kleisli maps for this monad (relative to simple maps).
One can thus show that they form a category, that the action on derivations is functorial, and so on.
We do not make this precise here, as it is not required for the present paper.

\subsection{Well-founded orders}

There will be several occasions when we shall have to prevent dependency cycles (between premises of a rule, or between rules of a type theory). For this purpose we review a notion of well-foundedness which accomplishes the task.

\begin{definition}
  \label{def:well-founded-order}%
  A \defemph{strict partial order} on a set $A$ is an irreflexive and transitive relation~$<$ on~$A$.
  A subset $S \subseteq A$ is \defemph{$<$-progressive} when, for all $x \in A$,
  \begin{equation*}
    (\all{y \in A} y < x \lthen y \in S) \lthen x \in S.
  \end{equation*}
  A \defemph{well-founded order} is a strict partial order~$<$ in which a subset is the entire set as soon as it is $<$-progressive.
  For each $x \in A$, the \defemph{initial segment} $\initialSegment{i} \defeq \set{y \in A \such y < x}$ is the set of elements preceding~$x$ with respect to the order.
\end{definition}

\noindent
In terms of an induction principle a strict partial order is well-founded when, for every predicate~$\varphi$ on~$A$,
\begin{equation*}
  \all{x \in A}{
    (\all{y \in A} y < x \lthen \varphi(y)) \lthen \varphi(x)
  } \lthen
  \all{x \in A}{\varphi(x)}.
\end{equation*}
Classically there are many equivalent definitions of well-founded orders.
Constructively, the situation is more complicated, cf.\ \cite[\textsection 2.5]{taylor:practical-foundations}; this definition is one of the most standard, and the most suited to our purpose.


\section{Raw syntax}
\label{sec:raw-syntax}

In this section, we set out our treatment of raw syntax with binding, sometimes called ``pre-syntax'' to indicate that no typing information is present at this stage.
There is nothing essentially novel --- briefly, we use a standard modern treatment, closely inspired by that of \citep{fiore-plotkin-turi}, but focus on concrete constructions rather than categorical characterisations.
So we take raw expressions as inductively generated trees, and scope systems, developed below, for keeping track of variable scopes and binding.
We spell out the details in order to have a self-contained presentation tailored to our requirements, and to set up terminology and notation we will use later.

\subsection{Scope systems}

We first address the question of how to treat variables and binding.
Should we use terms with named variables up to $\alpha$-equivalence, or de Bruijn indices, or reuse the binding structure of a framework language?
The last option is appealing, as it dispenses with many cumbersome details,
but we shall avoid it precisely because we want to confront the cumbersome details of type theory.
%

Rather than choosing a particular answer, we formulate and use a general structure for binding, abstracting away the implementation-specific details of several approaches, but retaining the common structure required for defining syntax.

\begin{definition}
  \label{def:scope-system}
  A \defemph{scope system} consists of:
  \begin{enumerate}
  \item a collection of \defemph{scopes}~$\cat{S}$;
  \item for each scope $\gamma$, a set of \defemph{positions}~$\position{\gamma}$;
  \item an \defemph{empty scope~$\emptyscope$} with no positions, $\position{\emptyscope} = \emptyset$;
  \item a \defemph{singleton scope} $\singletonscope$ with a unique position, $\position{\singletonscope} = 1$;
  \item operations giving for all scopes $\gamma$ and $\delta$ a \defemph{sum scope} $\sumscope{\gamma}{\delta}$, and functions
    \begin{equation*}
      \xymatrix{
        {\position{\gamma}} \ar[r]^{\inlscope} &
        {\position{\sumscope{\gamma}{\delta}}} &
        {\position{\delta}} \arrow[l]_{\inrscope}
      }
    \end{equation*}
    exhibiting $\position{\sumscope{\gamma}{\delta}}$ as a coproduct of
    $\position{\gamma}$ and $\position{\delta}$.
  \end{enumerate}
\end{definition}

\noindent
A scope may be seen as “a context, without the type expressions”: in raw syntax, one cares about what variables are in scope, without yet caring about their types.

The singleton scope $\singletonscope$ is not needed for most of the development of general type theories --- in the present paper, it is used only for sequential contexts (\cref{sec:sequential-contexts}) and notions building on these.
However, it is present in all examples of interest, so we include it in the general definition.

We will also use scopes to describe \emph{binders}:
if some primitive symbol $S$ binds $\gamma$ variables in its $i$-th argument, then for an instance of $S$ in scope~$\delta$, the $i$-th argument will be an expression in scope $\sumscope{\delta}{\gamma}$.
Most traditional constructors bind finitely many variables; to facilitate this, we let $\numscope{n}$ denote the sum of $n \in \N$ copies of $\singletonscope$, which also provides alternative notations $\numscope{0}$ and $\numscope{1}$ for the empty and singleton scopes, respectively.

\begin{example} \label{ex:de-bruijn-scope-systems}
  \defemph{De Bruijn indices} and \defemph{de Bruijn levels} can be seen as scope systems, with $\N$ as the set of scopes, $\position{n} \defeq \{0, 1, \ldots, n-1\}$, $\emptyscope \defeq 0$, and $\sumscope{m}{n} \defeq m + n$.

  The difference lies just in the choice of coproduct inclusions $\inlscope{} : \position{m} \to \position{m + n} \leftarrow \position{n} : \inrscope{}$.
  Setting $\inlscope(i) \defeq i + n$, $\inrscope(j) = j$ gives de Bruijn \emph{indices}, as going under a binder increments the variables outside it.
  Setting $\inlscope(i) \defeq i$, $\inrscope(j) = j + n$ gives de Bruijn \emph{levels}, as higher positions go to the innermost-bound variables.
  Over these scope systems, our syntax precisely recovers standard de Bruijn-style syntax, as in \citep{deBruijn:Lambda:1972} and subsequent work.
\end{example}

Both the de Bruijn scope systems are \emph{strict} in the sense that we have equalities $\sumscope{\gamma}{\emptyscope} = \gamma = \sumscope{\emptyscope}{\gamma}$ and $\sumscope{(\sumscope{\gamma}{\delta})}{\eta} = \sumscope{\gamma}{(\sumscope{\delta}{\eta})}$, whereas general scope systems provide only canonical isomorphisms.
The equalities help reduce bureaucracy in several proofs, so we shall occasionally indulge in assuming that we work with a strict scope system.
The doubtful reader may consult the formalisation, which makes no such assumptions.
%

\begin{example}
  The \defemph{finite sets} system takes scopes to be finite sets, along with any choice of coproducts.
  It may be prudent to restrict to a small collection, say the hereditarily finite sets.
  Syntax over these scope systems gives a concrete implementation of categorical approaches such as \citep{fiore-plotkin-turi}.
\end{example}

\begin{example}
  Scope systems are not intrinsically linked to dependent type theories, but provide a useful discipline for syntax of other systems.
  For instance, in geometric logic, the infinitary disjunction is usually given with a side condition that the free variables of the disjunction must remain finite \cite[D1.1.3(xi)]{johnstone:elephant-ii}.
  By using finite scopes, we can make the finiteness condition explicit from the start, and dispense with the side condition.
  This is similar in spirit to~\citep{fiore-plotkin-turi} and more closely mirrors the categorical semantics.
  By contrast, the classical Hilbert-type logics $\mathcal{L}_{\alpha,\beta}$ of \citep{karp:1964-book} allow genuinely infinite contexts and binders, which can be obtained by taking scopes to be ordinals $\delta \in \alpha$, with $\position{\delta} \defeq \delta$.
\end{example}

\begin{example}
  The traditional syntax using named variables, for both free and bound variables, is \emph{not} an example of a scope system.
  In that approach, a fresh variable is not introduced by summing with a scope, but rather by a multivalued map allowing extension by \emph{any} unused name.
\end{example}

Some implementations of syntax treat free and bound variables separately, for instance \emph{locally nameless syntax}~\citep{mckinna93:_pure_type_system_formal} uses concrete names for free variables but de Bruijn indices for bound variables.
Scope systems as defined here do not subsume such syntax, but could be generalised to do so.

Categorically, a scope system can be viewed precisely as a category~$\Scope$ with initial and terminal objects, binary coproducts, and a full and faithful functor into $(\Set, 0, +, 1)$ preserving this structure.
The categorical structure is induced from $\Set$: morphisms $\gamma \to \delta$ are taken as functions $\position{\gamma} \to \position{\delta}$ --- we call these \defemph{renamings}.
Two of these, $r : \gamma \to \delta$ and $r' : \gamma' \to \delta'$, give a sum map $\sumscope{r}{r'} : \sumscope{\gamma}{\gamma'} \to \sumscope{\delta}{\delta'}$ arising from the universal property of coproducts.

For the remainder of the paper, we fix a scope system. To make concrete examples readable, we shall write them using the de Bruijn scopes from \cref{ex:de-bruijn-scope-systems}. They can be easily adapted to any other scope system.

\subsection{Arities and signatures}
\label{sec:arities-signatures}

While the arity of a simple algebraic operation is just the number of its arguments, the situation is complicated in type theory by the presence of \emph{binders}. Each argument of a type-theoretic constructor may be a term or a type, and some of its variables may be bound by the constructor. We thus need a suitable notion of arity.

\begin{definition}
  \label{def:syntactic-class}\label{def:arity}%
  By \defemph{syntactic classes}, we mean the two formal symbols $\Ty$ and $\Tm$, representing \defemph{types} and \defemph{terms} respectively.
  An \defemph{arity}~$\alpha$ is a family of pairs $(c, \gamma)$ where $c$~is a syntactic class and $\gamma$ is a scope.
  We call the indices of~$\alpha$ \defemph{arguments} and write $\args \alpha$ for the index set of~$\alpha$.
  Thus, each argument $i \in \args \alpha$ has an associated syntactic class $\argclass{\alpha}{i}$ and a scope $\argbinder{\alpha}{i}$, which we call the \defemph{binder} associated with the argument~$i$, and we can write~$\alpha$ as $\alpha = \famtuple{(\argclass \alpha i, \argbinder \alpha i)}{i \in \args \alpha}$.
\end{definition}

\begin{example}
  In Martin-Löf type theory with the de Bruijn scope system, the constructor $\symPi$ has arity $[(\Ty,0),(\Ty,1)]$.
  That is, the arity has two type arguments, and binds one variable in the second argument.
  A simpler example is the successor symbol in arithmetic whose arity is $[(\Tm,0)]$, because it takes one term argument and binds nothing.
  Still simpler, the arity of a constant symbol is the empty family.
\end{example}

Note that arities express only the basic syntactic information and do not specify the types of term arguments and bound variables, which will be encoded later by typing rules.

\begin{definition}
  \label{def:signature}%
  A \defemph{signature} $\Sigma$ is a family of pairs of a syntactic class and an arity.
  We call the elements of its index set \defemph{symbols}. Thus each symbol $S \in \Sigma$ has an associated syntactic class $\class{S}$ and an arity $\arity{S}$.
  A \defemph{type symbol} is one whose associated syntactic class is~$\Ty$, and a \defemph{term symbol} is one whose associated syntactic class is~$\Tm$.
  The \defemph{arguments} $\args S$ of~$S$ are the arguments of its arity $\arity{S}$. Each argument $i \in \args S$ has an associated syntactic class $\argclass{S}{i}$ and binder $\argbinder{S}{i}$, as prescribed by the arity~$\arity{S}$.
\end{definition}

\begin{example} \label{ex:pi-types-arities}
  The following signature describes the usual constructors for dependent products:
 \begin{align*}
   \symPi & \mapsto (\Ty, [(\Ty,0),(\Ty,1)]), \\
   \symlambda & \mapsto (\Tm, [(\Ty,0), (\Ty,1), (\Tm,1)]), \\
   \symapp & \mapsto (\Tm, [(\Ty,0), (\Ty,1), (\Tm,0),(\Tm,0)]).
 \end{align*}
 Let us spell out the last line. The symbol $\symapp$ builds a term, because its syntactic class is~$\Tm$, from four arguments. The first and the second arguments are types, with one variable bound in the second argument, while the third and the fourth arguments are terms. We thus expect an application term to be written as $\symapp(A, B, s, t)$, with one variable getting bound in~$B$.
\end{example}

\begin{definition}
  \label{def:signature-map}%
  A \defemph{signature map} $F : \Sigma \to \Sigma'$ is a map of families between them, that is, a function from symbols of $\Sigma$ to symbols of $\Sigma'$, preserving the arities and syntactic classes.
  Signatures and maps between them form a category~$\Sig$.
\end{definition}

\subsection{Raw syntax}

Once a signature $\Sigma$ is given, we know how to build type and term expressions over it. We call this part of the setup ``raw'' syntax to emphasise its purely syntactic nature.

\begin{definition}
  \label{def:raw-syntax}%
  The \defemph{raw syntax} over $\Sigma$ consists of the collections of \defemph{raw type expressions} $\ExprTy{\Sigma}{\gamma}$ and \defemph{raw term expressions} $\ExprTm{\Sigma}{\gamma}$, which are generated inductively for any scope~$\gamma$ as follows:
  \begin{enumerate}
  \item for every position $i \in \position{\gamma}$, there is a \defemph{variable} expression $\synvar{i} \in \ExprTm{\Sigma}{\gamma}$;
  \item for every symbol $S \in \Sigma$ of syntactic class $\class{S}$, and a map
    \begin{equation*}
      \textstyle
      e \in \prod_{i \in \args S} \Expr{\argclass{S}{i}}{\Sigma}{\sumscope{\gamma}{\argbinder{S}{i}}},
    \end{equation*}
    there is an expression $S(e) \in \Expr{\class{S}}{\Sigma}{\gamma}$, the \defemph{application} of~$S$ to arguments~$e$.
  \end{enumerate}
\end{definition}

Let us walk through the definition.
The first clause states that the positions of~$\gamma$ play the role of available variable names, still without any typing information.
The second clause explains how to build an expression with a symbol~$S$: for each argument $i \in \args S$, an expression $e_i$ of a suitable syntactic class must be provided, where $e_i$ may refer to variable names given by~$\gamma$ as well as the variables that are bound by~$S$ in the $i$-th argument.
The expressions $e_i$ are conveniently collected into a function~$e$.
When writing down concrete examples we write the arguments as tuples.

\begin{example}%
  \label{ex:symbol-pi-arity}
  The symbol $\symPi$ has arity $[(\Ty,0),(\Ty,1)]$.
  So if $A$ is a type expression with free variables amongst $\sumscope{\gamma}{\emptyscope}$ (which is isomorphic to $\gamma$), and $B$ is a type expression with free variables in $\sumscope{\gamma}{1}$ (which has an extra free variable available), then $\synPi[A,B]$ is a type expression with free variables in $\gamma$.
\end{example}

\begin{definition}
  \label{renaming-action}%
  The \defemph{action of a renaming} $r : \gamma \to \delta$ on expressions is the map $\rename{r} : \Expr{c}{\Sigma}{\gamma} \to \Expr{c}{\Sigma}{\delta}$, defined by structural recursion:
  \begin{align*}
    \rename{r} (\synvar{i}) &\defeq \synvar{r(i)}, \\
    \rename{r} (S(e)) &\defeq
      S(\tuple{\rename{(\sumscope{r}{\idmap[\argbinder{S}{i}]})}(e_i)}{i \in \args{S}}).
  \end{align*}
\end{definition}

\noindent
Note how the definition uses the functorial action of~$\sumscope{}{}$ to extend the renaming~$r$ when it descends under the binders of a symbol.

\begin{definition}
  \label{def:signature-map-action}%
  The \defemph{action of a signature map} $F : \Sigma \to \Sigma'$ on expressions is the map $\act{F} : \Expr{c}{\Sigma}{\gamma} \to \Expr{c}{\Sigma'}{\gamma}$, defined by structural recursion:
  \begin{align*}
    \act{F} (\synvar{i}) & \defeq \synvar{i}, \\
    \act{F} (S(e)) &\defeq F(S)(\act{F} \circ e).
  \end{align*}
\end{definition}

\begin{propositionwithqed}
  \label{prop:commutation-renaming-signature-map}%
  The actions by renamings and signature maps commute with each other, and respect identities and composition.
  That is, they make expressions into a functor 
  $\Expr{}{}{} : \Sig \times \Scope \to \Set^{\set{\Ty, \Tm}}$. \qedhere
\end{propositionwithqed}

\subsection{Substitution}
\label{sec:raw-substitutions}

We next spell out substitution as an operation on raw expressions, and note its basic properties.

\begin{definition}
  \label{def:raw-substitution}%
  A \defemph{raw substitution} $f : \gamma \to \delta$ over a signature $\Sigma$ is a map $f : \position{\delta} \to \ExprTm{\Sigma}{\gamma}$.
The \defemph{extension $\sumscope{f}{\eta} : \sumscope{\gamma}{\eta} \to \sumscope{\delta}{\eta}$ by a scope~$\eta$} is the substitution
\begin{align*}
  (\sumscope{f}{\eta})(\inlscope(i)) &\defeq \act{{\inlscope}}(f(i)) & &\text{if $i \in \position{\delta}$}, \\
  (\sumscope{f}{\eta})(\inrscope(j)) &\defeq \synvar{\inrscope{(j)}} & &\text{if $j \in \position{\eta}$}.
\end{align*}
The (contravariant) \defemph{action of $f$ on an expression $e \in \Expr{c}{\Sigma}{\delta}$} gives the expression $\tca{f} e \in \Expr{c}{\Sigma}{\gamma}$, as follows:
\label{def:raw-substitution-action}
\begin{align*}
  \tca{f} (\synvar{i}) &\defeq f(i), \\
  \tca{f} (S (e)) &\defeq S(\fammap{\tca{(\sumscope{f}{\argbinder{S}{i}})} e_i}{i \in \arg S}).
\end{align*}
\end{definition}

\begin{example}
  \label{ex:de-bruijn-substitution}%
  The above definition, specialized to the de Bruijn scope systems of \cref{ex:de-bruijn-scope-systems}, precisely recovers the usual definition of substitution with de Bruijn indices or levels.
  Their explicit shift operators are abstracted away, in our setup, as renaming under coproduct inclusions.
\end{example}

\begin{definition}
  \label{def:variable-renaming}
  Any renaming $r : \gamma \to \delta$ induces a substitution $\bar{r} : \delta \to \gamma$, with $\bar{r}(i) \defeq \synvar{r(i)}$.
  In particular, each scope $\delta$ has an \defemph{identity substitution} $\idmap[\delta] : i \mapsto \synvar{i}$.
  Substitutions $f : \gamma \to \delta$ and $g : \delta \to \theta$ may be \defemph{composed} to give a substitution $g \circ f : \gamma \to \theta$, defined by $(g \circ f)(k) \defeq \tca{f}(g(k))$.
\end{definition}

We often write $r$ instead of $\bar{r}$, a slight notational abuse grounded in the next proposition.

\begin{proposition}
  For all suitable renamings $r$, substitutions $f$, $g$, and expressions $e$:
  \begin{enumerate}
  \item Substitution generalises renaming: $\tca{\bar{r}}e = \act{r}e$.
  \item Identity substitution is trivial: $\tca{\idmap[\delta]}e = e$.
  \item Substitution commutes with renaming:
    \begin{equation*}
      \act{r}(\tca{f}e) = \tca{(i \mapsto \act{r}f(i))} e
      \qquad\text{and}\qquad
      \tca{f}(\act{r}e) = \tca{(i \mapsto f(r(i)))} e.
    \end{equation*}

  \item Substitution respects composition: $\tca{f} (\tca{g} e) = \tca{(g \circ f)}$.
  \item Composition of substitutions is unital and associative:
    \begin{equation*}
      f = \idmap \circ f = f \circ \idmap
      \qquad\text{and}\qquad
      f \circ (g \circ h) = (f \circ g) \circ h.
    \end{equation*}
  \end{enumerate}
\end{proposition}

\begin{proof}
  All direct by structural induction on expressions, as in the standard proofs for de Bruijn syntax. 
\end{proof}

The interaction of substitutions with signature maps is similarly straightforward: signature maps act functorially on raw substitutions, and the constructions of this section are natural with respect to the action. More precisely:

\begin{propositionwithqed}
  Given a signature map $F : \Sigma \to \Sigma'$, and a raw substitution $f : \delta \to \gamma$ over $\Sigma$, there is a raw substitution $\act{F} f : \delta \to \gamma$ over $\Sigma'$ given by $(\act{F} f)(i) = \act{F}(f(i))$, and the action respects composition and identities in $F$.
  Moreover, given such $F$, for all suitable $f$ and $e$, we have $\act{F}(\tca{f} e) = \tca{(\act{F}f)} (\act{F}e)$; similarly, for all suitable $f$, $g$, we have $\act{F}(f \circ g) = \act{F}f \circ \act{F}g$. 
\end{propositionwithqed}

\subsection{Metavariable extensions and instantiations}
\label{sec:metas-and-instantiations}

As mentioned in the introduction, when we write down the \emph{rules} of type theories, we will need to extend the ambient signature~$\Sigma$ by new symbols to represent the metavariables of the rule.

For instance, consider the constructor $\symPi$, with arity $[(\Ty,0),(\Ty,1)]$.
In the rule for~$\symPi$ formation (\cref{ex:symbol-pi-arity}), we shall use two new symbols $\symA$, $\symB$, corresponding to the arguments of~$\symPi$ in the premises of the rule.
The classes and arities of these new symbols are given by their classes and binders as arguments of~$\symPi$: they are both type symbols; the first one takes no arguments, and the second one takes one term argument.

\begin{definition}
  \label{def:simple-arity}%
  The \defemph{simple arity~$\simplearity \gamma$} of a scope~$\gamma$ is the arity indexed by the positions~$\position{\gamma}$, and whose arguments all have syntactic class $\Tm$ with no binding, i.e., $\simplearity \gamma \defeq \famtuple{(\Tm, \emptyscope)}{i \in \position{\gamma}}$.
\end{definition}

\begin{definition}
  \label{def:metavariable-extensions}%
  The \defemph{metavariable extension $\mvextend{\Sigma}{\alpha}$ of~$\Sigma$ by arity $\alpha$} is the signature indexed by $\famindex{\Sigma} + \args{\alpha}$, defined by
  \begin{equation*}
    (\mvextend{\Sigma}{\alpha})_{\inl(S)} \defeq \Sigma_S
    \qquad\text{and}\qquad
    (\mvextend{\Sigma}{\alpha})_{\inr(i)} \defeq (\argclass{\alpha}{i}, \simplearity (\argbinder{\alpha}{i})).
  \end{equation*}
  We usually treat the injection of symbols $\Sigma \to \mvextend{\Sigma}{\alpha}$ as an inclusion, writing $S$ instead of $\inl(S)$;
  and for each $i \in \args{\alpha}$, we write $\synmeta{i}$ for the corresponding new symbol $\inr(i)$ of $\mvextend{\Sigma}{\alpha}$, the \defemph{metavariable symbol} for $i$.
\end{definition}

\begin{example}
  The symbol $\symlambda$ has arity
  $
    [(\Ty,0), (\Ty,1), (\Tm,1)]
  $.
  It thus gives rise to the metavariable extension $\Sigma + [(\Ty,0), (\Ty,1), (\Tm,1)]$, which adjoins to~$\Sigma$ three metavariable symbols $\synmeta{0}$, $\synmeta{1}$, $\synmeta{2}$, which for readability we may rename to $\symA$, $\symB$, $\symb{t}$, with classes and arities $(\Ty, [])$, $(\Ty, [(\Tm, 0)])$, and $(\Tm, [(\Tm,0)])$ respectively.
  That is, $\symA$ and $\symB$ are type symbols and $\symb{t}$ a term symbol, with the latter two each taking a term argument.
  These will then appear in the rule for $\symlambda$, to denote the arguments of a generic instance of $\symlambda$.
  We will often use more readable names for metavariable symbols, as here, without further comment.
\end{example}

Let $\gamma$ be a scope and $\synmeta{i}(e)$ a raw expression over signature $\mvextend{\Sigma}{\alpha}$ and~$\gamma$.
Then~$e$ is a map which assigns to each $j \in \position{\argbinder{\alpha}{i}}$ an expression in $\ExprTm{\mvextend{\Sigma}{\alpha}}{\gamma}$ because $\sumscope{\gamma}{\emptyscope} = \gamma$.
Thus we may construe $e$ as a raw substitution $e : \argbinder{\alpha}{i} \to \gamma$.
With this in mind, the following definition explains how metavariables are replaced with expressions.

\begin{definition}
  \label{def:instantiation}%
  Given a signature~$\Sigma$, an arity $\alpha$, and a scope~$\gamma$, an \defemph{instantiation~$I \in \Inst \Sigma \gamma \alpha$} of $\alpha$ in scope $\gamma$ is a family of expressions $I_i \in \Expr{\argclass{i}{\alpha}}{\Sigma}{\sumscope{\gamma}{\argbinder{i}{\alpha}}}$, for each $i \in \args{\alpha}$.
  Such an instantiation acts on expressions $e \in \Expr{c}{\mvextend{\Sigma}{\alpha}}{\delta}$ to give expressions $\act{I} e \in \Expr{c}{\Sigma}{\sumscope{\gamma}{\delta}}$, by
replacing each occurrence of $\synmeta{i}$ in $e$ with a copy of~$I_i$, with the arguments of $\synmeta{i}$ recursively substituted for the corresponding variables in $I_i$:
\begin{align*}
  \act{I} (\synmeta{i}(e)) &\defeq \tca{(\sumscope{\gamma}{e})} I_i, \\
  \act{I} (\synvar{j})     &\defeq \synvar{\inr(j)}, \\
  \act{I} (S(e))           &\defeq S \, (\fammap{\act{I} e_j}{j \in \args{S}}).
\end{align*}
  We call $\act{I}e$ the \defemph{instantiation of $e$ with $I$}.
\end{definition}

\begin{example} \label{ex:app-instantiation}
  Anticipating \cref{ex:raw-rule-app}, the rule for function application will be written as follows, where for readability $x$ stands for $\synvar{0}$:
  \begin{equation*}
  \inferrule
    { \istype{ }{\symA} \\
      \istype{x \of \symA}{\symB(x)} \\
      \isterm{}{\symb{s}}{\symPi(\symA, \symB(x)}) \\
      \isterm{}{\symb{t}}{\symA}
    }{
      \isterm{}{\symapp(\symA, \symB(x), \symb{s}, \symb{t})}{\symB(\symb{t})}
    }
  \end{equation*}
  All expressions are in the metavariable extension $\mvextend{\Sigma}{\arity{\symapp}}$, where $\Sigma$ is some ambient signature including $\symPi$ and $\symapp$, and $\arity{\symapp} = [(\Ty,0), (\Ty,1), (\Tm,0),(\Tm,0)]$ as in \cref{ex:pi-types-arities}.
  The symbols $\symA$, $\symB$, $\symb{s}$, $\symb{t}$ are the metavariable symbols of this extension.
  An instantiation of $\arity{\symapp}$ in scope $\gamma$ consists of expressions $A \in \Expr{\Ty}{\Sigma}{\gamma}$, $B \in \Expr{\Ty}{\Sigma}{\sumscope{\gamma}{\singletonscope}}$, and $s, t \in \Expr{\Tm}{\Sigma}{\gamma}$.
  Instantiating the conclusion with these expressions, over some context $\Gamma$, gives the judgement $\isterm{\Gamma}{\symapp(A, B, s, t)}{B[t/x]}$.
\end{example}

Building on the above, instantiations also act on other objects built out of expressions, including substitutions and other instantiations;
at the same time, being themselves syntactic objects, instantiations are acted upon by substitutions and signature maps;
and all of these are suitably natural and functorial, as follows.

\begin{definition}
  \label{def:instantiation-actions}%
  Given an instantiation $I \in \Inst{\Sigma}{\gamma}{\alpha}$:
  \begin{enumerate}

  \item The \defemph{instantiation $I$ acts on a substitution} $f : \delta' \to \delta$ over $\mvextend{\Sigma}{\alpha}$ to give a substitution $\act{I} f : \sumscope{\gamma}{\delta'} \to \sumscope{\gamma}{\delta}$, defined by
    \begin{align*}
      (\act{I}f)(\inlscope i) & \defeq \inlscope i & (i \in \gamma), \\
      (\act{I}f)(\inrscope j) & \defeq \act{I}(f j) & (j \in \delta).
    \end{align*}

  \item The \defemph{instantiation $I$ acts on an instantiation} $J \in \Inst{\mvextend{\Sigma}{\alpha}}{\delta}{\beta}$ to give an instantiation $\act{I}J \in \Inst{\Sigma}{\sumscope{\gamma}{\delta}}{\beta}$, defined by $(\act{I}J)_j \defeq \act{I}(J_j)$.  

  \item A \defemph{substitution} $f : \delta \to \gamma$ over $\Sigma$ \defemph{acts on the instantiation~$I$} to give an instantiation $\tca{f} I \in \Inst{\Sigma}{\delta}{\alpha}$, defined by $(\tca{f}I)_i \defeq \tca{(\sumscope{f}{\argbinder{\alpha}{i}})}I_i$.

  \item The \defemph{instantiation $I$ is translated along a signature map $F : \Sigma \to \Sigma'$} to give an instantiation $\act{F} I \in \Inst{\Sigma'}{\gamma}{\alpha}$, defined by $(\act{F}I)_i \defeq \act{F}(I_i)$.
  \end{enumerate}
\end{definition}

\begin{propositionwithqed}
  \label{prop:instantiation-boilerplate}%
  For all suitable instantiations $I$, $J$, signatures maps $F$, $G$, substitutions $f$, $g$, arities $\alpha$, expressions $e$, and scopes $\gamma$, $\delta$, $\theta$:
  \begin{enumerate}

  \item
    Translation along signature maps is functorial:
    %
    \begin{equation*}
       \act{(G \circ F)} I = \act{G} (\act{F} I)
      \qquad\text{and}\qquad
      \act{{\idmap[\Sigma]}} I = I.
    \end{equation*}

  \item
    The actions of substitutions and instantiations on expressions and on each other are natural with respect to translation along signature maps:
    %
    \begin{align*}
      \act{F}(\act{I} e) & = \act{(\act{F} I)} (\act{(\mvextend{F}{\alpha})} e), \\
      \act{F}(\act{I} f) & = \act{(\act{F} I)} (\act{(\mvextend{F}{\alpha})} f), \\
      \act{F}(\act{I} J) & = \act{(\act{F} I)} (\act{(\mvextend{F}{\alpha})} J), \\
      \act{F}(\tca{f} I) & = \tca{(\act{F} f)} (\act{F} I).
    \end{align*}

  \item
    The action of substitutions on instantiations is functorial:
    \begin{equation*}
      \tca{(f \circ g)} I = \tca{g} (\tca{f} I)
      \qquad\text{and}\qquad
      \tca{{\idmap[\gamma]}} I = I.
    \end{equation*}

  \item
    The action of instantiations is natural with respect to substitutions:
    \begin{align*}
      \act{(\tca{f} I)} e & = \tca{(\sumscope{f}{\delta})} (\act{I} e), \\
      \act{I} (\tca{g} e) &= \tca{(\act{I} g)} (\act{I} e).
    \end{align*}

  \item
    The action of instantiations is “associative” in the sense that
    \begin{equation*}
      \act{(\act{I} J)} e = \act{I}(\act{J} (\act{(\mvextend{\inl}{\beta})} e))
    \end{equation*}
    holds modulo the canonical associativity isomorphism between the scopes $\sumscope{(\sumscope{\gamma}{\delta})}{\theta}$ and $\sumscope{\gamma}{(\sumscope{\delta}{\theta})}$ of the left- and right-hand sides. \qedhere
  \end{enumerate}
\end{propositionwithqed}

The above properties, while routine to prove, are a lot of boilerplate.
They can be incorporated, if desired, into the statement that syntax forms a \emph{scope-graded monad} on signatures in the sense of~\citep{orchard-wadler-eades}, and that instantiations are certain Kleisli maps for this monad.
As ever, however, we emphasise in this paper the elementary viewpoint, rather than categorical abstraction.


\section{Raw type theories}
\label{sec:typing}

Having described raw syntax, we proceed with the formulation of raw type theories. These hold enough information to prevent syntactic irregularities, and can be used to specify derivations and derivability, but are qualified as ``raw'' because they allow arbitrariness and abnormalities that are generally considered undesirable.

\subsection{Raw contexts}

\begin{definition}
  \label{def:raw-context}%
  A \defemph{raw context~$\Gamma$} is a scope~$\gamma$ together with a family on $\ExprTy{\Sigma}{\gamma}$ indexed by $\position \gamma$, i.e., a map $\position{\gamma} \to \ExprTy{\Sigma}{\gamma}$.
  The positions of $\gamma$ are also called the \defemph{variables} of~$\Gamma$.
  We often write just~$\Gamma$ for~$\gamma$, e.g., $i \in \position{\Gamma}$ instead of $i \in \position{\gamma}$, and $\Expr{c}{\Sigma}{\Gamma}$ instead of $\Expr{c}{\Sigma}{\gamma}$.
  We use a subscript for the application of a context $\Gamma$ to a variable $i$, such that $\Gamma_i$ is the type expression at index $i$.
\end{definition}

This definition is somewhat non-traditional for dependent type theories, in a couple of ways.
Contexts are more usually defined as lists, so their variables are ordered, and the type of each variable is assumed to depend only on earlier variables, i.e.~$A_i \in \ExprTy{\Sigma}{i}$.
In our definition, the variables form an arbitrary scope, with no order assumed; and each type may \emph{a priori} depend on any variables of the context.

One may describe the two approaches as \emph{sequential} and \emph{flat} contexts, respectively.
The flat notion contains all the information needed when contexts are \emph{used} in derivations; we view the sequentiality as extra information that may be provided later by a derivation of well-formedness of a context, cf.\ \cref{sec:sequential-contexts}, but that is not needed at the raw level.

\begin{example}%
  \label{ex:de-bruijn-context}
  With the de Bruijn index scope system, a raw context $\Gamma$ of scope $n$ may be written as
  $\famtuple{A_i \in \ExprTy \Sigma n}{i \in \position n}$. A more familiar way to display $\Gamma$ is as list $[(n-1) \mto A_{n-1}, \ldots, 0 \mto A_0]$, where each $A_i \in \ExprTy{\Sigma}{n}$, but as raw contexts follow the flat approach, we should not think of this as imposing an order on the variables, and hence the list $[0 \mto A_0, \ldots, (n-1) \mto A_{n-1}]$ denotes the same context $\Gamma$.
  Note, by the way, that at this stage contexts on de Bruijn indices are indistinguishable from contexts on de Bruijn levels.
  The difference becomes apparent once we consider context extension, and the scope coproduct inclusions come into play.
\end{example}

\begin{definition}%
  \label{def:context-extension}%
  Let $\Gamma$ be a raw context, $\delta$ be a scope, and $\Delta : \position{\delta} \to \ExprTy{\Sigma}{\sumscope {\position{\Gamma}} {\delta}}$ a family of expressions indexed by $\position \delta$.
  The \defemph{context extension}~$\ctxextend \Gamma \Delta$ is the raw context of scope~$\sumscope {\position \Gamma} \delta$, defined as
  \begin{equation*}
    (\ctxextend \Gamma \Delta)_{\inlscope j} \defeq (\act \inlscope \circ\, \Gamma)_j
    \qquad\text{and}\qquad
    (\ctxextend \Gamma \Delta)_{\inrscope k} \defeq \Delta_k.
  \end{equation*}
  In other words, the extended raw context~$\ctxextend \Gamma \Delta$ is the map $[\act \inlscope \circ\, \Gamma, \Delta]$ induced by the universal property of the coproduct~$\position{\sumscope {\position{\Gamma}} \delta}$.
\end{definition}

\begin{example}
To continue \cref{ex:de-bruijn-context}, we can consider how context extension works for de Bruijn indices. Let $\Gamma = [2 \mto A_2, 1 \mto A_1, 0 \mto A_0]$, where each $A_i \in \ExprTy{\Sigma}{3}$, and $\Delta = [1 \mto B_1, 0 \mto B_0]$ with $B_j \in \ExprTy{\Sigma}{3 + 2}$. The coproduct inclusions are $\inlscope i = i + 2$ and $\inrscope j = j$. The context $\ctxextend \Gamma \Delta$ has scope $3 + 2 = 5$, and is given by $[\act {(i \mapsto i + 2)} \circ\, \Gamma, \Delta]$, which computes to $[4 \mto \act \inlscope A_2, 3 \mto \act \inlscope A_1, 2 \mto \act \inlscope A_0, 1 \mto B_1, 0 \mto B_0]$. The effect is that the variables from $\Gamma$ are renamed according to the scope of $\Delta$, and the renaming acts on the associated type expressions accordingly, i.e., by shifting all variables by 2.
\end{example}

Note that with raw contexts, we weaken types when extending a context.
In approaches using sequential contexts with scoped syntax, weakening is instead performed when types are taken \emph{out} of a context: that is, the variable rule (precisely stated) concludes $\isterm{\Gamma}{\synvar{i}}{\rename{\iota}\Gamma_i}$, where $\iota : i \to n$ is the inclusion of an initial segment into the full context.

In concrete examples, we will write contexts in a more traditional style, as lists of variables names with their associated types:
\[  x_1 \of A_1, \ldots, x_n \of A_n. \]

Like other syntactic objects, raw contexts are acted upon by signature maps and instantiations. The functoriality and commutation properties for these actions follow directly from the corresponding properties for expressions.

\begin{definition}
  Given a signature map $F : \Sigma \to \Sigma'$ and a raw context $\Gamma$ over $\Sigma$, the \defemph{translation of $\Gamma$ by $F$} is the raw context $\act{F} \Gamma$ over $\Sigma'$, with $\position{\act{F} \Gamma} \defeq \position{\Gamma}$ and $(\act{F}\Gamma)_i \defeq \act{F}\Gamma_i$. 
\end{definition}

\begin{propositionwithqed}
  The action of signature maps on raw contexts is functorial:
  \begin{equation*}
    \act{F}(\act{G} \Gamma) = \act{(F \circ G)}\Gamma
    \qquad\text{and}\qquad
    \act{{\idmap[\Sigma]}} \Gamma = \Gamma,
  \end{equation*}
  for all suitable $F$, $G$, $\Gamma$.
\end{propositionwithqed}

The action of instantiations is a little more subtle.
Acting pointwise on the expressions of the context  is not enough, since the instantiated expressions lie in a larger scope.
We need to supply extra types for the extra scope, i.e., the scope of the instantiation must itself underlie a raw context.

\begin{definition}
  An \defemph{instantiation $I$ in context $\Gamma$} over $\Sigma$ for arity $\alpha$ is an instantiation $I \in \Inst{\Sigma}{\position \Gamma}{\alpha}$.
  Then for any raw context $\Delta$ over $\mvextend{\Sigma}{\alpha}$, the \defemph{context instantiation $\act{(I,\Gamma)}\Delta$} is the context over $\Sigma$ with scope $\sumscope{\position{\Gamma}}{\position{\Delta}}$ and with type expressions
  \begin{align*}
    \act{(I,\Gamma)}\Delta_{\inlscope(i)} &\defeq \tca{\inlscope}\Gamma_i & (i \in \gamma), \\
    \act{(I,\Gamma)}\Delta_{\inrscope(j)}&\defeq \act{I}\Delta_j & (j \in \delta).
  \end{align*}

  Briefly, $\act{(I,\Gamma)}\Delta$ is the context extension of $\Gamma$ by the instantiations of the types of $\Delta$.
\end{definition}

\begin{propositionwithqed}
  Given instantiations 
  \begin{align*}
    &\text{$I \in \Inst{\Sigma}{\position \Gamma}{\alpha}$ in context $\Gamma$,}\\
    &\text{$K \in \Inst{\mvextend{\Sigma}{\alpha}}{\position \Delta}{\beta}$ in context $\Delta$, and}\\
    &\text{$\Theta \in \Inst{\mvextend{(\mvextend{\Sigma}{\alpha})}{\beta}}{\position \Theta}{\gamma}$ in context $\Theta$},
  \end{align*}
  the equation
  \begin{equation*}
    \act{(I,\Gamma)} (\act{(K,\Delta)} \Theta) =
    \act{(\act{I}K, \act{(I,\Gamma)} \Delta)} \Theta
  \end{equation*}
  holds modulo the canonical isomorphism between the scopes
  $\sumscope{\position \Gamma}{(\sumscope{\position \Delta}{\position \Theta})}$ and $\sumscope{(\sumscope{\position \Gamma}{\position \Delta})}{\position \Theta}$ of their right- and left-hand sides.
\end{propositionwithqed}

\subsection{Judgements}

Our type theories have four primitive judgement forms, following Martin-Löf \citep{martin-lof:bibliopolis}:
\begin{center}
\begin{tabular}{ll}
  $A \type$           &\qquad ``$A$ is a type'' \\
  $t : A$             &\qquad ``term $t$ has type $A$'' \\
  $A \judgeq B$       &\qquad ``$A$ and $B$ are equal as types''  \\
  $s \judgeq t : A$   &\qquad ``$s$ and $t$ are equal as terms of type $A$''
\end{tabular}
\end{center}
These are represented with symbols $\Ty$, $\Tm$, $\TyEq$, and $\TmEq$, respectively.
For our needs we need to describe the judgement forms quite precisely. In fact, the following elaboration may seem a bit \emph{too} precise, but we found it quite useful in the formalisation to make explicit all the concepts involved and distinctions between them.

Each judgement form has a family of \defemph{boundary slots} and possibly a \defemph{head slot}, where each slot has an associated syntactic class, as follows:
\begin{center}
\begin{tabular}{lll}
  Form & Boundary & Head \\
  \midrule
  $\Ty$ & $[]$ & $\Ty$ \\
  $\Tm$ & $[\Ty]$ & $\Tm$ \\
  $\TyEq$ & $[\Ty, \Ty]$ & \\
  $\TmEq$ & $[\Tm, \Tm, \Ty]$ & \\
\end{tabular}
\end{center}
The table encodes the familiar constituent parts of the judgement forms:
\begin{enumerate}
\item ``$A \type$'' has no boundary slots; the head slot, indicated by $A$, is a type.
\item ``$t : A$'' has one boundary type slot indicated by $A$, called the \defemph{underlying type}; the head, indicated by $t$, is a term.
\item ``$A \equiv B$'' has two type slots indicated by $A$ and $B$, called the \defemph{left-hand side} and \defemph{right-hand side}; there is no head.
\item ``$s \equiv t : A$'' has two term slots indicated by $s$ and $t$, and a type slot indicated by~$A$, called the \defemph{left-hand side}, the \defemph{right-hand side} and the \defemph{underlying type}, respectively; there is no head.
\end{enumerate}
The \defemph{slots} of a judgement form are the slots of its boundary, and the head, if present.

\begin{definition}
  \label{def:judgement}
  Given a raw context $\Gamma$ and a judgement form~$\phi$, a \defemph{hypothetical judgement} of that form over~$\Gamma$ is a map $J$ taking each slot of $\phi$ of syntactic class $c$ to an element of $\Expr{c}{\Sigma}{\Gamma}$.
  We write $\Judg{\Sigma}$ for the set of all hypothetical judgements over~$\Sigma$.
  The types of $\Gamma$ are the \defemph{hypotheses} and $J$ is the \defemph{thesis} of the judgement.
  When there is no ambiguity, we will (following traditional usage) speak of a \defemph{judgment} to mean either a whole hypothetical judgement $\Gamma \typesjudgement J$, or just a thesis $J$.
 
  A hypothetical judgement is an \defemph{object judgement} if it is a term or a type judgement, and an \defemph{equality judgement} if it is a type or a term equality judgement.
\end{definition}

\begin{example}
  The boundary and the head slots of the judgement form $\Tm$ are $[\Ty]$ and $\Tm$, respectively. Thus a hypothetical judgement of this form over a raw context~$\Gamma$ is a map taking the slot in the boundary to a type expression $A \in \ExprTy{\Sigma}{\Gamma}$ and the head slot to a term expression $t \in \ExprTm{\Sigma}{\Gamma}$.
  This corresponds precisely to the information conveyed by a traditional hypothetical term judgement ``$\isterm{\Gamma}{t}{A}$''.
\end{example}

In view of the preceding example we shall write a hypothetical judgement over $\Gamma$ given by a map $J$ in the traditional type-theoretic way
\begin{equation*}
  \Gamma \typesjudgement J
\end{equation*}
where the elements of $J$ are displayed in the corresponding slots.

Just like raw contexts, judgements are acted on by signature maps and instantiations.

\begin{definition} \label{def:judgements-functorial}
  Given a signature map $F : \Sigma \to \Sigma'$ and a judgement $\Gamma \typesjudgement J$ over $\Sigma$, the translation $\act{F}(\Gamma \typesjudgement J)$ is the judgement $\act{F} \Gamma \typesjudgement \act{F} J$ over $\Sigma'$ of the same form, where the thesis $\act{F} J$ is $J$ with $\act{F}$ applied pointwise to each expression.
\end{definition}

\begin{propositionwithqed} \label{prop:judgements-functorial}
  This action is functorial: $\act{F}(\act{G} (\Gamma \typesjudgement J)) = \act{(FG)}(\Gamma \typesjudgement J)$, and $\act{(\idmap[\Sigma])}(\Gamma \typesjudgement J) = (\Gamma \typesjudgement J)$, for all suitable $F$, $G$, $\Gamma$, $J$.
\end{propositionwithqed}

\begin{definition} \label{def:instantiate-judgement}
  Given a signature $\Sigma$, a hypothetical judgement $\Delta \typesjudgement J$ over a metavariable extension $\mvextend \Sigma \alpha$, a raw context $\Gamma$ over $\Sigma$, and an instantiation~$I$ of $\alpha$ in context $\Gamma$, the \defemph{judgement instantiation} $\act{(I,\Gamma)}{(\Delta \typesjudgement J)}$ is the judgement
  $\ctxextend{\Gamma}{\act{I} \Delta} \typesjudgement {\act I J}$ over $\Sigma$, where the thesis $\act I J$ is just $J$ with $\act{I}$ applied pointwise.
\end{definition}

\begin{propositionwithqed} \label{prop:instantiate-instantiate-judgement}
   Given instantiations
   \begin{align*}
     & \text{$I \in \Inst{\Sigma}{\position \Gamma}{\alpha}$ in context $\Gamma$ and}\\
     & \text{$K \in \Inst{\mvextend{\Sigma}{\alpha}}{\position \Delta}{\beta}$ in context $\Delta$},
   \end{align*}
   and a judgement $\Theta \typesjudgement J$ over $\mvextend{(\mvextend{\Sigma}{\alpha})}{\beta}$, the equation
   \begin{equation*}
     \act{(I,\Gamma)}(\act{(K,\Delta)}{(\Theta \types J)}) = \act{(\act{I}K,\act{(I,\Gamma)}\Delta)}(\Theta \types J)
   \end{equation*}
   holds modulo the canonical associativity renaming between their scopes.
\end{propositionwithqed}

\subsection{Boundaries}
\label{sec:boundaries}

In many places, one wants to consider data amounting to a hypothetical judgement without a head expression (if it is of object form, and so should have a head).
For instance, a \emph{goal} or \emph{obligation} in a proof assistant is specified by such data;
or when adjoining a new well-formed rule to a type theory, before picking a fresh symbol for it (if it is an object rule), the conclusion is specified by such data. 

These entities crop up frequently, and seem almost as fundamental as judgements, so deserve a name.

\begin{definition}
  Given a raw context $\Gamma$ and a judgement form~$\phi$, a \defemph{hypothetical boundary} of form $\phi$ over~$\Gamma$ is a map $B$ taking each \emph{boundary} slot of~$\phi$ of syntactic class $c$ to an element of  $\Expr{c}{\Sigma}{\Gamma}$.
\end{definition}

We display boundaries as judgements with a hole~$\bdryhead$ where the head should stand, or with $\qjudgeq$ in place of $\judgeq$:
\begin{equation*}
  \bdryhead \type
  \qquad\qquad
  \bdryhead : A
  \qquad\qquad
  A \qjudgeq B
  \qquad\qquad
  s \qjudgeq t : A
\end{equation*}
Since equality judgements have no heads, there is no difference in data between an equality judgement and an equality boundary, but there is one of sense: the judgement $A \judgeq B$ asserts an equality holds, whereas the boundary $A \qjudgeq B$ is a goal to be established or postulated.
Analogously,  $\bdryhead \type$ and $\bdryhead : A$ can be read as goals, the former that a type be constructed, and the latter that~$A$ be inhabited.

The terminology about judgements, as well as many constructions, carry over to boundaries.
In particular, the action of signature maps and instantiations on boundaries is defined just as in \cref{def:judgements-functorial,def:instantiate-judgement}, and enjoys analogous properties to \cref{prop:judgements-functorial,prop:instantiate-instantiate-judgement}.

Finally, and crucially, boundaries can be completed to judgements.
The data required depends on the form: completing an object boundary requires a head expression; completing an equality boundary, just a change of view.
\begin{definition} \label{def:completion-boundary}
  Let $B$ be a boundary in scope $\gamma$ over $\Sigma$.
  \begin{enumerate}
  \item If $B$ is of object form, then given an expression $e$ of the class of $B$ in scope $\gamma$, write $\plug{B}{e}$ for the \defemph{completion} of $B$ with head $e$, a judgement over $\gamma$.
  \item If $B$ is of equality form, then the completion of $B$ is just $B$ itself, viewed as a judgement.
  \end{enumerate}
\end{definition}

\begin{propositionwithqed}
  Completion of boundaries is natural with respect to signature maps: $\act{F}(\plug{B}{e}) = \plug{(\act{F}B)}{\act{F}e}$.
\end{propositionwithqed}

\subsection{Raw rules} \label{sec:raw-rules}

We now come to raw rules, syntactic entities that capture the notion of ``templates'' that are traditionally used to display the inference rules of a type theory.
The raw rules include all the information needed in order to be \emph{used}, for defining derivations and derivability of judgements --- but they do not yet include the extra properties we typically check when considering rules, and which guarantee good properties of the resulting derivability predicates.
We return to these later, in \cref{sec:acceptable-rules}.

\begin{definition}
  \label{def:raw-rule}%
  A \defemph{raw rule} $R$ over a signature $\Sigma$ consists of an arity $\arity{R}$,
  together with a family of judgements over the extended signature $\mvextend{\Sigma}{\arity{R}}$, the \defemph{premises} of~$R$,
  and one more judgement over $\mvextend{\Sigma}{\arity{R}}$, the \defemph{conclusion} of $R$.
  An \defemph{object rule} is one whose conclusion is an object judgement, otherwise it is an \defemph{equality rule}.
\end{definition}

\begin{example}%
\label{ex:raw-rule-app}
  Following on from \cref{ex:pi-types-arities}, the raw rule for function application has arity
  \begin{equation*}
    \arity{\symapp} = [(\Ty,0), (\Ty,1), (\Tm,0),(\Tm,0)].
  \end{equation*}
  Writing $\symA$, $\symB$, $\symb{s}$, $\symb{t}$ for the metavariable symbols of the extended signature $\mvextend{\Sigma}{\arity{\symapp}}$, the premises of the rule are the four-element family:
  \begin{equation*}
    [\;
    \istype{}{\symA}, \quad
    \istype{[0 \mto \symA]}{\symB(\synvar{0})}, \quad
    \isterm{}{\symb{s}}{\symPi(\symA, \symB(\synvar{0})}), \quad
    \isterm{}{\symb{t}}{\symA}
    \;] 
  \end{equation*}
  and its conclusion is the hypothetical judgement
  \begin{equation*}
    \isterm{}{\symapp(\symA, \symB(\synvar{0}), \symb{s}, \symb{t})}{\symB(\symb{t})}.
  \end{equation*}
  Of course, the traditional type-theoretic way of displaying such a rule is
  \begin{equation*}
  \inferrule
    { \istype{ }{\symA} \\
      \istype{x \of \symA}{\symB(x)} \\
      \isterm{}{\symb{s}}{\symPi(\symA, \symB(x)}) \\
      \isterm{}{\symb{t}}{\symA}
    }{
      \isterm{}{\symapp(\symA, \symB(x), \symb{s}, \symb{t})}{\symB(\symb{t})}
    }
  \end{equation*}
  It may seem surprising that we have $\symB(x)$ and $\symB(\symb{t})$ rather than, say, $B$ and $B[\symb{t}/x]$, since this style is usually apologised for as an abuse of notation.
  Here, it is precise and formal; $\symB$ is a metavariable \emph{symbol} in $\mvextend{\Sigma}{\arity{\symapp}}$, so it is applied to arguments when used in the syntax.
  Also note that the occurrences of~$x$ in the third premise and the conclusion are implicitly bound by~$\symPi$ and~$\symapp$, as can be discerned from their arities.
  When we instantiate the rule below with actual expressions $A$, $B$, $s$ and~$t$, then $\symB(x)$ and $\symB(\symb{t})$ will be translated into $B$ and $B[t/x]$ respectively.
\end{example}

\begin{definition}%
  \label{def:raw-rule-fmap}
  The functorial \defemph{action} of a signature map $F : \Sigma \to \Sigma'$ is the map $\act{F}$ which takes a rule $R$ over $\Sigma$ to the rule~$\act{F} R$ over~$\Sigma'$ whose arity is the arity~$\arity{R}$ of~$R$, and its premises and conclusion are those of $R$, all translated along the action of the induced map $\mvextend{F}{\arity{R}} : \mvextend{\Sigma}{\arity{R}} \to \mvextend{\Sigma'}{\arity{R}}$.
\end{definition}

A raw rule should not itself be thought of as a closure rule (though formally it is one), but rather as a template specifying a whole family of closure rules.

\begin{definition}
  \label{def:induced-closure-rule}\label{def:associated-closure-system}
  Given a rule $R$, a raw context $\Gamma$, and an instantiation $I$ of its arity~$\arity{R}$ over $\Gamma$, all over a signature~$\Sigma$, the \defemph{instantiation of $R$ under $I$, $\Gamma$}, is the closure rule $\act{(I,\Gamma)} R$ on $\Judg{\Sigma}$ whose premises and conclusion are the instantiations of the corresponding judgements of~$R$ under~$I$, $\Gamma$.
  The \defemph{closure system $\clos R$ associated to $R$} is the family
  $\famtuple{\act{(I,\Gamma)} R}{\Gamma \in \Context{\Sigma},\, I \in \Inst{\Sigma}{\Gamma}{\arity R}}$
  of all such instantiations.
\end{definition}

\begin{example}
  Continuing \cref{ex:raw-rule-app}, the raw rule $R_\symapp$ for application gives the closure system $\clos(R_\symapp)$, containing for each raw context~$\Gamma$ (with scope~$\gamma$) and expressions $A, B \in \ExprTy{\Sigma}{\gamma}$, $s, t \in \ExprTm{\Sigma}{\gamma}$ a closure rule
   \begin{align*}
    \inferrule
     { \istype{\Gamma}{A} \\
       \istype{\Gamma, x \of A}{B} \\
       \isterm{\Gamma}{s}{\symPi(A,B)} \\
       \isterm{\Gamma}{t}{A}
     }{
       \isterm{\Gamma}{\symapp(A, B, s, t)}{B[t/x]}.
     }
  \end{align*}
  
  This is visually similar to $R_\symapp$ itself, but not to be confused with it.
  The instantiation is a single closure rule, written over the ambient signature $\Sigma$; the original raw rule, written over $\mvextend{\Sigma}{\arity{\symapp}}$, is a template specifying the whole family of such closure rules.
  In the raw rule, $\symA$, $\symB$, $\symb{s}$, $\symb{t}$ are metavariable symbols from the extended signature $\mvextend{\Sigma}{\arity{\symapp}}$; in the instantiatiaion, the symbols $A$, $B$, $s$, $t$ (note the difference in fonts) are the actual syntactic expressions the raw rule was instantiated with.
\end{example}

This construction of $\clos$ formalises the usual informal explanation that a single written rule is a shorthand for a scheme of closure conditions, with the quantification of the scheme inferred from the written rule.

\begin{proposition} \label{prop:closure-system-of-raw-rule-under-signature-map}
  The construction $\clos$ is \emph{laxly natural} in signature maps,
  in that given $F : \Sigma \to \Sigma'$ and a rule $R$ over $\Sigma$, there is an induced simple map of closure systems $\clos R \to \clos \act{F}R$, over $\act{F} : \Judg \Sigma \to \Judg \Sigma'$.
\end{proposition}

\begin{proof}
  For each instantiation $I \in \Inst{\Sigma}{\position \Gamma}{\arity{R}}$, we have $\act{F}I \in \Inst{\Sigma'}{\position \Gamma}{\arity{R}}$ and $\act{(\act{F} I, \act{F} \Gamma)} R = \act{F} (\act{(I, \Gamma)} R)$.
\end{proof}

This is lax in the sense that the resulting map $\clos R \to \clos \act{F} R$ will not usually be an isomorphism: in general, not every instantiation of $\arity{R}$ over~$\Sigma'$ is of the form~$\act{F} I$.
This illustrates the need for considering raw rules formally, rather than just viewing a type theory as a collection of closure rules: when translating a type theory between signatures, we want not just the translations of the original closure rules, but all instantiations of the translated raw rules.

One might hope for $\clos$ to be similarly laxly natural under instantiations.
However, this is not so straightforwardly true; we will return to this in \cref{prop:instantiation-of-closure-system-of-raw-rule}, once the structural rules are introduced, and show a weaker form of naturality.

\subsection{Structural rules}

The rules used in derivations over a type theory will fall into two groups:
\begin{enumerate}
\item the \defemph{structural rules}, governing generalities common to all type theories;
\item the \defemph{specific rules} of the particular type theory.
\end{enumerate}

The structural rules over a signature~$\Sigma$ are a family of closure rules on $\Judg{\Sigma}$, which we now lay out.
They are divided into four families:
\begin{itemize}
\item the variable rules,
\item rules stating that equality is an equivalence relation,
\item rules for conversion of terms and term equations between equal types, and
\item rules for substitutions,
\end{itemize}
We have chosen the rules so that the development of the general setup requires no hard meta-theorems, as far as possible. In particular, we include the substitution rules into the formalism so that we can postpone proving elimination of substitution until \cref{sec:elimination-substitution}. You might have expected to see congruence rules among the structural rules, but those we take care of separately in \cref{sec:congruence-rules} because they depend on the specific rules.

The first three families of structural rules are straightforward.

\begin{definition}
  \label{def:variable-rule}%
  For each raw context $\Gamma$ over a signature $\Sigma$, and for each $i \in \position{\Gamma}$, the corresponding \emph{variable rule} is the closure rule
  \begin{equation*}
    \infer{
      \istype{\Gamma}{\Gamma_i}
    }{
      \isterm{\Gamma}{\synvar{i}}{\Gamma_i}
    }
  \end{equation*}
  Taken together, the variable rules form a family indexed by such pairs $(\Gamma, i)$.
\end{definition}

While this had to be given directly as a family of closure rules, the next two groups of structural rules can be expressed as raw rules.

\begin{definition}
  \label{def:equivalence-relation-rule}%
  The \defemph{raw equivalence relation rules} are the following raw rules:
\begin{mathpar}
\infer
  {
    \istype{}\symA
  }{
    \eqtype{}{\symA}{\symA}
  }
\and
\infer{
  \istype{}{\symA}
  \\
  \istype{}{\symB}
  \\
  \eqtype{}{\symA}{\symB}
}{
  \eqtype{}{\symB}{\symA}
}
\and
\infer{
  \istype{}{\symA}
  \\
  \istype{}{\symB}
  \\
  \istype{}{\symC}
  \\
  \eqtype{}{\symA}{\symB}
  \\
  \eqtype{}{\symB}{\symC}
}{
  \eqtype{}{\symA}{\symC}
}
\and
\infer
  {
   \istype{} \symA
   \\
   \isterm{}{\symb{s}} \symA
  }{
    \eqterm{}{\symb{s}}{\symb{s}} \symA
  }
\and
\infer{
  \istype{}{\symA}
  \\
  \isterm{}{\symb{s}} \symA
  \\
  \isterm{}{\symb{t}} \symA
  \\
  \eqterm{}{\symb{s}}{\symb{t}}\symA
}{
  \eqterm{}{\symb{t}}{\symb{s}}\symA
}
\and
\infer{
  \istype{}{\symA}
  \\
  \isterm{}{\symb{s}} \symA
  \\
  \isterm{}{\symb{t}} \symA
  \\
  \isterm{}{\symb{u}} \symA
  \\
  \eqterm{}{\symb{s}}{\symb{t}} \symA
  \\
  \eqterm{}{\symb{t}}{\symb{u}} \symA
}{
  \eqterm{}{\symb{s}}{\symb{u}} \symA
}
\end{mathpar}
The \defemph{equivalence relation rules over~$\Sigma$} is the sum of the closure systems associated to the above equivalence relation rules, over a given signature~$\Sigma$.
\end{definition}

We trust the reader to be able read off the arities of the metavariable symbols appearing in raw rules.
For instance, from the use of~$\symA$ and~$\symb{s}$ in the above term reflexivity rule
we can tell that the rule has arity $[(\Ty,0),(\Tm,0)]$.

The conversion rules are written as raw rules, as well.

\begin{definition}
  \label{def:conversion-rule}%
  The \defemph{raw conversion rules} are the following raw rules:
  \begin{mathpar}
  \infer{
    \istype{}{\symA}
    \\
    \istype{}{\symB}
    \\
    \isterm{}{\symb{s}}{\symA}
    \\
    \eqtype{} \symA \symB
  }{
    \isterm{}{\symb{s}}{\symB}
  }
  \and
  \infer{
    \istype{}{\symA}
    \\
    \istype{}{\symB}
    \\
    \isterm{}{\symb{s}}{\symA}
    \\
    \isterm{}{\symb{t}}{\symA}
    \\
    \eqterm{}{\symb{s}}{\symb{t}}{\symA}
    \\
    \eqtype{}{\symA} \symB
  }{
    \eqterm{}{\symb{s}}{\symb{t}}{\symB}
  }
\end{mathpar}
Again, the \defemph{conversion rules over $\Sigma$} is the sum of the closure systems associated to the above conversion rules, over a signature~$\Sigma$.
\end{definition}

The remaining groups are the substitution and equality-substitution rules.

The substitution rule should formalize the notion that ``well-typed'' substitutions preserve derivability of judgements.
Treatments taking simultaneous substitution as primitive usually say something like: a raw substitution $f : \Delta \to \Gamma$ is well-typed if $\isterm{\Delta}{f(i)}{\tca{f}\Gamma_i}$ for each $i \in \position{\Gamma}$.
However, taking all these judgements as premises in the substitution rule is rather profligate: most substitutions in practice act non-trivially only on a small part of the context.
For instance, a single-variable substitution may be represented as a raw substitution $\Gamma \to \ctxextend{\Gamma}{A}$ acting trivially on $\Gamma$, so no checking should be required there.
Indeed, in treatments taking single-variable substitution as primitive, only require checking of the substituted expression.
To abstract this situation, we define the substitution rule as follows. Recall that a subset $X \subseteq Y$ is \emph{complemented} when $X \cup (Y \setminus X) = Y$, a condition that is vacuously true in classical logic.

\begin{definition}
  \label{def:substitution-rule}%
  A \defemph{raw substitution} $f : \Delta \to \Gamma$ \defemph{acts trivially at $i \in \position{\Gamma}$} when there is some (necessarily unique) $j \in \position{\Delta}$ such that $f(i) = \synvar{j}$ and $\Delta_j = \tca{f} \Gamma_i$.
  Given a complemented subset $K \subseteq \position{\Gamma}$ on which~$f$ acts trivially, the corresponding \defemph{substitution rule} is the closure rule
  \begin{equation}
    \label{eq:substitution-rule}
    \infer{
      \Gamma \typesjudgement J
      \\\\
     \text{for each $i \in \position{\Gamma} \setminus K$:} \quad
     \isterm{\Delta}{f(i)}{\tca{f}\Gamma_i}
    }{
      \Delta \typesjudgement \tca{f}J
    }
  \end{equation}
  The substitution rules form a family of closure rules, indexed by $f : \Delta \to \Gamma$, $K$, and $\Gamma \typesjudgement J$.
\end{definition}

\noindent
In the above definition, $K$ is thought of as a set of positions at which~$f$ is
\emph{guaranteed} to act trivially, but it may also do so outside~$K$, as there is no harm in checking positions at which~$f$ acts trivially.

The substitution rules are formulated carefully for another, more technical reason.
In inductions over derivations (e.g.\ for \cref{lem:admissibility-substitution}), when a substitution descends under a binder, it gets extended to act trivially on the  variables introduced by the binder.
Within the inductive cases, we may not yet have enough information to conclude that the types of the bound variables are well-formed, but we can rely on the trivial action of the substitution.
Keeping substitution rules flexible and economical in this way therefore keeps these inductive proofs much cleaner.

Along similar lines, we have rules stating that substitution of equal terms gives equal results.
\begin{definition}
  \label{def:equality-substitution-rule}%
  Raw substitutions $f, g : \Delta \to \Gamma$ \defemph{act jointly trivially} at $i \in \position{\Gamma}$ when there is some (necessarily unique) $j \in \position{\Delta}$ such that $f(i) = g(i) = \synvar{j}$ and $\Delta_j = \tca{f} \Gamma_i = \tca{g} \Gamma_i$.
  Given a complemented subset $K \subseteq \position{\Gamma}$ on which $f$ and $g$ act jointly trivially,
  the corresponding \defemph{equality-substitution rules} are the closure rules
  \begin{mathpar}
    \infer{
      \istype{\Gamma}{A}
      \\\\
      \text{for each $i \in \position{\Gamma} \setminus K$:}\\\\
      \isterm \Delta {f(i)} {\tca f \Gamma_i}\\
      \isterm \Delta {g(i)} {\tca g \Gamma_i} \\
      \eqterm{\Delta}{f(i)}{g(i)}{\tca{f}\Gamma_i}
    }{
      \eqtype{\Delta}{\tca{f}A}{\tca{g}A}
    }
    \and
    \infer{
      \isterm{\Gamma}{t}{A}
      \\\\
      \text{for each $i \in \position{\Gamma} \setminus K$:}\\\\
      \isterm \Delta {f(i)} {\tca f \Gamma_i}\\
      \isterm \Delta {g(i)} {\tca g \Gamma_i} \\
      \eqterm{\Delta}{f(i)}{g(i)}{\tca{f}\Gamma_i}
    }{
      \eqterm{\Delta}{\tca{f} t}{\tca{g} t}{\tca{f} A}
    }
  \end{mathpar}
  The equality-substitution rules form a family of closure rules, indexed by $f : \Delta \to \Gamma$, $K$, and either $\istype{\Gamma}{A}$ or $\isterm{\Gamma}{t}{A}$.
\end{definition}

\begin{definition}
  \label{def:structural-rules}
  The \defemph{structural rules over $\Sigma$}, denoted $\StructuralRules \Sigma$, is the sum of the families of closure rules set out above: the variable, equivalence relation, conversion, substitution, and equality-substitution rules.
\end{definition}

\begin{proposition}
  \label{prop:structural-rules-under-signature-map}%
  Given a signature map $F : \Sigma \to \Sigma'$, there is a simple map of closure systems $\StructuralRules \Sigma \to \StructuralRules \Sigma'$ over $\act{F} : \Judg \Sigma \to \Judg \Sigma'$.
\end{proposition}

\begin{proof}
  This is straightforward, amounting to checking that for each instance of a structural rule over $\Sigma$, $F$ acts on the data to give an instance of the same structural rule over $\Sigma'$, and the resulting closure condition is the translation along $\act{F}$ of the original closure condition over $\Sigma$.
\end{proof}

Before giving a similar statement about instantiations of structural rules, we must first tie up the loose end from above about instantiation of closure systems of raw rules.

\begin{proposition} \label{prop:instantiation-of-closure-system-of-raw-rule}
  Let $R$ be a raw rule over $\Sigma$, $\mvextend{R}{\beta}$ its translation to an extension $\mvextend{\Sigma}{\beta}$, and $I$ an instantiation of~$\beta$ in some context~$\Gamma$.
  Then there is a closure system map $\clos (\mvextend{R}{\beta}) \to \clos R + \StructuralRules \Sigma$, over $\act{(I,\Gamma)} : \Judg{(\mvextend{\Sigma}{\beta})} \to \Judg \Sigma$.
\end{proposition}

\begin{proof}
  We need to show that for each instantiation $K \in \Inst{\mvextend{\Sigma}{\beta}}{\position \Delta}{\arity{R}}$ in some context~$\Delta$, the closure condition $\act{I} (\act{(K,\Gamma)} R)$ is derivable from $\clos R + \StructuralRules \Sigma$.

  Given such $K$ and $\Gamma$, we can instantiate both under $I$ to get an instantiation of~$R$ over $\Sigma$.
  We might hope that $\act{(\act{I}K,\act{I}\Gamma)}R = \act{I}\left(\act{(K,\Gamma)}(\mvextend{R}{\beta}\right)$;
  by \cref{prop:instantiate-instantiate-judgement}, we see that this does not strictly, but only up to an associativity renaming in each judgement.

  The substitution structural rule comes to our rescue here.
  For each judgement $J$ involved in $R$, with context $\Theta$, the associativity renamings give substitutions between $\act{I}(\act{K}\Theta)$ and $\act{(\act{I}K)}\Theta$ acting trivially at every position, so the substitution rule lets us derive $\act{I}(\act{K}J)$ from $\act{(\act{I}K)}J$ (with no further premises), and vice versa.

  The desired derivation of $\act{I}\left(\act{(K,\Gamma)}(\mvextend{R}{\beta}\right)$ from $\clos R + \StructuralRules \Sigma$ therefore consists of $\act{(\act{I}K,\act{I}\Gamma)}R$, together with an instance of the substitution rule after the conclusion and before each premise, implementing the associativity renamings.
\end{proof}

\begin{proposition}  \label{prop:instantiation-of-structural-rules}%
  Let $I \in \Inst{\Sigma}{\position \Gamma}{\alpha}$ be an instantiation in context $\Gamma$.
  Then there is a closure system map $\StructuralRules (\mvextend{\Sigma}{\alpha}) \to \StructuralRules \Sigma$, over $\act{(I,\Gamma)} : \Judg {(\mvextend{\Sigma}{\alpha})} \to \Judg{\Sigma}$.
\end{proposition}

\begin{proof}
  For the structural rules given as raw rules, the required derivations are given by \cref{prop:instantiation-of-closure-system-of-raw-rule}.
  
  For the other structural rules, we start as in \cref{prop:structural-rules-under-signature-map}.
  Given an instance of a structural rule over $\mvextend{\Sigma}{\alpha}$, we instantiate its data under $I$ to get an instance of the same structural rule over $\Sigma$.
  Wrapping this instance in associativity renamings, derived by the substitution rule as in \cref{prop:instantiation-of-closure-system-of-raw-rule}, gives the required derivation of the instantiation of the original instance.
\end{proof}

\subsection{Congruence rules}
\label{sec:congruence-rules}

Congruence rules, which state that judgemental equality commutes with type and term symbols, are peculiar enough to demand special attention.

They are present in almost all type theories, but rarely explicitly written out, and are often classified as structural rules.
We reserve that term for the rules of the preceding section, which are independent of the specific theory under consideration.
Congruence rules, by contrast, depend on the specific rules of a theory; for instance, the congruence rule for~$\symPi$ is determined by the formation rule for~$\symPi$.

In this section we define how any object rule determines an associated congruence rule.
We first set up an auxiliary definition, associating equality judgements to object judgements.

\begin{definition}
  \label{def:judgement-associated-congruence}
  For signature maps $\ell, r : \Sigma \to \Sigma'$ and an object judgement $J$ over~$\Sigma$, we define the equality judgment $\tca{(\ell,r)}J$ over~$\Sigma'$ by
  \begin{align*}
     \tca{(\ell,r)}(\istype{\Gamma}{A})
    \ &\defeq \ 
    (\eqtype{\act{\ell} \Gamma}{\act{\ell} A}{\act{r} A}),
    \\
    \tca{(\ell,r)}(\isterm{\Gamma}{t}{A})
    \  &\defeq \ 
    (\eqterm{\act{\ell} \Gamma}{\act{\ell} t}{\act r t}{\act{\ell} A}).
  \end{align*}
\end{definition}

\begin{definition}
  \label{def:congruence-rule}%
  Suppose $R$ is a raw object rule over a signature $\Sigma$, with premises $\famtuple{P_i}{i \in I}$ and conclusion $C$.
  Let $\phi_i$ be the judgement form of $P_i$, and take $I_{\ob} \defeq \set{i \in I \such \phi_i \in \set{\Ty, \Tm}}$, the set of object premises of~$R$.
  The \defemph{associated congruence rule} $\congrule{R}$ is a raw rule with arity
  $\arity{\congrule{R}} \defeq \arity{R} + \arity{R}$, defined as follows, where
  $\ell, r : \mvextend{\Sigma}{\arity{R}} \to \mvextend{\Sigma}{\arity{\congrule{R}}}$ are signature maps
  \begin{align*}
    \ell(\inl(S)) &\defeq \inl(S), &
    r(\inl(S)) &\defeq \inl(S), \\
    \ell(\inr(M)) &\defeq \inr(\inl(M)), &
    r(\inr(M)) &\defeq \inr(\inr(M)):
  \end{align*}
  \begin{enumerate}
  \item The premises of $\congrule{R}$ are indexed by the set $I + I + I_{\ob}$, and are given by:
    \begin{enumerate}
    \item the $\inl(i)$-th premise is $\act{\ell} P_i$,
    \item the $\inr(j)$-th premise is $\act{r} P_j$,
    \item the $\iota_2(k)$-th premise is the equality $\tca{(\ell,r)} P_k$, cf.\ \cref{def:judgement-associated-congruence}.
    \end{enumerate}

  \item The conclusion of $\congrule{R}$ is $\tca{(\ell,r)}C$.
  \end{enumerate}
\end{definition}

\begin{example}
  \label{ex:pi-congruence-rule}%
  \Cref{def:congruence-rule} works as expected. For example, the congruence rule
  associated with the usual product formation rule
  \begin{equation*}
    \inferrule{
      \istype{}{\symA} \\
      \istype{x \of \symA}{\symB(x)}
    }{
      \istype{}{\symPi(\symA, \symB(x))}
    }
  \end{equation*}
  comes out to be
  \begin{equation*}
    \inferrule{
      \istype{}{\symA'} \\
      \istype{x \of \symA'}{\symB'(x)} \\\\
      \istype{}{\symA''} \\
      \istype{x \of \symA''}{\symB''(x)} \\
      \\\\
      \eqtype{}{\symA'}{\symA''}
      \\
      \eqtype{x \of \symA'}{\symB'(x)}{\symB''(x)}
    }{
      \eqtype{}{\symPi(\symA', \symB''(x))}{\symPi(\symA'', \symB''(x))}
    }
  \end{equation*}
\end{example}

\subsection{Raw type theories}

After a considerable amount of preparation, we are finally in position to formulate what a rudimentary general type theory is.

\begin{definition}
  \label{def:raw-type-theory}%
  A \defemph{raw type theory}~$T$ over a signature $\Sigma$ is a family of raw rules over~$\Sigma$.
\end{definition}

\begin{definition} \label{def:closure-system-of-type-theory}
  The \defemph{associated closure system} of a raw type theory~$T$ over~$\Sigma$ is the closure system $\clos T \defeq \StructuralRules \Sigma + \coprod_{R \in T} \clos R$ on $\Judg{\Sigma}$; that is, it consists of the structural rules for~$\Sigma$, and the closure rules generated by the instantiations of the rules of~$T$.
  A \defemph{derivation in $T$} is a derivation over the closure system $\clos T$, in the sense of \cref{def:closure-system-derivation}.
\end{definition}

Note that we have not included the congruence rules into the closure system associated with a raw type theory. Instead, the presence of congruence rules will be required separately as a well-behavedness condition in \cref{sec:acceptable-type-theories}.
Derivability and admissibility of rules may now be defined as follows.

\begin{definition} \label{def:derivable-raw-rule}
  Let $T$ be a raw type theory over~$\Sigma$, and $R$ a raw rule over~$\Sigma$.
  \begin{enumerate}
  \item $R$ is \defemph{derivable} from~$T$ if its conclusion is derivable from its premises, over $\mvextend{T}{\arity{R}}$.
  \item $R$ is \defemph{admissible} for~$T$ if for every instance $\act{(I,\Gamma)}R$,
    its conclusion is derivable if its premises are derivable, all over~$T$.
  \end{enumerate}
\end{definition}

We record the basic category-theoretic structure of raw type theories.

\begin{definition}
  Given a signature map $F : \Sigma \to \Sigma'$, and raw type theories $T$, $T'$ over $\Sigma$ and $\Sigma'$ respectively, a \defemph{simple map $\bar{F} : T \to T'$ over $F$} is a family map $T \to T'$ over $\act{F} : \RawRule{\Sigma} \to \RawRule{\Sigma'}$.
  Such $\bar{F}$ is thus a map giving for each rule~$R$ of~$T$ a rule~$\bar{F}(R)$ of~$T'$, whose premises and conclusion are the translations along~$F$ of those of~$R$.
  There are evident identity simple maps, and composites over composites of signature maps, forming a category over the category of signatures.

  Furthermore, a signature map $F : \Sigma \to \Sigma'$ \defemph{acts on a raw type theory} $T$ over $\Sigma$, to give a raw type theory $\act{F}(T)$
  over $\Sigma'$, which consists of the translations $\act{F} R$ of the rules~$R$ of~$T$.
  As with family maps, maps $T \to T'$ over $F$ correspond precisely to maps $\act{F} T \to T'$ over $\idmap[\Sigma']$.
  In the case of the inclusion to a metavariable extension $\inl : \Sigma \to \mvextend{\Sigma}{\alpha}$, we write $\mvextend{T}{\alpha}$ for the translation $\act{{\inl}} T$ of $T$ to $\mvextend{\Sigma}{\alpha}$. 
\end{definition}

\begin{proposition}%
  \label{prop:cl-functorial-simple-maps} \label{prop:derivations-functorial-simple-maps}
  The construction $\clos{}$ is functorial in simple maps:
  a simple map of raw type theories $\bar{F} : T \to T'$ over $F : \Sigma \to \Sigma'$ induces a map $\act{\bar{F}} : \clos{T} \to \clos{T'}$ over $\act{F} : \Judg{\Sigma} \to \Judg{\Sigma'}$,
  and hence provides a translation of any derivation~$D \in \derivation{T}{H}{(\Gamma \typesjudgement J)}$ to a derivation $\act{\bar{F}} D \in \derivation{T'}{\act{F} H}{(\act{F} \Gamma \typesjudgement \act{F} J)}$.
\end{proposition}

\begin{proof}
  Direct from the functoriality and naturality properties of structural rules (\cref{prop:structural-rules-under-signature-map}) and of the closure systems associated to raw rules (\cref{prop:closure-system-of-raw-rule-under-signature-map}).
\end{proof}

\begin{corollary} \label{cor:derivations-functorial-signature-maps}
  A signature map $F : \Sigma \to \Sigma'$ acts on
  $D \in \derivation{T}{H}{(\Gamma \typesjudgement J)}$ to give a derivation
  $\act{F} D \in \derivation{\act{F} T}{\act{F} H}{\act{F}(\Gamma \typesjudgement J)}$, functorially so.
\end{corollary}

\begin{proof}
 By \cref{prop:derivations-functorial-simple-maps}, using the canonical simple map $T \to \act{F}T$ over~$F$.
\end{proof}

We use the previously corollary quite frequently to translate a derivation over a raw type theory to its extension. We mostly leave such applications implicit, as they are easily detected.

Instantiations also preserve derivability, but this is a significantly more involved construction --- more so than one might expect --- bringing together many earlier constructions and lemmas, and relying in particular on almost all the properties of \cref{prop:instantiation-boilerplate}.

\begin{proposition} \label{cor:instantiation-acts-on-flattening}
  Given a raw type theory $T$ over $\Sigma$, an instantiation $I \in \Inst{\Sigma}{\Gamma}{\alpha}$ induces a closure system map $\act{(I,\Gamma)} : \clos{\mvextend{T}{\alpha}} \to \clos{T}$ over $\act{(I,\Gamma)} : \Judg{\mvextend{\Sigma}{\alpha}} \to \Judg{\Sigma}$, where $\mvextend{T}{\alpha}$ is the translation of~$T$ by the inclusion $\Sigma \to \mvextend{\Sigma}{\alpha}$.
\end{proposition}

\begin{proof}
  Again, direct from similar properties of structural rules (\cref{prop:instantiation-of-structural-rules}) and closure systems of raw rules (\cref{prop:instantiation-of-closure-system-of-raw-rule}).
\end{proof}

\begin{corollary} \label{cor:instantiation-of-derivations}
  Let $T$ be a raw type theory over~$\Sigma$.
  An instantiation $I \in \Inst{\Sigma}{\Gamma}{\alpha}$ acts on a derivation $D \in \derivation{\mvextend{T}{\alpha}}{H}{(\Delta \typesjudgement J)}$ to give the \defemph{instantiation} $\act{(I, \Gamma)} D \in \derivation{T}{\act{(I, \Gamma)} H}{\act{I}(\Delta \typesjudgement J)}$.
\end{corollary}

Note that the hypotheses $H$ and the judgement $\Delta \typesjudgement J$ in the statement reside in the translation $\mvextend{T}{\alpha}$ by the inclusion $\Sigma \to \mvextend{\Sigma}{\alpha}$.

\subsection{Summary}

\emph{Raw type theories} give a conceptually minimal way to make precise what is meant by traditional specifications of type theories, and a similarly minimal amount of data from which to define \emph{derivability} on judgements.

As the name suggests, raw type theories are not a finished product.
Type theories in nature almost always satisfy further well-formedness properties, and are rejected by audiences if they do not.
In the next two sections, we will discuss these well-formedness properties.

In some ways, raw type theories may therefore be viewed as an unnatural or undesirable notion.
However, most of the well-behavedness properties --- or rather, the conditions on rules implying well-behavedness --- themselves involve checking derivability of certain judgements.
So raw type theories, as the minimal data for defining derivability, give a natural intermediate stage on the way to our main definition of “reasonable” type theories.


\section{Well-behavedness properties}
\label{sec:well-behavedness}

In this section we identify easily-checked syntactic properties of the rules specifying a type theory, and prove basic fitness-for-purpose meta-theorems, which together articulate the rules-of-thumb that researchers habitually use to verify that some collection of inference rules defines a “reasonable” type theory.

\subsection{Acceptable rules}
\label{sec:acceptable-rules}

Not all raw rules are deemed reasonable from a type-theoretic point of view.
But what standard of “reasonable” are we aiming to delineate?
Essentially, the same as for the axioms of a theory in first-order logic: the axioms must be well-formed enough to be given some meaning, although that meaning may be “false”, “wrong”, or otherwise unexpected.

Consider for instance the following possible modifications of the rule for $\symapp$, all written as raw rules:
\begin{gather}
\label{eq:example-app-1}
\inferrule{
  \istype {} {\symA} \\ 
  \istype {x \of \symA} {\symB(x)} \\
  \isterm {} {\symf} {\synPi[\symA,\symB(x)]} \\
  \isterm {} {\syma} {\symA}
}{
  \isterm {} {\symapp{(\symA,\symB(x),\symf,\syma)}} {\symB(\syma)}
}
\\[1ex]
\label{eq:example-app-2}
\inferrule{
  \istype {} {\symA} \\ 
  \istype {x \of \symA} {\symB(x)} \\
  \isterm {} {\symf} {\symA} \\
  \isterm {} {\syma} {\symA}
}{
  \isterm {} {\symapp{(\symA,\symB(x),\symf,\syma)}} {\symB(\syma)}
}
\\[1ex]
\label{eq:example-app-3}
\inferrule{
  \istype {} {\symA} \\ 
  \istype {x \of \symA} {\symB(x)} \\
  \isterm {} {\symf} {\synPi[\symA,\symB(x)]} \\
  \isterm {} {\syma} {\synPi[\symA,\symB(x)]}
}{
  \isterm {} {\symapp{(\symA,\symB(x),\symf,\syma)}} {\symB(\syma)}
}
\end{gather}

The first is the usual rule for $\symapp$, and should certainly be considered acceptable.

The second asks the argument $\symf$ to be of type $\symA$.
This is “wrong” under the usual reading of $\symPi$ and $\symapp$, but not entirely meaningless: one can introduce $\symapp$ with this typing rule, and obtain a well-behaved (if bizarre) type theory.
So this should be accepted as a type-theoretic rule.

The third asks the argument $\syma$ to be of type $\synPi[\symA,\symB(x)]$.
This is “not even wrong”: the conclusion purports to introduce a term of type $\symB(\syma)$, but that is not a well-formed type, since $\symB$ expects an argument of type $\symA$, so $\syma$ is not suitable (at least in the absence of other rules implying that $\eqtype{}{\symA}{\synPi[\symA,\symB(x)]}$).
This will therefore \emph{not} be an acceptable rule.

Another unacceptable rule would be:
\begin{gather}
  \label{eq:example-app-4}
  \inferrule{
    \istype {} {\symA} \\ 
    \istype {x \of \symA} {\symB(x)} \\
    \isterm {} {\symf} {\synPi[\symA,\symB(x)]} \\
    \isterm {} {\syma} {\symA}  \\
    \isterm {} {\syma} {\synPi[\symA,\symB(x)]} \\
  }{
    \isterm {} {\symapp{(\symA,\symB(x),\symf,\syma)}} {\symB(\syma)}
  }
\end{gather}
This is again clearly nonsense: it introduces $\syma$ twice, with two different types.

There are rules which are not uncommon in practice, but which we will not accept directly, such as:
\begin{gather}
\label{eq:example-app-5}
\inferrule{
   \isterm {} {\symf} {\synPi[\symA,\symB(x)]} \\
   \isterm {} {\syma} {\symA}
}{
   \isterm {} {\symapp{(\symA,\symB(x),\symf,\syma)}} {\symB(\syma)}
}
\end{gather}
While the rule is completely reasonable, making sense of it is rather subtle: checking, for instance, that $\symB(\syma)$ in the conclusion is well-formed requires applying some kind of inversion principle, to the type $\synPi[\symA,\symB(x)]$ from the premises.
Whether such an inversion principle is available depends on the particularities of the type theory under consideration.
In general, we want acceptable rules to be more straightforwardly and robustly well-behaved, so we expect that every metavariable used by the rule is explicitly introduced by some (unique) premise.

Finally, some rules have variant forms given by moving simple premises into the context of the conclusion.
For example, the rule for application is sometimes given as
\begin{equation}
\label{eq:example-app-6}
\inferrule{
  \istype {} {\symA} \\ 
  \istype {x \of \symA} {\symB(x)}
}{
 \isterm
   {x \of \symA, y \of \synPi[\symA, \symB(x)]}
   {\symapp{(\symA,\symB(x),y,x)}}
   {\symB(x)}
}
\end{equation}

This variant has been called the \emph{hypothetical} form, in contrast to the \emph{universal} form~\eqref{eq:example-app-1}.
With substitution included as a structural rule, the two forms are equivalent: each is derivable from the other.
In the absence of a substitution rule, they are not equivalent; the hypothetical form is too weak.
We have also heard it argued that the universal form should be seen as conceptually prior.
So both forms are arguably reasonable; but the universal form \eqref{eq:example-app-1}, with empty conclusion context, has the clearer claim, and no generality is lost by restricting to such forms.

Summarising the above discussion, there are several simple syntactic criteria commonly used as rules-of-thumb to determine “reasonability” of rules.
We now formally define these criteria, and collect them into a definition of \emph{acceptability} of rules.

\begin{definition}
  \label{def:tight-rule}%
  Suppose $R$ is a raw rule with arity $\arity{R}$ over a signature $\Sigma$. We say that $R$ is \defemph{tight} when
  there exists a bijection $\beta$ between the arguments of $\arity{R}$ and the object premises of $R$,
  such that for each argument $i$ of $\arity{R}$,
  \begin{enumerate}
    \item\label{item:tight-rule-ctx} the context of the premise $\beta(i)$ has the scope $\argbinder{\arity{R}}{i}$;
    \item\label{item:tight-rule-jf} the judgement form of the premise $\beta(i)$ is $\argclass{\arity{R}}{i}$;
    \item\label{item:tight-rule-hd} the head expression of the premise $\beta(i)$ is $\synmeta{i}(\fammap{\synvar{j}}{j \in \argbinder{\arity{R}}{i}})$.
  \end{enumerate}
\end{definition}

Note that the bijection~$\beta$ is unique, if it exists.
The definition of tightness is admittedly a bit technical, but it captures a well-formedness condition of rules which is familiar but infrequently discussed explicitly. Namely, a rule is tight if its object premises provide the ``typing'' of its metavariable symbols.

Tightness alone does not suffice to make a rule reasonable, e.g., the rule~\eqref{eq:example-app-3} is tight but still broken because the type expression $\symB(\syma)$ is senseless. We need another condition which ensures that the type and term expressions appearing in the rule make sense.

\begin{definition}
  To each judgement $\Gamma \types J$, we associate the family of \defemph{presuppositions} $\Presup {(\Gamma \types J)}$, defined as the judgements formed over $\Gamma$ by placing the boundary slots of $J$ in the head position as follows:
  \begin{align*}
  \Presup {(\istype \Gamma A)} &\defeq [ \; ], \\
  \Presup {(\isterm \Gamma s A)} &\defeq [ \istype \Gamma A ], \\
  \Presup {(\eqtype \Gamma A B)} &\defeq [ \istype \Gamma A, \istype \Gamma B ], \\
  \Presup {(\eqterm \Gamma s t A)} &\defeq [ \istype \Gamma A, \isterm \Gamma s A, \isterm \Gamma t A ].
  \end{align*}
\end{definition}

We shall need to know later on that presuppositions are natural with respect to the action of signature maps, instantiations, and raw substitutions.

\begin{proposition}
  \label{prop:presuppositions-action-signature-map}
  Let $\Gamma \types J$ be a judgement over~$\Sigma$, and $F : \Sigma \to \Sigma'$ a signature map. Then $\Presup {(\act{F} \Gamma \types \act{F} J)} = \act{F}(\Presup {(\Gamma \types J)})$.
\end{proposition}

\begin{proof}
  This is clear, for instance the presupposition of $\isterm{\act{F} \Gamma}{\act{F} s}{\act{F} A}$ is $\istype{\act{F} \Gamma}{\act{F} A}$, which is precisely what we get when $F$ acts on $\istype{\Gamma}{A}$, the presupposition of $\isterm{\Gamma}{s}{A}$.
\end{proof}

The reasoning that established the analogous statements about the actions of instantiations and raw substitutions is similarly easy.

\begin{propositionwithqed}
  \label{prop:presuppositions-action-instantiation}
  Let $\Gamma \types J$ be a judgement over a metavariable extension $\mvextend{\Sigma}{\alpha}$ and $I \in \Inst{\Sigma}{\gamma}{\alpha}$ an instantiation.
  Then $\Presup {(\act{I} \Gamma \types \act{I} J)} = \act{I}(\Presup {(\Gamma \types J)})$.
\end{propositionwithqed}

\begin{propositionwithqed}
  \label{prop:presuppositions-action-substitution}
  Let $\Gamma \types J$ be a judgement and $f : \Delta \to \Gamma$ a raw substitution. Every presupposition of $\Delta \types \tca{f} J$
  has the form $\Delta \types \tca{f} J'$, where $\Gamma \vdash J'$ is a presupposition of $\Gamma \types J'$.
\end{propositionwithqed}

There is a weaker and a stronger condition that we can impose on a rule with regards to the presuppositions of its conclusion.

\begin{definition}%
  \label{def:weakly-presuppositive-rule}%
  Let $T$ be a raw type theory over a signature $\Sigma$ and $R$ a raw rule over~$\Sigma$:
  \begin{enumerate}
  \item a raw rule $R$ is \defemph{weakly presuppositive over~$T$} when every presupposition of the conclusion of~$R$ is derivable in $T$ (translated from $\Sigma$ to $\mvextend{\Sigma}{\arity{R}}$) from the premises of~$R$ and the presuppositions of the premises of~$R$,

  \item a raw rule~$R$ is \defemph{presuppositive over~$T$} when all presuppositions of the conclusion and of the premises of~$R$ are derivable in $T$ (translated from $\Sigma$ to $\mvextend{\Sigma}{\arity{R}}$) from the premises of~$R$.
  \end{enumerate}
\end{definition}

As far as derivability is concerned, weakly presuppositive rules are good enough, for a rule cannot be applied unless its premises have already been derived, in which case their presuppositions will be derivable as well --- which is the gist of the proof of \cref{thm:presuppositions}.
However, if we were to give a meaning to a raw rule on its own, we would be hard-pressed to explain what the premises are about, unless their presuppositions were derivable as well, hence we take the stronger variant as the standard one.

\begin{definition}
  \label{def:acceptable-rule}%
  A raw rule $R$ is \defemph{acceptable} for a raw type theory $T$ if it is tight,
  presuppositive over~$T$, and has empty conclusion context.
\end{definition}

\begin{example}
  \parbox{0pt}{}
  \begin{enumerate}

  \item The above rules~\eqref{eq:example-app-1},~\eqref{eq:example-app-2},~\eqref{eq:example-app-3}, and~\eqref{eq:example-app-6} are tight.

  \item The above rules~\eqref{eq:example-app-1},~\eqref{eq:example-app-2},~\eqref{eq:example-app-4}, and~\eqref{eq:example-app-6} are presuppositive.

  \item The rule
    \begin{equation*}
      \infer{ }{\istype{}{\symA}}
    \end{equation*}
    which allows us to infer that every type expression is a type, is not tight.

  \item If $S$ is a type symbol with the empty arity, the rule
    \begin{equation*}
      \infer { } {\istype {} {S ()}}
    \end{equation*}
    is presuppositive and tight.

  \item Symmetry of type equality comes in two versions:
  \begin{equation*}
    \infer{
      \eqtype{}{\symA}{\symB}
    }{
      \eqtype{}{\symB}{\symA}
    }
    \qquad\qquad
    \infer{
      \istype{}{\symA}
      \\
      \istype{}{\symB}
      \\
      \eqtype{}{\symA}{\symB}
    }{
      \eqtype{}{\symB}{\symA}
    }
  \end{equation*}
  The left-hand one is not tight and is presuppositive, and the right-hand one is tight and presuppositive.
  \end{enumerate}
\end{example}

\begin{proposition}
  The congruence rule associated to an acceptable object rule is acceptable.
\end{proposition}
\begin{proof}
  Let~$R$ be a tight and presuppositive raw object rule over a raw type theory~$T$ with premises $\famtuple{\Gamma_i \types J_i}{i \in I}$.
  There exists a bijection~$\beta_R$ between object premises of~$R$ and the arguments of~$\arity R$.

  \cref{def:congruence-rule} lays out the associated congruence rule~$\congrule{R}$. Its arity is $\arity{\congrule{R}} = \arity R + \arity R$ and its premises are indexed by $I + I + I_{\ob}$, where $I_{\ob}$ is the set of object premises of~$R$.
  The bijection $\beta_{\congrule{R}}$ witnessing tightness of~$\congrule{R}$ is given by
  \begin{equation*}
    \beta_{\congrule{R}} (\inl(i)) \defeq \inl(\beta_R(i))
    \qquad\text{and}\qquad
    \beta_{\congrule{R}} (\inr(j)) \defeq \inr(\beta_R(j)).
  \end{equation*}
  Let us verify that the properties for tightness of~$\congrule{R}$ required in \cref{def:tight-rule} follow directly from the tightness of~$R$.
  For any $\inl(i) \in  \args {\arity{\congrule{R}}}$:
  \begin{enumerate}

  \item[\eqref{item:tight-rule-ctx}] The context of the premise $\beta_{\congrule{R}}(\inl(i)) = \inl(\beta_R(i))$ is $\act \ell \Gamma_i$. The signature map~$\ell : \mvextend{\Sigma}{\arity R} \to \mvextend{\Sigma}{\arity{\congrule{R}}}$ does not change the underlying scope of~$\Gamma_i$, and thus~$\act \ell \Gamma_i$ has the same underlying scope as~$\Gamma_i$, which equals $\argbinder {\arity R} i$ because~$R$ is tight. Furthermore, $\argbinder {\arity R} i = \argbinder {\arity{\congrule{R}}} (\inl(i))$, as required.

  \item[\eqref{item:tight-rule-jf}] The premise~$\beta_{\congrule{R}}(\inl(i))$ has judgement form~$\argclass{\arity{\congrule{R}}}{\inl(i)}$ by the analogous reasoning.

  \item[\eqref{item:tight-rule-hd}] The head of the premise~$\beta_{\congrule{R}}(\iota_i(i))$ is~$\act \ell e$ where $e = \synmeta{i}(\fammap{\synvar{j}}{j \in \argbinder{\arity{R}}{i}})$ by tightness of~$R$. We need to show that $\act \ell e = \synmeta{\inl(i)}(\fammap{\synvar{j}}{j \in \argbinder{\arity{\congrule{R}}}{\inl(i)}})$, but this equation holds by the definitions of~$\ell$ and of~$\arity{\congrule{R}}$.
  \end{enumerate}
  The case of $\inr(j) \in \args {\arity{\congrule{R}}}$ is symmetric.

  We also need to show that all presuppositions of the premises and the conclusion of~$\congrule{R}$ are derivable in~$T_{\congrule{R}}$ from the premises of~$\congrule{R}$, where~$T_{\congrule{R}} = \act {\inl{}} \circ T$ is the translation of~$T$ along $\inl : \Sigma \to \mvextend \Sigma {\arity{\congrule{R}}}$.

  Consider the premise $\Gamma_{\inl(i)} \types J_{\inl(i)}$ at index $\inl(i)$ for some $i \in I$. By \cref{prop:presuppositions-action-signature-map}, a presupposition of this premise is a presupposition $P = (\Gamma_i \types J')$ of the corresponding premise in~$R$, translated along the signature map~$\ell$.
  By presuppositivity of~$R$, the judgement~$P$ is derivable from~$T_R = \act {\inl{}} \circ T$, the translation of~$T$ along $\inl{} : \Sigma \to \mvextend \Sigma {\arity R}$. By \cref{cor:derivations-functorial-signature-maps}, we can translate such a derivation of~$P$ along~$\act \ell$, yielding a derivation in~$T' = \act \ell \circ T_R$, where~$T'$ is~$T_R$ translated along~$\ell$. But $T' = \act \ell \circ T_R = \act \ell \circ \act {\inl{}} \circ T = \act {\inl{}} T = T_{\congrule{R}}$, so we obtain a derivation in the correct theory.

  The case of a premise indexed by $\inr(j)$ with $j \in I$ is similar, but the last step requires translation along the signature map $r = id_{\mvextend \Sigma {\arity R}} + \inr$ instead, mapping the metavariable symbols of $\arity R$ to the right-hand side metavariables of~$\congrule{R}$.

  A premise~$P$ indexed by $\iota_2(k)$ is an equality associated to the $k$-th object premise of~$R$. The presuppositions of~$P$ are derived by the corresponding object premises~$\inl(k)$ and~$\inr(k)$, and in the case of a term equation, the presupposition of the left-hand side.

  A presupposition of the conclusion is derivable by appeal to the rule~$R$ itself for left and right hand side of the equation. In case~$\congrule{R}$ is a term equation, the type judgement arising as presupposition of the conclusion of~$R$ is derivable in~$T_R$ by presuppositivity of~$R$, and can be translated along~$\act \ell$ in the same way that we treated the left-hand copies of the premises.
\end{proof}

\begin{proposition}%
  \label{prop:structural-rules-acceptable}
  The raw structural rules, i.e., the equivalence relation rules and the conversion rules are acceptable for any type theory.
\end{proposition}

\begin{proof}
  Tightness is obvious. Presuppositivity is obvious for all but the conclusion of the equality conversion rule $\eqterm{}{\symb{s}}{\symb{t}}{\symB}$, which immediately follows from the ordinary conversion rule for term judgements.
\end{proof}

\subsection{Acceptable type theories}
\label{sec:acceptable-type-theories}

It may happen that a raw type theory is flawed, even though each of its rules is acceptable. For instance, we might simply forget to state a rule governing one of the symbols, or provide two contradicting rules for the same symbol. Thus we also need a notion of acceptability of a raw type theory.

\begin{definition}%
  \label{def:symbol-rule}%
  Suppose $\Sigma$ is a signature and $S \in \Sigma$ has arity $\arity{S}$.
  The \defemph{generic application of~$S$} is the expression
  \begin{equation*}
     \genapp{S} \defeq
     S(\fammap
         {\synmeta{i}(\tuple{\synvar{j}}{j \in \argbinder {\alpha_S} i})}
         {i \in \args \arity{S}}
     ).
  \end{equation*}
  We say that an inference rule~$R$ is a \defemph{symbol rule for~$S$} when its arity is $\arity{S}$, the judgement form of the conclusion is the syntactic class of~$S$, and its head is $\genapp{S}$.
\end{definition}

\begin{definition}
  \label{def:theory-good-properties}%
  A raw type theory $T$ over~$\Sigma$ is:
  \begin{enumerate}
  \item \defemph{tight} if its rules are tight and there is a bijection $\beta$ from the index set of~$\Sigma$ to the object rules of~$T$ such that, for every symbol $S$ of~$\Sigma$, $\beta(S)$ is a symbol rule for~$S$;
  \item \defemph{presuppositive} if all of its rules are presuppositive over $T$;
  \item \defemph{substitutive} if all its rules have empty conclusion context; and
  \item \defemph{congruous} if for every object rule of~$T$ the associated congruence rule (cf.~\cref{def:congruence-rule}) is a rule of~$T$.
  \end{enumerate}
  A raw type theory is \defemph{acceptable} if it enjoys all of these properties.
\end{definition}


The definition omits a common criterion for being ``reasonable'', namely there being a well-founded order that prevents cyclic references between parts of the theory. We address well-foundedness separately in \cref{sec:well-founded-type-theories}, and provide a couple of examples showing how cyclic references may appear in an acceptable type theory.

\begin{example}
  \label{ex:cyclic-quantifier}%
  Let $\symb{Q}$ be a quantifier-like type symbol which takes a type and a term, and binds one variable in the term, with the raw rule
  \begin{equation*}
    \infer{
      \istype{}{\symA}
      \\
      \isterm
        {\synvar{0} \of \symb{Q}(\symA, \symb{t}(\synvar{0}))}
        {\symb{t}(\synvar{0})}
        {\symA}
    }{
      \istype{}{\symb{Q}(\symA, \symb{t}(\synvar{0}))}
    }
  \end{equation*}
  The context in the second premise is \emph{not} cyclic because $\symb{Q}$ binds $\synvar{0}$, but the premise itself is cyclic because the term metavariable $\symb{t}$ is introduced in a context that mentions it, and the rule is only presuppositive thanks to itself.
  Even so, the rule can still be used to derive judgements. For example, for any
  $\isterm{}{t}{A}$ we can form the type $\symb{Q}(A, t)$.
  It is not clear what one would do with such rules, but we have no reason to banish them outright.
\end{example}

\begin{example}
  \label{ex:type-in-type}%
  The second example of cyclic references is a Tarski-style universe that contains itself, formulated as follows.
  Let $\symb{u}$ be a term constant and $\symb{El}$ a type symbol taking one term argument, with the raw rules
  \begin{mathpar}
    \infer{
    }{
      \isterm{}{\symb{u}}{\symb{El}(\symb{u})}
    }

    \infer{
      \isterm{}{\symb{a}}{\symb{El}(\symb{u})}
    }{
      \istype{}{{\symb{El}(\symb{a})}}
    }
  \end{mathpar}
  Think of $\symb{u}$ as the code of the universe $\symb{El}(\symb{u})$ that contains itself, and $\symb{El}$ as the constructor taking codes to types. The rules themselves are not cyclic,
  and the type theory comprising them and the associated congruence rules is acceptable. However, in order to derive $\istype{}{\symb{El}(\symb{u})}$, which is a presupposition for both rules, we need both rules.
  In this case the cycles can be broken easily enough: introduce a type constant $\istype{}{\symb{U}}$ and the equation $\eqtype{}{\symb{U}}{\symb{El}(\symb{u})}$, then use $\symb{U}$ in place of $\symb{El}(\symb{u})$ in the above rules.
  In \cref{sec:well-founded-replacement} we shall provide a general method for removing cyclic dependencies between rules by introduction of new symbols.
\end{example}

\subsection{Derivability of presuppositions}

Our first meta-theorem is a fairly easy one, giving a property that is always desired but not often explicitly discussed.

\begin{theorem}[Presuppositions theorem]
  \label{thm:presuppositions}%
  Let $T$ be a raw type theory with all rules weakly presuppositive.
  If a judgement is derivable over $T$, then so are all its presuppositions.
\end{theorem}

\begin{proof}
  We proceed by induction on derivations $D$ over~$T$.

  If $D$ ends with a variable rule (\cref{def:variable-rule}), then the only presupposition appears directly as the premise of the rule, so we may re-use its subderivation.

  If $D$ ends with a substitution rule (\cref{def:substitution-rule}), then its conclusion must be $\Delta \types \tca{f} J$ for some substitution~$f : \Delta \to \Gamma$ and judgement~$\Gamma \types J$.
  By \cref{prop:presuppositions-action-substitution}, each presupposition of the conclusion is $\Delta \types \tca{f} J'$ for some presupposition $\Gamma \types J'$ of $\Gamma \types J$.
  But $\Gamma \types J$ is a premise of the last rule of $D$, so by induction we have a derivation $D'$ of $\Gamma \types J'$.
  So we  can apply the substitution rule with $\Gamma \types J'$ (derived by $D'$) and the same substitution $f$ (with its premises derived as in $D$) to get the desired derivation of $\Delta \types \tca{f} J'$.
  
  Similarly, if $D$ ends with an equality substitution rule (\cref{def:equality-substitution-rule}), substituting an pair $f,g : \Delta \to \Gamma$ into a judgement $\Gamma \types J$, each presupposition of the conclusion can be derived by either a substitution (along $f$ or $g$ individually) or an equality substitution (along $f,g$) into some presupposition of $\Gamma \types J$.

  The equivalence and conversion rules (\cref{def:equivalence-relation-rule,def:conversion-rule}) are presuppositive by \cref{prop:structural-rules-acceptable}, so we treat them together with the specific raw rules of~$T$.

  If $D$ ends with an instance $\act{(I,\Gamma)} R$ of a raw rule $R$ (either specific or structural), then its conclusion is of the form $\act{(I,\Gamma)} \Delta \types \act{I} J$, where $\Delta \types J$ is the conclusion of~$R$.
  Now \cref{prop:presuppositions-action-instantiation} tells us that each presupposition of the conclusion is an instantiation $\act{(I,\Gamma)} \Delta \types \act{I} J'$ of some presupposition $\Delta \types J'$ of $\Delta \types J$.
  Since $R$ is weakly presuppositive, $\Delta \types J'$ is derivable from the premises of~$R$ plus their presuppositions.
  So by \cref{cor:instantiation-of-derivations}, $\act{(I,\Gamma)} \Delta \types \act{I} J'$ is derivable from the premises of $\act{(I,\Gamma)} R$ plus their presuppositions, which in turn are derivable by induction.
\end{proof}

\subsection{Elimination of substitution}
\label{sec:elimination-substitution}

In this section we show that over an acceptable type theory, the substitution rules (\cref{def:substitution-rule,def:equality-substitution-rule}) can be eliminated: anything derivable with them is derivable without.
At least, this will hold over a strict scope system;
for a general scope system, it can be almost eliminated but not quite entirely.

\begin{definition}
  An instance of the substitution rule (\cref{def:substitution-rule}) is a \defemph{trivial renaming}, or just \defemph{trivial}, if its substitution $f : \Delta \to \Gamma$ corresponds to a renaming on underlying scopes of the form $\inlscope^{-1} : \position{\Gamma} = \sumscope{\position{\Delta}}{\emptyscope} \to \position{\Delta}$, acting trivially at all positions.
\end{definition}

Typically these arise with $\Gamma = \act{(I,\Delta)}\emptycxt$, an instantiation of the empty context;
a trivial renaming is therefore of the form
\[
  \infer{
      \act{(I,\Delta)}\emptycxt \typesjudgement J
    }{
      \Delta \typesjudgement \act{(\inlscope^{-1})} J.
    }
\]
In a \emph{strict} scope system, trivial renamings are identities and hence redundant.

To avoid ambiguity with variance, we will in this section distinguish more carefully than usual between a renaming function $r : \position{\Gamma} \to \position{\Delta}$ and its associated substitution $\bar{r} : \Delta \to \Gamma$.

\begin{definition}
  Call a derivation over a raw type theory~$T$ \defemph{substitution-free} if it uses only trivial instances of the substitution rule, and does not use the equality substitution rule.
  Equivalently, it uses just the variable rule, equality rules, conversion rules, trivial renamings, and the specific rules of~$T$.
\end{definition}

The core of this section, \cref{lem:admissibility-substitution}, will be that substitution is admissible for substitution-free derivations; this can be seen as defining an action of substitution on such derivations.
We first need an analogous action of renaming, paralleling how substitution on expressions needed renaming to be defined first.

\begin{lemma}[Admissibility of renaming]
  \label{lem:admissibility-renaming}%
  Let $T$ be a substitutive type theory, with signature $\Sigma$.
  Let $\Gamma$ and $\Gamma'$ be contexts over $\Sigma$, and $r : \position{\Gamma} \to \position{\Gamma'}$ a renaming acting trivially at all positions in the sense of~\cref{def:substitution-rule}, i.e.\ such that $\Gamma'_{r(i)} = \act{r}\Gamma_i$ for all $i \in \Gamma$.
  Then given a substitution-free derivation $D$ of $\Gamma \types J$ in $T$, there is a substitution-free derivation $\act{r}D$ of $\Gamma' \types \act{r} J$.
\end{lemma}

\begin{proof}
  For this proof, we say a renaming \defemph{respects types} when it acts trivially at all positions; and say a derivation $D$ with conclusion $\Gamma \types J$ is \defemph{renameable} if we have an operation giving, for every $\Gamma'$ and renaming $r : \position{\Gamma} \to \position{\Gamma'}$ respecting types, a derivation $\act{r}D$ of $\Gamma' \types \act{r} J$.
  
  We show by induction that every derivation is renameable.
  Call the derivation under consideration $D$, and suppose given in each case suitable $\Gamma'$, $r$.
  
  If $D$ concludes with a variable rule, giving $\isterm{\Gamma}{\synvar{i}}{\Gamma_i}$, then by induction, we can rename the derivation of the premise ${\istype{\Gamma}{\Gamma_i}}$ to a derivation of $\istype{\Gamma'}{\Gamma'_{r(i)}}$,
  and then apply the variable rule to derive $\isterm{\Gamma'}{\synvar{r(i)}}{\Gamma'_{r(i)}}$, which is the desired judgement since~$r$ respects types.

  If $D$ concludes with a trivial renaming $s : \Gamma \to \Delta$,
  then by induction, the derivation of the premise is renameable; so renaming it along $r \circ s$, we are done.

  Otherwise, $D$ concludes with an instantiation $\act{(I,\Theta)} R$, where $I \in \Inst{\Sigma}{\Theta}{\arity{R}}$ is an instantiation, and~$R$ is either an equality rule, a conversion rule, or a specific rule of~$T$.
  %
  In each cases the conclusion of~$R$ is of the form $\typesjudgement J'$, with empty context; so $\Gamma$ is exactly $\act{(I,\Theta)}\emptycxt$, and $J$ is has the form $\act{I} J'$.
  %
  %
  So now to derive $\Gamma' \types \act{r}{(\act{I} J')}$, we will apply the same raw rule~$R$ with the instantiation $I' \defeq \act{(r \circ \inlscope)}{I}$.
  Computing with renamings and instantiations according to (\cref{prop:instantiation-boilerplate}) shows that the conclusion of $\act{(I',\Gamma')} R$ is not quite $\Gamma' \typesjudgement \act{r}J$, but is the same modulo a trivial renaming, according to the following commutative square:
  \[
    \xymatrix@C=3em{
      \Gamma' \ar[r]^{\overline{\inlscope} \circ {\bar{r}}} \ar[dr]^{\bar{r}} & \Theta \\
      \act{(I',\Gamma')} \emptycxt \ar[r]  \ar[u]^{\overline{\inlscope}} & \Gamma \mathrlap{{} = \act{(I,\Theta)} \emptycxt} \ar[u]_{\overline{\inlscope}}
    }
    \phantom{\act{(\Theta)} } 
  \]
  We can therefore conclude the derivation of $\act{r}D$ by the rule $\act{(I',\Gamma')} R$ followed by a trivial renaming.

  It remains to derive the premises of $\act{(I',\Gamma')} R$.
  Each such premise is linked to a corresponding premise of $\act{(I,\Theta)} R$ by a renaming repecting types --- specifically, a context extension of $r \circ \inlscope$.
  But by induction, we have renameable derivations of the premises of $\act{(I,\Theta)} R$; so we are done.
\end{proof}

It is worthwhile to record a special case of admissibility of renaming.

\begin{corollary}[Admissibility of weakening]
  \label{cor:admissibility-weakening}%
  If a substitutive raw type theory derives $\Gamma \types J$ substitution-free, then it also derives $\Delta \types \act{w} J$ substitution-free for any \defemph{weakening} $w : \Gamma \to \Delta$, i.e., an injective variable renaming such that $\Delta_{w(i)} = \act{w} \Gamma_i$ for all $i \in \position{\Gamma}$. \qed
\end{corollary}

We can now give the action of substitution on derivations.

\begin{lemma}[Admissibility of substitution]%
  \label{lem:admissibility-substitution}%
  Let $T$ be a substitutive raw type theory over signature~$\Sigma$.
  Let $f : \Delta \to \Gamma$ be a raw substitution over~$\Sigma$, and $K \subseteq \position{\Gamma}$ a complemented subset such that:
  \begin{enumerate}
  \item \label{item:subst-trivial-case} $f$ acts trivially at each $i \in K$ in the sense of~\cref{def:substitution-rule}, i.e., for some $j \in \Delta$, $f(i) = \synvar{j}$ and $\Delta_j = f^*\Gamma_i$
  \item \label{item:subst-nontrivial-case} for each $i \in \position{\Gamma} \setminus K$, $T$ derives $\isterm{\Delta}{f(i)}{f^*\Gamma_i}$ without substitutions.
  \end{enumerate}
  Given a substifution-free derivation $D$ of $\Gamma \types J$ over $T$, there is a substitution-free derivation $\tca{f}D$ of $\Delta \types \tca{f} J$.
\end{lemma}

Before proceeding with the proof, we take a moment to comment on the condition the lemma assumes on~$f$. It matches the condition used in the premises of the substitution rule, skipping type-checking on a set of indices~$K$ on which~$f$ acts trivially, and requiring derivability of $\isterm{\Delta}{f(i)}{\tca{f} \Gamma_i}$ only for $i \in \position{\Gamma} \setminus K$.
An alternative, maybe more conventional condition would be to require $\isterm{\Delta}{f(i)}{\tca{f} \Gamma_i}$ for all $i \in \position{\Gamma}$.

There are a couple of reasons to weaken the condition as we do, thus strengthening the statement of the lemma. The superficial one is its applicability to raw substitutions that potentially contain ill-formed type expressions.
The more essential one is that strengthened formulation is needed to keep the proof structurally inductive, allowing us to descend under a premise with a non-empty context, without needing to check that in the process the domain of $\tca{f}$ is extended with well-formed types.
Even if they are in fact well formed, we cannot show this by appealing to an induction hypothesis, because the derivations involved are not structural subderivations of the one we are recursing over.
What happens instead is that verification of well-formedness of types in contexts is deferred until their variables are accessed, at which point the variable rule provides the desired structural subderivations.
This phenomenon seems to be a genuine consequence of spelling out the proof for a \emph{general} class of type theories.
For any specific type theory, only certain concrete type-schemes will occur in contexts of premises of rules; and
these specific type-schemes are always designed by their authors in such a way that they can be shown well-formed individually, so that the inductive arguments do not break.

\begin{proof}[Proof of \cref{lem:admissibility-substitution}]
  Within this proof, all derivations are assumed substitution-free, and a (substitution-free) derivation~$D$ of a judgement $\Gamma \types J$ is called \defemph{substitutable} when, for all~$\Delta$, $f$ and $K$ satisfying the condition of the lemma, we have a (substitution-free) derivation of $\Delta \types f^*J$.
  We prove by induction that every derivation $D$ is substitutable.
  Much of the proof parallels that of \cref{lem:admissibility-renaming}.

  Suppose $D$ concludes with a variable rule showing $\isterm{\Gamma}{\synvar{i}}{\Gamma_i}$, and $f : \Delta \to \Gamma$ is a suitable substitution, acting trivially on $K \subseteq \position{\Gamma}$.
  When $i \in K$, we work just as in \cref{lem:admissibility-renaming}:
  given a suitable substitution into the conclusion, we inductively substitute the premise derivation along the same substitution, and then conclude with the variable rule.
  Otherwise, for $i \in \position{\Gamma}\setminus K$, we use the derivation of $\isterm{\Delta}{f(i)}{\tca{f}\Gamma_i}$ given by assumption. 
  
  Next, if $D$ concludes with a trivial renaming $\overline{\inlscope^{-1}} : \Gamma \to \Gamma'$, to conclude $\Gamma \typesjudgement J$, then suppose $f : \Delta \to \Gamma$ is a substitution acting trivially on $K$ and with derivations of $\isterm{\Delta}{f(i)}{\tca{f}\Gamma_i}$ for $i \in \position{\Gamma} \setminus K$.
  Then the substitution $\overline{\inlscope^{-1}} \circ f : \Delta \to \Gamma'$ acts trivially on $\inlscope(K) \subseteq \position{\Gamma'}$, and the same derivations witness that $\isterm{\Delta}{f(\inlscope^{-1}i)}{\tca{f}\Gamma'_i}$ for $i \in \position{\Gamma'} \setminus \inlscope{(K)}$.
  So by induction, we can substitute the derivation of the premise along $\overline{\inlscope^{-1}} \circ f $ to derive $\Delta \types \tca{f}J$as required.
  
  Otherwise, $D$ must conclude with an instantiation $\act{(I,\Theta)} R$, for some instantiation $I \in \Inst{\Sigma}{\Theta}{\arity{R}}$ and~$R$ a flat rule (structural or specific) with empty-context conclusion $\typesjudgement J$.
  So, suppose given a suitable substitution $f : \Delta \to \act{(I,\Theta)}\emptycxt$, acting trivially on $K \subseteq \position{\act{(I,\Theta)}\emptycxt} = \sumscope{\Theta}{\emptyscope}$; we need to derive $\Delta \typesjudgement f^*\act{I}J$.

  Just as in \cref{lem:admissibility-renaming}, we substitute $I$ along $\overline{\inlscope} \circ f : \Delta \to \Theta$ to get another instantiation $I'$ of $\arity{R}$ over $\Delta$, such that the conclusion of $\act{(I',\Delta)}R$ is just a trivial renaming away from $\Delta \typesjudgement f^*\act{I}J$.
  So it remains just to derive the premises of $\act{(I',\Delta)}R$.
  
  Again as in \cref{lem:admissibility-renaming}, by induction we have substitutable derivations of all premises of $\act{(I,\Theta)}R$.
  So it suffices to give, for each premise $\act{(I',\Delta)} \Psi \typesjudgement \act{I'}J'$ of $\act{(I',\Delta)}R$, a substitution $g : \act{(I',\Delta)} \Psi  \to \act{(I,\Theta)} \Psi$ to the corresponding premise of $\act{(I,\Theta)}R$, with $g^* \act{I}J' = \act{I'}J'$, and with $g$ satisfying the conditions of the lemma.
  (Recall that $\act{(I,\Theta)} \Psi$ is the context extension of $\Theta$ by the instantiations of types from $\Psi$, with positions $\sumscope{\position{\Theta}}{\position{\Psi}}$, and $\act{(I',\Delta)}$ similarly.)
  We define:
  \[ g \defeq \sumscope{(\overline{\inlscope} \circ f)}{\position{\Psi}} : \act{(I',\Delta)} \to \act{(I,\Theta)} \Psi. \]
  
  Now $g^* \act{I}J' = \act{I'}J'$ follows directly from \cref{prop:instantiation-boilerplate} (which we will continue to use without further comment), the definitions of $I'$ and $g$, and the following commuting diagram.
  \[
    \xymatrix@C=3em{
      \act{(I',\Delta)} \Psi \ar[d]_{\overline{\inlscope}} \ar[r]^{g} & \act{(I,\Theta)} \Psi \ar[d]^{\overline{\inlscope}} \\   
      \Delta \ar[r]^{\overline{\inlscope} \circ f} \ar[dr]^{f} & \Theta \\
      \act{(I',\Delta)} \emptycxt \ar[r] \ar[u]^{\overline{\inlscope}} & \act{(I,\Theta)} \emptycxt \ar[u]_{\overline{\inlscope}} 
    }
  \]
 
  Next, $g$ clearly acts trivially on $\act{\inrscope}\position{\Psi} \subseteq \position{\act{(I,\Theta)} \Psi}$.
  It also acts trivially on the subset of $\act{\inlscope}{\position{\Theta}}$ corresponding to the given $K \subseteq \sumscope{\position{\Theta}}{\emptyscope}$ on which $f$ acts trivially.
  
  Taking the union of these as the trivial set for $g$, it remains to show that for $i$ in the subset of $\act{\inlscope}{\position{\Theta}}$ corresponding to $\sumscope{\position{\Theta}}{\emptyscope} \setminus K$, we have $\isterm{\act{(I',\Delta)}\Psi}{g(i)}{\tca{g}(\act{(I,\Theta)} \Psi)_i}$.
  But this judgement is just the renaming of $\isterm{\Delta}{f(j)}{\tca{f}(\act{(I,\Theta)}\emptyset)_i}$ along the evident map $\sumscope{\position{\Delta}}{\emptyscope} \to \sumscope{\position{\Delta}}{\position{\Psi}}$. 
  So using \cref{lem:admissibility-renaming} to rename the derivation of $\isterm{\Delta}{f(j)}{\tca{f}(\act{(I,\Theta)}\emptycxt)_i}$ supplied with $f$, we are done.
\end{proof}

Next, we show that substitution respects judgemental equality of raw substitutions.
For this, we introduce a handy notation: for raw substitutions $f, g : \Gamma' \to \Gamma$ and an object judgement $J$, with head expression $e$ and boundary $B$, we write $\tca{(f \equiv g)} J$ for the equality judgement asserting that $\tca{f} e$ and $\tca{g} e$ are equal over the boundary $\tca{f} B$. Thus $\Gamma' \types \tca{(f \equiv g)} (A \type)$ stands for $\eqtype{\Gamma'}{\tca{f} A}{\tca{g} A}$ and $\Gamma' \types \tca{(f \equiv g)} (e : A)$ stands for $\eqterm{\Gamma'}{\tca{f} e}{\tca{g} e}{\tca{f} A}$.

\begin{lemma}[Admissibility of equality substitution]
  \label{lem:admissibility-equality-substitution}
  Let $T$ be a substitutive and congruous raw type theory over~$\Sigma$.
  Let $f, g : \Gamma' \to \Gamma$ be raw substitutions over~$\Sigma$,
  and $K \subseteq \position{\Gamma}$ a complemented subset such that:
  \begin{enumerate}

  \item \label{item:adm-subst-eq-trivial-case}%
    for each $i \in K$ there exists some (necessarily unique) $j \in \position{\Gamma'}$ such that
    $f(i) = g(i) = \synvar{j}$, and $\Gamma'_j = \tca{f} \Gamma_i$ or $\Gamma'_j = \tca{g} \Gamma_i$,

  \item \label{item:adm-subst-eq-nontrivial-case}%
    for each $i \in \position{\Gamma} \setminus K$,
    $T$ derives $\isterm{\Gamma'}{f(i)}{ \tca{f} \Gamma_i}$ and $\isterm{\Gamma'}{g(i)}{\tca{g} \Gamma_i}$ and $\eqterm{\Gamma'}{f(i)}{g(i)}{\tca{f} \Gamma_i}$ without substitutions.
  \end{enumerate}
  If $T$ derives $\Gamma \types J$ without substitutions, then $T$ derives $\Gamma' \types \tca{f} J$, $\Gamma' \types \tca{g} J$, and (if $J$ is an object judgement) $\Gamma' \types \tca{(f \equiv g)} J$, still without substitutions.
\end{lemma}

The assumption on~$f$ and~$g$ is perhaps a little surprising, especially the last ``or'' in case~\eqref{item:adm-subst-eq-trivial-case}. Another peculiarity is the fact that we include $\Gamma' \types \tca{f} J$ and $\Gamma' \types \tca{g} J$ in the conclusion rather than obtaining them by elimination of substitution.
The point is that the induction arguments need to work when we pass into extended contexts of premises, whereby types of the form $\tca{f} A$ \emph{or} $\tca{g} A$ are introduced; so we cannot assume either~$f$ or~$g$ satisfying the conditions of \cref{thm:elimination-substitution} individually, but need to give a condition on them together that is preserved.
And since this condition is too weak for applying elimination of substitution to~$f$ or~$g$, we carry the conclusion of that along as well.

\begin{proof}[Proof of \cref{lem:admissibility-equality-substitution}]
  The proof proceeds by induction on the derivation $D$ of $\Gamma \types J$.
  The details are closely analogous to elimination of substitution, so we spell out fewer.

  Consider the case when $D$ ends with a variable rule
  \begin{equation*}
    \infer{
      \istype{\Gamma}{\Gamma_i}
    }{
      \isterm{\Gamma}{\synvar{i}}{\Gamma_i}
    }
  \end{equation*}
  If case~\eqref{item:adm-subst-eq-nontrivial-case} applies for~$i$, we are done immediately.
  If case~\eqref{item:adm-subst-eq-trivial-case} applies, we obtain derivations of $\istype{\Gamma'}{\act{f} \Gamma_i}$, $\istype{\Gamma'}{\act{g} \Gamma_i}$, and $\eqtype{\Gamma'}{\act{f} \Gamma_i}{\act{g} \Gamma_i}$ by induction hypothesis.
  If $\Gamma'_j = \act{f} \Gamma_i$ then $\isterm{\Gamma'}{\synvar{j}}{\act{f} \Gamma_i}$ follows by the variable rule and $\isterm{\Gamma'}{\synvar{j}}{\act{g} \Gamma_i}$ from it by conversion. And of course, $\eqterm{\Gamma'}{\synvar{j}}{\synvar{j}}{\Gamma'_j}$ is derivable by reflexivity. If $\Gamma'_j = \act{g} \Gamma_i$, the situation is symmetric.

  Otherwise, $D$ concludes with an instantiation $\act{I} R$ where $I \in \Inst{\Sigma}{\Gamma}{\arity{R}}$ and~$R$ is either an equality rule, a conversion rule, or a specific rule of~$T$. The conclusion of $\act{I} R$ has the form $\Gamma \types \act{I} J'$.
  We need to derive $\Gamma' \types \tca{f} (\act{I} J')$, $\Gamma' \types \tca{g} (\act{I} J')$, and if $R$ is an object rule then also $\Gamma' \types \tca{(f \equiv g)} (\act{I} J')$.

  Let us first verify that the raw substitutions $f' \defeq \sumscope f {\position{\Delta}} : \ctxextend {\Gamma'}{\act{(\tca{f} I)} \Delta} \to \ctxextend \Gamma {\act{I} \Delta}$ and $g' \defeq \sumscope g {\position{\Delta}} : \ctxextend {\Gamma'}{\act{(\tca{f} I)} \Delta} \to \ctxextend \Gamma {\act{I} \Delta}$ satisfy the conditions~\eqref{item:adm-subst-eq-trivial-case} and~\eqref{item:adm-subst-eq-nontrivial-case} of the lemma for the complemented subset $K' \subseteq \position{\ctxextend{\Gamma}{\act{I} \Delta}} = \position{\Gamma} + \position{\Delta}$, given by
  $
    K' = \set{\inl(i) \such i \in K} \cup \set{\inr(k) \such k \in \position{\Delta}}
  $:
  \begin{enumerate}

  \item
    For $\inl(i) \in K'$, there is $j \in \Gamma'$ such that $f(i) = g(i) = \synvar{j}$ and either $\Gamma'_j = \tca{f} \Gamma_i$ or $\Gamma'_j = \tca{g} \Gamma_i$.
    Now $f'$ and $g'$ satisfy condition~\eqref{item:subst-trivial-case} because $f'(\inl(i)) = \synvar{\inl(j)} = g'(\inl(i))$ and
    $(\ctxextend{\Gamma'}{\act{(\tca{f} I) \Delta}})_{\inl(j)} = \act{{\inl}} \Gamma'_j$, which is equal either to $\act{{\inl}} (\tca{f} \Gamma_i) = \tca{f'} (\ctxextend \Gamma {\act{I} \Delta})$ or to $\act{{\inl}} (\tca{g} \Gamma_i) = \tca{g'} (\ctxextend \Gamma {\act{I} \Delta})$, as the case may be.

  \item
    For $\inr(k) \in K'$, condition~\eqref{item:subst-trivial-case} is satisfied by $f'$ and $g'$: it is clear that
    $
    f'(\inr(k)) = \synvar{\inr(k)} = g'(\inr(k))
    $,
    while
    $
    \tca{f'} (\ctxextend \Gamma {\act{I} \Delta})_{\inr(k)}
    = \tca{f'} (\act{I} \Delta_k)
    = \act{(\tca{f} I)} \Delta_k
    = (\ctxextend {\Gamma'}{\act{(\tca{f} I)} \Delta})_{\inr(k)}
    $
    holds, where we used \cref{prop:instantiation-boilerplate} in the second step.

  \item
    For $\inl(j) \in (\position{\Gamma} + \position{\Delta}) \setminus K'$, $f'$ and $g'$ satisfy \eqref{item:subst-nontrivial-case} because the desired judgements
    \begin{align*}
      \isterm
        {\ctxextend {\Gamma'}{\act{(\tca{f} I)} \Delta} &}
        {f'(\inl(j))}
        {\tca{f'} (\ctxextend \Gamma {\act{I} \Delta})_{\inl(j)}}
      \\
      \isterm
        {\ctxextend {\Gamma'}{\act{(\tca{f} I)} \Delta} &}
        {g'(\inl(j))}
        {\tca{g'} (\ctxextend \Gamma {\act{I} \Delta})_{\inl(j)}}
      \\
      \eqterm
        {\ctxextend {\Gamma'}{\act{(\tca{f} I)} \Delta} &}
        {f'(\inl(j))}
        {g'(\inl(j))}
        {\tca{f'} (\ctxextend \Gamma {\act{I} \Delta})_{\inl(j)}}
      \\
      \intertext{are respectively equal to}
      \isterm
        {\ctxextend {\Gamma'}{\act{(\tca{f} I)} \Delta} &}
        {\act{{\inl}} (f(j))}
        {\act{{\inl}} (\tca{f} \Gamma_j)}
      \\
      \isterm
        {\ctxextend {\Gamma'}{\act{(\tca{f} I)} \Delta} &}
        {\act{{\inl}} (g(j))}
        {\act{{\inl}} (\tca{g} \Gamma_j)}
      \\
      \eqterm
        {\ctxextend {\Gamma'}{\act{(\tca{f} I)} \Delta} &}
        {\act{{\inl}} (f(j))}
        {\act{{\inl}} (g(j))}
        {\act{{\inl}} (\tca{f} \Gamma_j)}
    \end{align*}
    and these are derivable by \cref{lem:admissibility-renaming} applied to the renaming $\inl$ and the assumptions~\eqref{item:subst-nontrivial-case} for~$f$ and~$g$.
  \end{enumerate}

  We similarly check that the raw substitutions $f'' \defeq \sumscope f {\position{\Delta}} : \ctxextend {\Gamma'} {\act{(\tca{g} I)} \Delta} \to \ctxextend \Gamma {\act{I} \Delta}$ and $g'' \defeq \sumscope g {\position{\Delta}} : \ctxextend {\Gamma'} {\act{(\tca{g} I)} \Delta} \to \ctxextend \Gamma {\act{I} \Delta}$ satisfy the conditions of the lemma with the same set $K'$, too. The verification is similar to the case of $f'$ and $g'$ above, and at this point the disjunction in~\eqref{item:subst-trivial-case} lets us exchange the role of~$g$ and~$f$.

  We may now derive $\Gamma' \types \tca{f} (\act{I} J')$ by the closure rule $\act{(\tca{f}I)} R$, as it has the correct conclusion by \cref{prop:instantiation-boilerplate}. To see that its premises are derivable, we verify that for any premise $\Delta \types J''$ of~$R$, the instantiation by $\tca{f} I$, namely
  \begin{equation}
    \label{eq:adm-subst-eq-1}
    \ctxextend {\Gamma'}{\act{(\tca{f} I)} \Delta} \types \act{(\tca{f} I)} J''
  \end{equation}
  is derivable. The corresponding premise of $\act{I} R$, which is
  \begin{equation}
    \label{eq:adm-subst-eq-2}
    \ctxextend{\Gamma} {\act{I} \Delta} \types \act{I} J'',
  \end{equation}
  is derivable by assumption, and so we obtain~\eqref{eq:adm-subst-eq-1} by the induction hypothesis for~\eqref{eq:adm-subst-eq-2} applied to $f'$ and $g'$.

  By a similar argument $\Gamma' \types \tca{g} (\act{I} J')$ is derivable by the closure rule $\act{(\tca{g} I)} R$, we only need to use $f''$ and $g''$ instead of $f'$ and $g'$ to derive the premise
  \begin{equation}
    \label{eq:adm-subst-eq-3}
    \ctxextend {\Gamma'} {\act{(\tca{g} I)} \Delta} \types \act{(\tca{g} I)} J''.
  \end{equation}

  It remains to be checked that $\Gamma' \types \tca{(f \equiv g)} (\act{I} J')$ is derivable when~$R$ is an object rule. Because $T$ is congruous, the congruence rule~$C$ associated with~$R$ is a specific rule of~$T$. Let $\ell, r : \mvextend{\Sigma}{\arity{R}} \to \mvextend{\Sigma}{\arity{C}}$ be the signature maps from \cref{def:congruence-rule}
  and $I' \defeq \tca{f} I + \tca{g} I \in \Inst{\Sigma}{\Gamma'}{\arity{R} + \arity{R}}$. Note that $\act{I'} \circ \act{\ell} = \tca{f} I$ and $\act{I'} \circ \act{r} = \tca{g} I$.
  The instantiation $\act{I'} C$ is a closure rule whose conclusion is precisely $\Gamma' \types \tca{(f \equiv g)} (\act{I} J')$, so we only have to establish that its premises are derivable, of which there are three kinds:
  \begin{enumerate}
  \item For each premise $\Delta \types J''$ of $R$, there is a corresponding premise $\act{I'}(\act{\ell}(\Delta \types J''))$, which is equal to~\eqref{eq:adm-subst-eq-1}.
    We have already seen that it is derivable.
  \item For each premise $\Delta \types J''$ of $R$, there is a corresponding premise $\act{I'}(\act{r}(\Delta \types J''))$, which is equal to~\eqref{eq:adm-subst-eq-3}.
    Its derivability has been established, too.
  \item For each object premise $\Delta \types J''$ of $R$, there is a corresponding premise, namely the associated equality judgement (\cref{def:judgement-associated-congruence}) instantiated by~$I'$. A short calculation relying on \cref{prop:instantiation-boilerplate} shows that the judgement is
    \begin{equation*}
      \ctxextend {\Gamma'}{\act{(\tca{f} I)} \Delta} \types \tca{(f' \equiv g')} J'',
    \end{equation*}
    which is one of the consequences of the induction hypothesis for~\eqref{eq:adm-subst-eq-2} applied to $f'$ and $g'$. \qedhere
  \end{enumerate}
\end{proof}

We can now put these together into the main theorem of this section.

\begin{theorem}[Elimination of substitution]
  \label{thm:elimination-substitution}%
  Let $T$ be a substitutive and congruous raw type theory; then every derivable judgement over $T$ has a substitution-free derivation.
\end{theorem}

\begin{proof}
  Work by induction over the original derivation.
  At substitution rules, apply \cref{lem:admissibility-substitution}; and at equality substitution rules, \cref{lem:admissibility-equality-substitution}.
\end{proof}



\subsection{Uniqueness of typing}

Whether it is desirable for a term to have many types depends on one's motivations, but certainly in our setting, where the terms record detailed information about premises, we should expect a term to have at most one type, which we prove here.

\begin{theorem}
  \label{thm:tight-uniqueness-of-typing}
  If a tight, substitutive raw type theory $T$ derives $\istype \Gamma A$, $\istype \Gamma B$, $\isterm \Gamma t A$ and $\isterm \Gamma t B$ then it also derives $\eqtype \Gamma A B$.
\end{theorem}

\begin{proof}
  By \cref{thm:elimination-substitution} it suffices to prove the claim for substitution-free derivations.
  %
  Suppose we have derivations $D_A$, $D_B$, $D_1$ and $D_2$:
  \begin{mathpar}
    \infer
    {D_A}
    {\istype{\Gamma}{A}}

    \infer
    {D_B}
    {\istype{\Gamma}{B}}

    \infer
    {D_1}
    {\isterm{\Gamma}{t}{A}}

    \infer
    {D_2}
    {\isterm{\Gamma}{t}{B}}
  \end{mathpar}
  The proof proceeds by a double induction on the derivations $D_1$
  and $D_2$.

  Consider the case where $D_1$ ends with a conversion:
  \begin{equation*}
  \infer
    {    \infer {D_{1,A'}} {\istype \Gamma {A'}}
    \and \infer {D_{1,A}} {\istype \Gamma A}
    \and \infer {D_{1,t}} {\isterm \Gamma t {A'}}
    \and \infer {D_{1,\mathrm{eq}}} {\eqtype \Gamma {A'} A}}
    {\isterm \Gamma t A}
  \end{equation*}
  We apply the induction hypothesis to $D_{1,t}$ and $D_2$ to derive $\eqtype \Gamma {A'} {B}$. The desired $\eqtype \Gamma A B$ now follows from $D_{1,\mathrm{eq}}$ by symmetry and transitivity of equality.
  The case where $D_2$ ends with a conversion is symmetric, except that it does not require the use of symmetry.

  Consider the case where $D_1$ ends with a variable rule:
  \[ \infer {D_1'}
    {\infer {\istype \Gamma {\Gamma_j}} {\isterm \Gamma {\synvar j}{\Gamma_j}}}
  \]
  Because $T$ is tight $D_2$ must end with a variable or a conversion rule. We have already dealt with the latter one.
  If~$D_2$ ends with a variable rule, then $A = \Gamma_j = B$, and we may conclude $\eqtype{\Gamma}{A}{B}$ by reflexivity.

  In the remaining case $D_1$ and $D_2$ both end with instantiations of specific rules of~$T$.
  Let $\beta$ be the map which takes each symbol $S \in \Sigma$ to the corresponding symbol rule in~$T$.
  There is a unique symbol $S \in \Sigma$, such that $D_1$ and $D_2$ both end with instantiations of $\beta(S)$:
  \begin{mathpar}
    \inferrule{
      D_1
    }{
      \isterm \Gamma {\act I S(\fammap{\synmeta{i}}{i \in \args S})} {\act I C}
    }

    \inferrule{
      D_2
    }{
      \isterm \Gamma {\act J S(\fammap{\synmeta{j}}{j \in \args S})} {\act J C}
    }
  \end{mathpar}
  Of course, $\act I C$ is just $A$ and $\act J C$ is $B$, and both heads are equal to~$t$, from which it follows that
  \begin{equation*}
    S(\fammap{\act I {\synmeta i}}{i \in \args S}) = S(\fammap{\act J {\synmeta i}}{i \in \args S}),
  \end{equation*}
  and so $I$ and $J$ are equal because they match on every $i \in \args S$.
  Thus $A = \act I C = \act J C = B$, and we may derive $\eqtype \Gamma A B$ by reflexivity.
\end{proof}

We record a more economical version of uniqueness of typing, which one can afford in reasonable situations.

\begin{corollary}
  \label{cor:accaptable-uniqueness}
  If an acceptable type theory derives $\isterm{\Gamma}{e}{A}$ and $\isterm{\Gamma}{e}{B}$ then it also derives
  $\eqtype{\Gamma}{A}{B}$.
\end{corollary}

\begin{proof}
  Apply \cref{thm:presuppositions} to $\isterm{\Gamma}{e}{A}$ and $\isterm{\Gamma}{e}{B$} to obtain
  $\istype{\Gamma}{A}$ and $\istype{\Gamma}{B}$, and conclude by \cref{thm:tight-uniqueness-of-typing}.
\end{proof}

Acceptability also easily gives us uniqueness of typing for term equalities.

\begin{corollary}
  If an acceptable type theory derives $\eqterm{\Gamma}{s}{t}{A}$ and $\eqterm{\Gamma}{s}{t}{B}$ then it also derives $\eqtype{\Gamma}{A}{B}$.
\end{corollary}

\begin{proof}
  Again, apply \cref{thm:presuppositions} to get $\istype{\Gamma}{A}$ and $\istype{\Gamma}{B}$, and conclude by \cref{thm:tight-uniqueness-of-typing}.
\end{proof}

\subsection{An inversion principle}

Given the fact that a judgement $\Gamma \types J$ is derivable, to what extent can a derivation of it be constructed just from the information given in the judgement? We show in this section that, for sufficiently well-behaved type theories, one can read off the proof-relevant part of a derivation from the head of~$J$. The proof-irrelevant parts are the applications of conversion rules, and subderivations of equalities. The former may be arranged to always appear just once after variable and symbol rules, while the latter must be dealt with on a case-by-case basis, as a particular type theory may or may not possess an algorithm that checks derivability of equalities.

When we attempt to reconstruct a derivation from a judgement, the first obstacle we face is what types should be given to the subterms appearing in a judgement. For the variables the answer is clear, while for symbol expressions it is natural to use the types dictated by the corresponding rules, as follows.

\begin{definition}
  Let $T$ be a tight raw type theory over $\Sigma$ and $\beta$ the assignment of rules to the symbols of~$\Sigma$.
  Thus for each term symbol $S \in \Sigma$, the conclusion of $\beta(S)$ takes the form
  $
    \isterm{}
    {\genapp{S}}{A_S}
  $
  for some $A_S \in \Expr{\Ty}{\mvextend{\Sigma}{\arity{R}}}{\emptyscope}$.
  Given a term expression $t \in \Expr{\Tm}{\Sigma}{\Gamma}$, its \defemph{natural type} $\natty{\Gamma}{t} \in \Expr{\Ty}{\Sigma}{\Gamma}$ is defined by
  \begin{equation*}
    \natty{\Gamma}{\synvar{i}} \defeq \Gamma_i
    \qquad\text{and}\qquad
    \natty{\Gamma}{S(e)} \defeq \act{e} A_S,
  \end{equation*}
  %
  %
  where we used $e \in \prod_{i \in \args S} \Expr{\argclass{S}{i}}{\Sigma}{\sumscope{\gamma}{\argbinder{S}{i}}}$ as an instantiation, so that $\act{e} A_S$ is the expression in which each $\synmeta{i}(e')$ is replaced by $\tca{(e')} e_i$.
\end{definition}

To put it more simply, the natural type of $S(e)$ is the type one obtains by applying the symbol rule for~$S$ to the premises determined by~$e$.

\begin{theorem}[Inversion principle]
  \label{thm:inversion-principle}%
  Let $T$ be an acceptable type theory over~$\Sigma$.
  \begin{enumerate}

  \item If $T$ derives $\isterm{\Gamma}{\synvar{i}}{A}$ then it does so by an application of a variable rule, followed by a conversion:
    \begin{equation*}
      \infer{
        \infer{D'}{\isterm{\Gamma}{\synvar{i}}{\Gamma_i}}
        \\
       \infer{D''}{\eqtype{\Gamma}{\Gamma_i}{A}}
      }{
        \isterm{\Gamma}{\synvar{i}}{A}
      }
    \end{equation*}
  \item If $T$ derives $\isterm{\Gamma}{S(e)}{A}$ then it does so by an application of the symbol rule for~$S$, followed by a conversion:
    \begin{equation*}
      \infer{
        \infer{D'}{\isterm{\Gamma}{S(e)}{\natty{\Gamma}{S(e)}}}
        \\
       \infer{D''}{\eqtype{\Gamma}{\natty{\Gamma}{S(e)}}{A}}
      }{
        \isterm{\Gamma}{S(e)}{A}
      }
    \end{equation*}
  \item If $T$ derives $\istype{\Gamma}{S(e)}$ then it does so by an application of the symbol rule for~$S$.
  \end{enumerate}
\end{theorem}

\begin{proof}
  Let $T$ be an acceptable type theory over~$\Sigma$, and $\beta$ the assignment of symbol rules to the symbols of~$\Sigma$.
  To establish the first two claims, we proceed by induction on a substitution-free derivation $D$, which exists by \cref{thm:elimination-substitution}.

  If $D$ ends with a variable rule,
  \begin{equation*}
    \infer{
      \infer{D}{\istype{\Gamma}{\Gamma_i}}
    }{
      \isterm{\Gamma}{\synvar{i}}{\Gamma_i}
    }
  \end{equation*}
  then we obtain the desired derivation by attaching a dummy conversion rule:
  \begin{equation*}
    \infer{
      \infer{
        \infer{D}{\istype{\Gamma}{\Gamma_i}}
      }{
        \isterm{\Gamma}{\synvar{i}}{\Gamma_i}
      }
      \\
      \infer{
        \infer{D}{\istype{\Gamma}{\Gamma_i}}
      }{
        \eqtype{\Gamma}{\Gamma_i}{\Gamma_i}
      }
    }{
      \isterm{\Gamma}{\synvar{i}}{\Gamma_i}
    }
  \end{equation*}
  Otherwise, $D$ ends with an application of the conversion rule
  \begin{equation*}
    \infer{
      \infer{D'}{\isterm{\Gamma}{\synvar{i}}{B}}
      \\
      \infer{D''}{\eqtype{\Gamma}{B}{A}}
    }{
      \isterm{\Gamma}{\synvar{i}}{A}
    }
  \end{equation*}
  We apply the induction hypothesis to $D'$ to obtain a derivation of the form
  \begin{equation*}
    \infer{
      \infer{D^*}{\isterm{\Gamma}{\synvar{i}}{\Gamma_i}}
      \\
      \infer{D^{**}}{\eqtype{\Gamma}{\Gamma_i}{B}}
    }{
      \isterm{\Gamma}{\synvar{i}}{B}
    }
  \end{equation*}
  Using the transitivity rule, we combine $D^{**}$ and $D''$ into a derivation of $\eqtype{\Gamma}{\Gamma_i}{A}$, which can then be used together with $D^{*}$ to get the desired form of derivation.

  If $D$ ends with an application of the symbol rule $\beta(S)$,
  \begin{equation*}
    \infer{D'}{\isterm{\Gamma}{S(e)}{\natty{\Gamma}{S(e)}}},
  \end{equation*}
  then by \cref{thm:presuppositions} there is a derivation $D''$ of the presupposition $\istype{\Gamma}{\natty{\Gamma}{S(e)}}$. We apply reflexivity to $D''$ to obtain $\eqtype{\Gamma}{\natty{\Gamma}{S(e)}}{\natty{\Gamma}{S(e)}}$, and then conversion to get the desired derivation.
  Otherwise, $D$ ends with an application of a conversion rule, in which case we proceed as in the variable case.

  The third claim is trivial, because $\beta(S)$ is the only rule which can be instantiated to have the conclusion $\istype{\Gamma}{S(e)}$, apart from substitution rules, which we have dispensed with.
\end{proof}

The above theorem may be applied repeatedly to obtain a canonical form of the proof-relevant part of a derivation. The missing subderivations of equalities must be provided by other means. Also notice that it is easy enough to avoid insertion of unnecessary appeals to conversion rules along reflexivity.

A useful consequence of \cref{thm:inversion-principle} is the fact that a type of a term may be calculated directly from the term (and the symbol rules), as long as it has one.

\begin{corollary}
  In an acceptable type theory, a typeable term has its natural type.
\end{corollary}

\begin{proof}
   Whenever $\isterm{\Gamma}{e}{A}$ is derivable, then so is $\isterm{\Gamma}{e}{\natty{\Gamma}{e}}$ because its derivation appears as a subderivation in the statement of
   \cref{thm:inversion-principle}.
\end{proof}








\section{Well-founded presentations}
\label{sec:well-founded-type-theories}

So far, our type theories have omitted one typical characteristic occurring in practice: the \emph{ordering} of the presentation of the theory.
This ordering appears, implicitly or explicitly, at three levels:
\begin{enumerate}
\item The \emph{positions of a context} usually form a finite sequence, and each type depends only on the preceding part of the context.
\item The \emph{premises of each rule} typically follow some well-founded order, usually simply a finite sequence, and the boundary of each premise depends only on the earlier ones.
\item The \emph{rules of the theory} are themselves well-founded, and each rule depends only on the earlier rules. This order is quite often infinite, in for instance theories with hierarchies of universes, and need not be total, as seen in the example below.
\end{enumerate}

At each of these three levels, the “depends only on” holds in two senses:
\begin{enumerate}
\item \emph{Raw expressions}: each type expression of a context uses only the preceding variables; in a rule, the expressions of each premise boundary only use previously-introduced metavariables; and in a theory, the raw premises and boundary of a rule only use previously-introduced symbols of the theory.
\item \emph{Derivations for presuppositivity}: each type expression in a context is a derivable type over just the preceding part of the context; each premise of a rule can be checked well-formed using just the preceding premises; and so can each rule using just the earlier rules.
\end{enumerate}

\begin{example}
The $\symPi$-formation rule uses no symbols of the signature in its premises or boundary, and relies only on structural rules for its well-formedness.
The rules for $\lambda$-abstraction and function application both use $\symPi$ in their raw expressions, and depend on the $\symPi$-formation rule for their reasonability, but not on any other earlier symbols or rules.
And the $\beta$-reduction rule in turn depends on all three of these.

Similarly, there is a natural order within the rules for $\symSigma$-types.
On the other hand, neither of the $\symSigma$- or $\symPi$-type groups naturally precedes the other, and it would be unnatural to force them into a total order.
\end{example}

Traditionally, well-foundedness is treated in two different ways, depending on the levels.
Since the class of all contexts of a theory is formally defined --- contexts are ``user-definable'' --- their well-foundedness must be explicitly mandated somehow; and so it is, usually by the context judgement $\iscxt{ \Gamma }$.

Rules and theories by contrast are not “user-definable”: each development usually presents a single theory, or a few, with a specific collection of rules.
The ordering on these can therefore be left entirely unstated,
but it is almost always clearly present.
The writer ensures it when setting up the theory; the reader follows it when understanding the theory, and convincing themself of its reasonableness; and it is respected in later proofs and constructions.

It would be jarring, for instance, and often logically impossible, to give the semantics of $\beta$-reduction before that of $\lambda$-abstraction.
On the other hand, it would be unsurprising if a writer introduced the rules for $\symPi$-types before those for $\symSigma$-types, but then gave semantics with $\symSigma$-types first.
Overall, the implicit partial order on rules is always respected, but no particular total extension of it is.

Defining what it means for a presentation to be ordered is a little subtler than one might expect;
we work up to it gradually, considering contexts first, then rules, and finally theories.

\subsection{Sequential contexts}
\label{sec:sequential-contexts}

In our setting, with scoped syntax, and with contexts as maps from positions to types (we henceforth refer to these as \emph{flat} contexts), traditional sequential contexts may be recovered in various ways.
They are all straightforwardly equivalent --- indeed, a sufficiently informal statement of the traditional definition could be read as any of them --- but explicitly comparing them provides a useful warmup for the less straightforward cases with rules and type theories later.

Recall that $\numscope{n}$ denotes the sum of $n \in \N$ copies of $\numscope{1}$.
When working with sequential contexts, we will identify the positions of $\numscope{n}$ with $\{0, \ldots, n-1\}$, and denote the evident “subscope inclusion” maps by $\numscopemap{i}{j} : \numscope{i} \to \numscope{j}$.

\begin{definition}[Sequential context I]
  \label{def:seq-cxt-as-wellpresentedness}
  A \defemph{raw sequential context} over a signature $\Sigma$ is a list $\Gamma
  = [\Gamma_0,\ldots,\Gamma_{n-1}]$, where $\Gamma_i \in \ExprTy{\Sigma}{\numscope{i}}$, for each $i \in \numscope{n}$.
  We write $\Gamma_{<i}$ for the initial segment $[\Gamma_0, \ldots, \Gamma_{i-1}]$.
  The \defemph{flattening} of a raw sequential context is the raw flat context of scope $\numscope{n}$ whose $i$-th type is the weakening $\rename{{\numscopemap{i}{n}}}(\Gamma_i)$.
  We typically leave flattening implicit, writing $\Gamma$ both for a sequential context and its flattening.
  
  Given a signature $\Sigma$ and a raw type theory $T$ over it,
  a raw sequential context $\Gamma$ over $\Sigma$ is \defemph{well-formed over~$T$} if for each $i \in \Gamma$, the judgement $\istype{\Gamma_{<i}}{\Gamma_i}$ is derivable.
\end{definition}

Alternately, we can define sequentiality as a property of flat contexts:

\begin{definition}[Sequential context II]
  \label{def:seq-cxt-by-variable-occurrence}%
  A raw flat context $\Gamma$ of scope $\numscope{n}$ is \defemph{sequential} if for each $i \in \numscope{n}$, all variables $\synvar{j}$ occurring in $\Gamma_i$ have $j < i$.
  Thus each $\Gamma_i$ is uniquely of the form $\rename{{\numscopemap{i}{n}}}(\overline{\Gamma_i})$,
  from which we define the initial segments $\Gamma_{<i}$ as sequential raw contexts of scope $\numscope{i}$.

  A sequential context $\Gamma$ over $\Sigma$ is \defemph{well-formed over $T$} if for each $i \in \Gamma$, the judgement $\istype{\Gamma_{<i}}{\Gamma_i}$ is derivable.
\end{definition}

Finally, we can define well-formed flat contexts via the traditional derivation rules,
without reference to raw sequential contexts.

\begin{definition}[Sequential context III]
  \label{def:seq-cxt-by-rules}%
  The property $\iscxt{\Gamma}$, read as “$\Gamma$ is \defemph{sequentially well-formed}” over a given theory, is the inductive predicate on flat contexts defined by the following closure conditions, the latter for all suitable $\Gamma$, $A$:
\begin{mathpar}
\infer{{}
}{\iscxt{\emptycxt}}
\and
\infer{\iscxt{\Gamma} \\ \istype{\Gamma}{A}}{\iscxt{\ctxextend{\Gamma}{A}}}
\end{mathpar}
\end{definition}

Each of the above definitions moreover has two possible readings: \emph{proof-relevant}, where by derivability of a judgement $\istype{\Gamma}{A}$ we mean that a specific derivation is given, and \emph{proof-irrelevant}, where we merely mean that some derivation exists.
We take the proof-relevant reading in all cases.

\begin{proposition}  \Cref{def:seq-cxt-as-wellpresentedness,def:seq-cxt-by-variable-occurrence,def:seq-cxt-by-rules} are all equivalent, as predicates on flat contexts, in both their proof-relevant and -irrelevant forms.
\end{proposition}

\begin{proof}
Essentially straightforward, given the fact, already mentioned in \cref{def:seq-cxt-by-variable-occurrence}, that the variable-occurrence constraint there precisely characterises the images of the weakenings $\rename{{\numscopemap{i}{n}}} : \ExprTy{\Sigma}{\numscope{i}} \injto \ExprTy{\Sigma}{\numscope{n}}$.
\end{proof}

Of these definitions, \cref{def:seq-cxt-by-variable-occurrence} is the simplest to state, especially if one sweeps under the rug the inverse-weakening required for defining initial segments.
However, when spelling out details carefully, this inverse-weakening is tedious to keep track of.
When we bump these definitions up to sequential rules or type theories, therefore, we will focus on approaches based on \cref{def:seq-cxt-as-wellpresentedness,def:seq-cxt-by-rules}.

\subsection{Sequential rules}
\label{sec:sequential-rules}

Next, we wish to define sequential \emph{rules}, in which premises form a finite sequence, and each refers only to the previous ones.
Analogously to \cref{def:seq-cxt-by-variable-occurrence}, the easiest version to state is to start from an ordinary raw rule, and add desirable properties, together with restrictions on how earlier parts can be used in later parts:

\begin{definition}[Sequential rule, provisional]
  Let $R$ be a tight raw rule over a signature~$\Sigma$, with premises indexed by $\numscope{n}$, and $\beta$ the bijection witnessing its tightness.
  We say that $R$ is \defemph{sequential} if for all $i \in \numscope{n}$ and $j \in \arity{R}$, if $\synmeta{j}$ appears in the $i$-th premise, then $\beta(j) < i$.

  Moreover, say that $R$ is \defemph{(sequentially) well-formed}
  over a raw type theory~$T$, also over~$\Sigma$,
  when for all $i \in [n]$, the presuppositions of the $i$-th premise of~$R$
  can be derived from the premises indexed by~$[i]$.
\end{definition}

This definition is adequate, but is in several regards somewhat unsatisfying:
\begin{enumerate}
\item
  We have said here, as in \cref{def:seq-cxt-by-variable-occurrence}, that the premises are formed over the extension by all the metavariables, and their presuppositions derived from all the premises, but use only the preceding ones.
  When applying this condition, one typically wants to consider them as formed, or derived, over the extension by just the preceding initial segment. So rather than restricting them back there, and having to keep track of such restriction, it is simpler to say from the start, as in \cref{def:seq-cxt-as-wellpresentedness}, that they are formed, or derived, over those initial segments.
  
\item “Tightness” gives a redundancy of data in two ways.
  Firstly, the heads of all object-judgement premises are redundant: each head must be the corresponding metavariable, applied to all variables of its scope.
  Besides this, the arity itself is determined by the indexing family of the rule together with the scopes and forms of the premises.
\end{enumerate}

These issues can be remedied by defining sequential presentations inductively, analogously to \cref{def:seq-cxt-by-rules}, and adding each premise not as a full judgement but just its \emph{boundary}, whose head, if any, will be filled in automatically to ensure tightness by construction.

\begin{definition}
  \label{def:sequential-premise-family}%
  Given a signature $\Sigma$ and a raw type theory $T$ over it, we define inductively the \defemph{sequential premise-families} $P$ with arities $\arity{P}$, and simultaneously their \defemph{flattenings} as families of judgements over $\mvextend{\Sigma}{\arity{P}}$, as follows.

  \begin{enumerate}
  \item
    The \defemph{empty sequential premise-family} $\emptyfam$ has empty arity $\arity{\emptyfam}$, and its flattening is the empty family.
    
  \item
    Let $P$ be a sequential premise-family, with arity $\arity{P}$ and flattening $F$.
    Let $\Delta \typesboundary B$ be a boundary over $\mvextend{\Sigma}{\arity{P}}$, such that all presuppositions of $\Delta \typesboundary B$ are derivable over $\mvextend{(T+F)}{\arity{P}}$ (i.e.\ the translation of $T+F$ from $\Sigma$ to $\mvextend{\Sigma}{\arity{P}}$), and $\Delta$ is a well-formed sequential context over the same theory.

    Then there is an \defemph{extension sequential premise-family} $P;(\Delta \typesboundary B)$.
    
    If $\Delta \typesboundary B$ is an object boundary of class~$c$, then the associated arity $\arity{P;(\Delta \typesboundary B)}$ is $\arity{P} + \singletonfamily{(c,\position{\Delta})}$; or if $\Delta \typesboundary B$ is an equality boundary, $\arity{P;(\Delta \typesboundary B)}$ is just $\arity{P}$.
    
    The flattening of $P;(\Delta \typesboundary B)$ is $\famtuple{\act{\iota}(\Gamma_i \typesjudgement J_i)}{i \in I} + \singletonfamily{(\act{\iota}{\Delta}) \typesjudgement J}$, where $\iota$ is the inclusion $\arity{P} \to \arity{P;(\Delta \typesboundary B)}$, and $J$ is $\act{\iota}B$
    with the head, if any, filled by the expression $\synmeta{\star}(\fammap{\synvar{i}}{i \in \delta})$, where $\star$ is the new argument adjoined to the arity.
  \end{enumerate}
\end{definition}

\noindent%
Notationally, we will not distinguish the flattening from the sequential premise-family itself.

\begin{definition}
  \label{def:sequential-rule}%
  A \defemph{sequential rule} $\SequentialRule{P}{J}$ over $\Sigma$ and $T$ is a sequential premise-family~$P$, together with a judgement $ \Gamma \typesjudgement J$ over $\mvextend{\Sigma}{\arity{P}}$,
  the \defemph{conclusion}, whose presuppositions are derivable over $T + P$ (with $T$ translated to the metavariable extension).
  A sequential rule has an evident flattening as a raw rule.
\end{definition}

\noindent%
Again, we do not notate the flattening explicitly.
Note that as in the definition of raw rules the conclusion~$J$ has an empty context.

The addition of a boundary in the extension step of  \cref{def:sequential-premise-family} is precisely analogous to the traditional context extension rule, as in \cref{def:seq-cxt-by-rules}.
There, the extension is \emph{specified} just by a type $A$, but its \emph{effect} is to add a term-of-type judgement $\synvar{i} \of A$,
where the term is automatically determined to be a (fresh) variable, rather than specified as input to the extension.

Reading \cref{def:sequential-premise-family} with an eye towards computer-formalisation, one may note it can be formalized in several ways: as an inductive-recursive definition of a set with functions to arities and families of judgements; as an inductive family of sets indexed over pairs of an arity and a family of judgements; or an $\N$-indexed sequence of sets together with functions to arities and families, by induction on $n \in \N$, the length of the family.
These are all equivalent, by standard generalities about inductive definitions.

\begin{proposition}
  The flattening of a sequential rule with empty conclusion context is acceptable.
\end{proposition}

\begin{proof}
  Tightness is immediate, by construction of the arity of the premise-family and the heads of its object-judgement premises.
  Presuppositivity is similarly by construction, from the well-formedness conditions in the definitions of sequential rules and premise-families, with the latter inductively translated along metavariable extensions as the premise-family is built up.
\end{proof}

As we defined premise-families using just boundaries rather than complete judgements, similarly when we define well-founded type theories we will specify them using sequential rules whose conclusions have no heads.
We will also (for substitutivity) restrict attention to empty conclusion contexts.

\begin{definition}
  \label{def:rule-boundary}%
  A \defemph{sequential rule-boundary} $\RuleBoundary{P}{B}$ over $\Sigma$ and $T$ is a sequential premise-family~$P$, together with a boundary $\Gamma \typesboundary B$ over $\mvextend{\Sigma}{\arity{P}}$ with empty context, whose presuppositions are derivable over $T + P$.
\end{definition}

Rule-boundaries can of course be completed to rules, by filling in a head if required.

\begin{definition}
  \label{def:rule-boundary-realisation}%
  The \defemph{realisation} of a sequential rule-boundary $R = (\RuleBoundary{P}{B})$ as a sequential rule (or, via flattening, a raw rule) is
  defined according to the form of $B$:
  \begin{enumerate}
  \item
    If $B$ is an object boundary of class~$c$, then given a symbol $\symS \in \Sigma$ with arity $\arity{P}$ and class~$c$, the \defemph{realisation $\plug{R}{\symS}$ of $R$ with $\symS$} is the sequential rule
    \begin{equation*}
      \SequentialRule
      {P}
      {\plug{B}{\genapp{S}}}
    \end{equation*}
    given by completing $B$ with the generic application of~$S$.

    \item
    If $B$ is an equality, no further input is required: the \defemph{realisation of $R$} is just $\RuleBoundary{P}{B}$ with $B$ viewed as a judgement.
\end{enumerate}
\end{definition}

This gives, by construction:

\begin{propositionwithqed}
  The realisation of an object rule-boundary for~$\symS$ yields a symbol rule for~$\symS$.
\end{propositionwithqed}

\begin{example}
  The sequential rule-boundary
  \begin{equation*}
    \RuleBoundary{
      (\istype{}{\symA}) ;
      (\istype{x \of \symA}{\symB(x)})
    }{
      (\istypebdry{})
    }
  \end{equation*}
  realised with the symbol $\symPi$ gives the sequential rule
  \begin{equation*}
    \SequentialRule{
      (\istype{}{\symA}) ;
      (\istype{x \of \symA}{\symB(x)})
    }{
      (\istype{}{\symPi(\symA, \symB(x))})
    }
  \end{equation*}
  whose flattening is the usual formation rule for dependent products, as in \cref{ex:pi-congruence-rule}.
  With $\symSigma$ instead of $\symPi$, it gives the formation rule for dependent sums.
\end{example}

\subsection{Well-presented rules}
\label{sec:well-presented-rules}

Sequential rules and rule-boundaries give a satisfactory treatment covering most example theories, and sufficing for many purposes, including implementation in proof assistants.
For instance, the Andromeda proof assistant~\citep{andromeda,bauer19} implements a variant of sequential rules and rule boundaries in the trusted nucleus.

Here we consider the generalisation from finite sequences to arbitrary well-founded orders, partly to encompass infinitary rules, but mainly as a warm-up for well-founded theories.

\Cref{def:well-founded-premises-shape,def:well-founded-premise-family,def:realisation-well-presented-rule} given below are rather long and pedantic, so we give first a guiding overview.
We follow the pattern first seen in \cref{def:seq-cxt-as-wellpresentedness}, where the components of the definitions must be stratified into several stages, with each stage making use of functions defined on earlier stages.

\begin{enumerate}
\item
  At the first stage, we can specify just the \emph{shape} of the family of premises.
  This consists of a well-ordered set $(P,<)$, to index the premises,
  along with for each $i \in P$, the judgement form~$\varphi_i$ and scope~$\gamma_i$ for the $i$-th premise.

  Form this data we can compute the arity of the rule, $\arity{P}$, and more generally the arity $\arity{P_{<i}}$, specifying what metavariables may occur in the $i$-th premise.
  
\item
  At the second stage, with the arities $\arity{P_{<i}}$ available, we can specify the \emph{raw syntax} of the premises.
  The $i$-th premise $P_i$ is given by a boundary $\Gamma_i \typesboundary B_i$ of form~$\varphi_i$, scope~$\gamma_i$, and written over $\mvextend{\Sigma}{\arity{P_{<i}}}$.

  From these, filling in heads of object premises as required for tightness, we can construct the flattening of $P$ as family of judgements over $\mvextend{\Sigma}{\arity{P}}$, and more generally the flattening of $P_{<i}$ over $\mvextend{\Sigma}{\arity{P_{<i}}}$.

\item
  At the third stage, with the flattenings available, we can now specify the \emph{well-formedness} conditions.
  Derivations of presuppositions of $P_i$ should be given over the ambient theory~$T$, translated up to $\mvextend{\Sigma}{\arity{P_{<i}}}$, and with (the flattenings of) preceding premises $P_{<i}$ available as hypotheses.

\item
  We are now done with the hard part.
  Having specified the premises, the conclusion is given as in the sequential case, as a well-formed boundary $B$ over $\mvextend{\Sigma}{\arity{P}}$, whose head (if $B$ is of object form) will later be filled in to yield a symbol rule.
\end{enumerate}

While the above explanation sounds plausible, it sweeps several technical subtleties under the rug.
Most importantly, since each premise is specified over its own signature $\mvextend{\Sigma}{\arity{P_{<i}}}$, we need to handle the translations between these extensions, and up to the overall extension $\mvextend{\Sigma}{\arity{P}}$.
Spelling out all details in full, we have the following definitions.

\begin{definition}
  \label{def:well-founded-premises-shape}%
  A \defemph{well-founded premises-shape} $(I, S)$ is given by
  a well-founded set $(I, {<})$, and
  a family $S = \fammap{(\varphi_i,\gamma_i)}{i \in I}$, where $\varphi_i$ is a judgement form and $\gamma_i$ a scope.
  Given these, we define:
  \begin{enumerate}
  \item
    The arity $\arity{S}$ of $S$ is the subfamily $\fammap{(\varphi_i, \gamma_i)}{\set{i \in I \such \varphi_i \in \set{\Ty, \Tm}}}$ of the object forms of~$S$.

  \item
    For each $i \in I$, the initial segment $S_{<i} \defeq \fammap{(\varphi_j, \gamma_j)}{j < i}$ is itself a well-founded premises-shape indexed by the initial segment $\initialSegment{i} \subseteq I$, and hence it also has an associated arity $\arity{S_{< i}}$.

  \item
    For each $i < j \in I$, there are evident family maps $\arity{S_{<i}} \to \arity{S_{<j}}$ and hence $S_{<i} \to S_{<j}$, satisfying evident composition conditions with each other and with the subfamily inclusions $S_{<j} \to S$.
  \end{enumerate}
\end{definition}

\begin{definition}
  \label{def:well-founded-premise-family}%
  Given a signature~$\Sigma$ and a well-founded premises-shape~$(I, S)$ as in \cref{def:well-founded-premises-shape}, a \defemph{well-founded premise-family~$P$} is given by
  a family $B = \fammap{B_i}{i \in I}$ where $B_i$ is a boundary of form $\varphi_i$ in scope $\gamma_i$, and over $\mvextend{\Sigma}{\arity{S_{<i}}}$.
  Given these, we define:
  \begin{enumerate}
  \item
    The \defemph{flattening~$P^\flat$} is the family of judgements $\fammap{P_i}{i \in I}$ over $\mvextend{\Sigma}{\arity{S}}$, where $P_i$ is the boundary $B_i$ translated along the inclusion $\mvextend{\Sigma}{\arity{S_{<i}}} \to \mvextend{\Sigma}{\arity{S}}$, and when $\varphi_i \in \set{\Ty, \Tm}$ completed with the head expression $\synmeta{i}(\fammap{\synvar{j}}{j \in \gamma_i})$.

  \item
    For each $i \in I$, the initial segment $B_{<i}$ yields a well-founded premise family $P_{<i}$ with respect to the well-founded premises-shape $S_{<i}$ indexed by the initial segment $\initialSegment{i}$. Thus it has its own flattening~$P_{<i}^\flat$.

  \item
    For each $j < i$, the judgement $P_j$ as a member of the flattening $P_{<i}^\flat$ translated along the signature inclusion $\mvextend{\Sigma}{\arity{S_{<i}}} \to \mvextend{\Sigma}{\arity{S}}$ yields the same flattening~$P_j$, but as a member of~$P^\flat$.

    This exhibits the translation of $P_{<i}^\flat$ along the inclusion $\mvextend{\Sigma}{\arity{S_{<i}}} \to \mvextend{\Sigma}{\arity{S}}$ as a subfamily of~$P^\flat$.

  \item
    Similarly, for all $k < j < i$, the judgement $P_k$ as a member of the flattening $P_{<j}^\flat$ translated along the signature inclusion $\mvextend{\Sigma}{\arity{S_{<j}}} \to \mvextend{\Sigma}{\arity{S_{<i}}}$ yields the same flattening~$P_k$, but as a member of~$P_{<i}^\flat$.

    This exhibits the translation of $P_{<j}^\flat$ along the inclusion $\mvextend{\Sigma}{\arity{S_{<j}}} \to \mvextend{\Sigma}{\arity{S_{<i}}}$ as a subfamily of $P_{<i}^\flat$.
  \end{enumerate}
\end{definition}

\begin{definition} \label{def:well-presented-premise-family}
  A well-founded premise-family $P$ as in \cref{def:well-founded-premise-family} is
  \emph{well-formed over $T$} if for each $i \in I$, there are derivations of all presuppositions of~$B_i$ from hypotheses $P_{<i}^\flat$, in the translation of $T$ to $\mvextend{\Sigma}{\arity{S_{<i}}}$.

  A \defemph{well-presented premise-family} is a well-formed well-founded premise-family~$P$.
  Its arity $\arity{P}$ is the associated arity $\arity{S}$ of the underlying premises-shape~$S$.
\end{definition}

\noindent
When no confusion can occur, we will write the flattening of a well-founded premise-family~$P$ just as $P$, rather than $P^\flat$.

\begin{definition}
  \label{def:well-presented-rule-boundary}%
  A \defemph{well-presented rule-boundary} $\RuleBoundary{P}{B}$ over $\Sigma$, $T$ consists of a well-presented premise-family~$P$ together with a boundary with empty context $\typesboundary B$ over $\Sigma + \arity{P}$, the \defemph{conclusion boundary}, such that all presuppositions of~$B$ derivable from~$P$ in the translation of~$T$ to $\Sigma + \arity{P}$.
  The \defemph{arity} of such a rule-boundary is the arity~$\arity{P}$ of its premise-family.
\end{definition}

\begin{definition}
  \label{def:realisation-well-presented-rule}%
  The \defemph{realisation} of a well-presented rule-boundary $R = (\RuleBoundary{P}{B})$ as a raw rule is defined according to the form of $B$.
  \begin{enumerate}
  \item
    If $B$ is an object boundary of class~$c$, then given a symbol $\symS \in \Sigma$ of arity $\arity{P}$ and class~$c$, the \defemph{realisation $\plug{R}{\symS}$ of $R$ with $\symS$} has premises the flattening of $P$, and conclusion $\plug{B}{\genapp{S}}$.

  \item
    If $B$ is an equality boundary, no extra input is required: the \defemph{realisation of $R$} has premises the flattening of $P$, and conclusion just $B$ viewed as an equality judgement.
  \end{enumerate}
\end{definition}

\subsection{Well-presented type theories}
\label{sec:well-presented-theories}

Finally, we reach well-foundedness for type theories.
Once again, a by now familiar pattern emerges.
It is fairly straightforward to define well-foundedness as an after-market property of acceptable type theories,
but a better definition is obtained by putting in a little more work.

We start with the simpler version.

\begin{definition}
  \label{def:well-founded-theory}%
  Let $T = \fammap{R_i}{i \in I}$ be an acceptable type theory over a signature $\Sigma$, and let $\beta : |\Sigma| \to I$ the bijection from symbols to their rules.
  Then $T$ is \defemph{well-founded} when all its rules are well-founded, and the index set $I$ has a well-founded order~$<$, such that:
  \begin{enumerate}
  \item If $\symS \in \Sigma$ appears in $R_i$ then $\beta(\symS) < i$.
  \item Each $R_j$ has derivations of presuppositions that only refer to symbols $\symS$ with $\beta(\symS) < j$ and rules~$R_i$ with $i < j$.
  \end{enumerate}
\end{definition}

For the more refined version, we follow a similar pattern to what we saw for well-presented rules in \cref{sec:well-presented-rules}, with the definition stratified into three stages:

\begin{enumerate}
  \item First, the \emph{shape}: a well-founded order (to index the rules), and the premises-shapes and judgement forms of all rules.
  This suffices to compute the signature of the theory, and of its initial segments.
  
  \item Next, the \emph{raw} part: for each rule of the theory, a well-founded premise-family, and (raw) conclusion boundaries, of the shapes and forms specified in the first stage, and over the signature of the appropriate initial segment.
  These suffice to compute the flattening of the theory as a raw type theory, and of its initial segments.

  \item Finally, the \emph{derivations} showing well-formedness of each rule over the preceding initial segment.
\end{enumerate}

Having previously given \cref{def:well-founded-premises-shape,def:well-founded-premise-family,def:well-presented-premise-family} in rather excruciating detail, we proceed here slightly more concisely, trusting the reader to be able to fill in the elided details along the lines spelled out in those definitions.

\begin{definition} \leavevmode
  \label{def:well-presented-type-theory}%
  \begin{enumerate}
  \item A \defemph{well-founded type theory shape $T$} consists of a well-founded order $(I,<)$, together with for each $i \in I$, a well-founded premises-shape $S_i$ and judgement form $\varphi_i$ (seen as the premises shape and conclusion form of the $i$th rule).

    From these, we can define the total signature $\Sigma_T$ of $T$: its symbols are just $\set{ \symS_i \in I \such \varphi_i \in \set{\Ty,\Tm}}$, with $\symS_i$ having arity $\arity{S_i}$ and class $\varphi_i$.
    Similarly, we get signatures for initial segments $\Sigma_{T_{<i}}$, and signature maps between these, and from these to $\Sigma_T$.

  \item A \defemph{well-founded raw type theory $T$} consists of a well-founded type theory shape as above (which we also call $T$), together with for each $i \in I$, a well-founded premise-family $P_i$ of shape $S_i$ and a boundary $B_i$ of form $\varphi_i$ over the signature~$\Sigma_{T_{<i}}$.

    From these, we can define the flattening $T^\flat$ of $T$ as a raw type theory over $\Sigma_T$.
    Its rules consist of the realisations of all rule-boundaries $\RuleBoundary{P_i}{B_i}$, using (when $i$ is of object form) the symbol $\symS_i$, together with the associated congruence rules of the object rules thus added.
    Similarly, we obtain the flattening $T_{<i}^\flat$ of each initial segments of $T$, as a raw type theory over $\Sigma_{T_{<i}}$.

  \item A \defemph{well-presented type theory $T$} consists of a well-founded raw type theory $T$ as above that is additionally \defemph{well-formed}, in that it is equipped with, for each $i$, derivations exhibiting $\RuleBoundary{P_i}{B_i}$ as a well-formed rule-boundary over $T_{<i}^\flat$.
  \end{enumerate}
\end{definition}

\noindent As with well-presented rules, we will not notate flattening, when there is no ambiguity.

\begin{proposition}
  The flattening of a well-presented type theory $T$ is acceptable and well-founded.
\end{proposition}

\begin{proof}
  Well-foundedness, tightness, substitutivity, and congruousness are immediate by construction.
  Presuppositivity is almost as direct, requiring just translation of well-formedness of rules from the signatures $\Sigma_{T_{<i}}$ and theories $T_{<i}$ up to the full signature $\Sigma$ and theory~$T$.
\end{proof}

\subsection{The well-founded replacement}
\label{sec:well-founded-replacement}

\Cref{def:theory-good-properties} of acceptable type theories allows cyclic references of three kinds: between types of a context, premises of a rule, or rules of a type theory.
We shall not concern ourselves with the former two, since type theories occurring in practice all avoid them by using sequential contexts and sequential rules from \cref{sec:sequential-contexts,sec:sequential-rules}. We address the latter one, to
vindicate our design choices from earlier sections, to demonstrate that our setup supports non-trivial meta-theoretic methods, and to give an interesting new construction that likely has further applications.

For the remainder of this section, all contexts, raw rules, and rule-boundaries are presumed to be sequential. Also, it will be convenient to speak of a raw theory~$T$ without explicitly displaying its underlying signature. When we need to refer to the signature, we do so by writing $\Sigma_T$.

In \cref{ex:type-in-type} an acceptable type theory was rectified to a well-founded one by the introduction of a new symbol and an equation. When one looks at other specific examples the same strategy works, possibly with the introduction of several symbols and equations.
In order to present a general method we first lay some category-theoretic groundwork. We save the adventure of spiralling into the depths of category theory for another day, and instead establish just enough structure to keep the syntactic constructions organized.

\begin{definition}
  \label{def:raw-syntax-map}%
  A \defemph{raw syntax map} $f : \Sigma \to \Sigma'$ is given by a family of expressions
  $f_S \in \Expr{\class{S}}{\mvextend{\Sigma'}{\args{S}}}{\emptyscope}$, one for each
  $S \in \Sigma$.
  Such a map acts on $e \in \Expr{c}{\Sigma}{\gamma}$ to give $\act{f} e \in \Expr{c}{\Sigma'}{\gamma}$ by
  \begin{equation}
    \label{eq:raw-syntax-map}%
    \act{f} (\synvar{i}) \defeq \synvar{i}
    \qquad\text{and}\qquad
    \act{f} (S(e)) \defeq \act{(\act{f} \circ e)} f_S.
  \end{equation}
  The \defemph{metavariable extension} $\mvextend{f}{\alpha} : \mvextend{\Sigma}{\alpha} \to \mvextend{\Sigma'}{\alpha}$ by arity~$\alpha$ is the raw syntax map defined by
  \begin{equation*}
    (\mvextend{f}{\alpha})_S \defeq f_S
    \qquad\text{and}\qquad
    (\mvextend{f}{\alpha})_{\synmeta{i}} \defeq 
    \synmeta{i}(\fammap{\synvar{j}}{j \in \argbinder{\alpha}{i}})
  \end{equation*}
\end{definition}

In words, a raw syntax map $\Sigma \to \Sigma'$ interprets each symbol in~$\Sigma$ as a suitable compound expression over~$\Sigma'$, the interpretation extends compositionally to all expressions over~$\Sigma$, and metavariable extensions act on such a map by extending it trivially.

Let us unravel the second clause in~\eqref{eq:raw-syntax-map}, as it is a bit terse. Given a symbol~$S \in \Sigma$ and its arguments
$e \in
   \prod_{i \in \args S} 
   \Expr{\argclass{S}{i}}{\Sigma}{\sumscope{\gamma}{\argbinder{S}{i}}}
$,
the composition $\act{f} \circ e$ takes each $i \in \args{S}$ to
$\act{f} e_i \in 
   \Expr{\argclass{S}{i}}{\Sigma'}{\sumscope{\gamma}{\argbinder{S}{i}}}
$ -- it is an instantiation of arity $\args{S}$, which thus acts on~$f_S$ to yield an expression $\act{(\act{f} \circ e)} f_S \in \Expr{\class{S}}{\Sigma'}{\gamma}$, where we took into account that $\sumscope{\gamma}{\emptyscope} = \gamma$.

The action of a raw syntax map evidently extends from expressions to context, judgements,
and boundaries, and thanks to the metavariable extensions also to raw rules and rule-boundaries.

\begin{proposition}
  Signatures and raw syntax maps form a category:
  \begin{itemize}

  \item
    The identity morphism $\idmap[\Sigma] : \Sigma \to \Sigma$ takes $S \in \Sigma$ to the generic application~$\genapp{S}$.

  \item
    The composition of $f : \Sigma \to \Sigma'$ and $g : \Sigma' \to \Sigma''$ is the
    map $g \circ f : \Sigma \to \Sigma''$ that interprets each $S \in \Sigma$ as $(g \circ f)_S \defeq \act{g} f_S$.
  \end{itemize}
  Raw syntax map actions are functorial.
\end{proposition}

\begin{proof}
  A straightforward application of the basic properties of instantiations.
\end{proof}

Given raw theories $T$ and $T'$, a raw syntax map $f : \Sigma_T \to \Sigma'_T$ between the underlying signatures may be entirely unrelated to~$T$ and~$T'$. Requiring it to map derivable judgements in~$T$ to derivable judgements in~$T'$ helps, but ignores the fact that raw theories are families of raw rules, not derivable judgements. Here is a better definition.

\begin{definition}
  A \defemph{raw theory map} $f : T \to T'$ is a raw syntax map $f : \Sigma_T \to \Sigma'_T$ on the underlying signatures which maps each specific rule~$R$ of~$T$ to a derivation of $\act{f} R$ in~$T'$.
\end{definition}

\begin{proposition}
  Raw theories and raw theory maps form a category $\RawTh$.
  A raw theory map $f : T \to T'$ acts functorially on a derivation $D$ of $\Gamma \typesjudgement J$ in~$T$ to give a derivation $\act{f} D$ of $\act{f} \Gamma \typesjudgement \act{f} J$ in $T'$.
\end{proposition}

\begin{proof}
  Let us first describe the action of $f$ on a derivation $D$ of $\Gamma \types J$.
  We proceed by recursion on the structure of~$D$.

  Structural rules are mapped to the corresponding structural rules, e.g., if $D$ concludes with a variable rule $\isterm{\Gamma}{\synvar{i}}{\Gamma_i}$ then $\act{f} D$ concludes with the variable rule $\isterm{\act{f} \Gamma}{\synvar{i}}{\act{f} \Gamma_i}$, and similarly for other structural rules.

  Consider the case where $D$ ends with an instantiation $\act{I} R$ of a specific rule~$R$ of~$T$, whose premises are $\fammap{\Gamma_k \types J_k}{k \in K}$. Suppose~$f$ takes $R$ to the derivation $D_R$ of $\act{f} R$ in~$T'$.
  First we recursively map each derivation $D_k$ of the $k$-th premise $\act{I} \Gamma_k \types \act{I} J_k$ to a derivation $\act{f} D_k$ of $\act{f} (\act{I} \Gamma_k \types \act{I} J_k)$ in~$T'$, which is the same as $\act{(\act{f} I)} \Gamma_k \types \act{(\act{f} I)} J_k$, because acting by~$I$ and then by~$f$ is the same as acting by the instantiation $\act{f} I$.
  We then take $\act{f} D$ to be the derivation $\act{(\act{f} I)} R_D$ with the derivations $\act{f} D_k$ of its hypotheses grafted onto it, as in \cref{def:derivation-grafting}.

  It is straightforward to verify that the action on derivations so constructed satisfies functoriality, $\act{(g \circ f)} D = \act{g} (\act{f} D)$.

  The categorical structure of $\RawTh$ is inherited from the structure of raw syntax maps. Additionally, given composable raw theory maps $f$ and $g$, we let their composition $g \circ f$ take a specific rule $R$ to the derivation $\act{g} D_R$, where $D_R$ is the derivation of $\act{f} R$ provided by~$f$.
\end{proof}

It may happen that a raw theory map $f : T \to T'$ maps an uninhabited type to an inhabited one, or an underivable equality to a derivable one.
Let us make precise the sense in which such a map fails to be conservative,
and at the same time generalise inhabitation of types to general completion of boundaries in the presence of premises.

\begin{definition}
  Given a raw theory $T$,
  and an object rule-boundary $\RuleBoundary{P}{B}$ over~$\Sigma_T$ whose syntactic class is~$c_B$, say that $e \in \Expr{c_B}{\mvextend{\Sigma_T}{\arity{P}}}{\emptyscope}$ \defemph{realises} the rule-boundary when $\RuleBoundary{P}{\plug{B}{e}}$ is derivable in~$T$.
\end{definition}

\begin{example}
  Inhabitation of a closed type~$A$ corresponds to realisation of the rule-boundary
  $\RuleBoundary{\emptyfam}{\bdryhead : A}$.
  We can also express more general inhabitation tasks, for instance
  $\RuleBoundary{(\istype{}{\symA})}{\bdryhead \type}$ asks for a construction of a type from a type parameter~$\symA$, and $\RuleBoundary{(\istype{}{\symA}); (\istype{[\synvar{0} \of A]}{\symB(\synvar{0})})}{\bdryhead \type}$ for a $\symPi$-like higher type constructor.
\end{example}

\begin{definition}
  A raw theory map $f : T \to T'$ is \defemph{conservative} when:
  \begin{enumerate}
  \item $f$ reflects equations: if $T'$ derives the equational rule $\act{f} R$ then $T$ derives the equational rule~$R$, and

  \item $f$ reflects realisers: there is a map $r$ such that if $e$ realises the object rule-boundary $\act{f} (\RuleBoundary{P}{B})$ in~$T'$ then $r(R, e)$ realises $\RuleBoundary{P}{B}$ in~$T$.
  \end{enumerate}
\end{definition}

In the definition, we asked for a map $r$ to witness reflection of realisers in order to avoid spurious applications of the axiom of choice.

\begin{lemma}
  \label{lem:map-factor-conservative}
  Every raw theory map $f : T \to U$ factors through a
  conservative map $f^\dagger : T^\dagger \to U$
  \begin{equation*}
    \xymatrix{
      {T} \ar[rr]^{f} \ar[rd] & & {U} \\
      & {T^\dagger} \ar[ur]_{f^\dagger} &
    }
  \end{equation*}
  in a weakly universal fashion.
\end{lemma}

In the lemma, by weak universality we mean that whenever $f$ factors through a conservative map $g : V \to U$, there is a map $h : T^\dagger \to V$, not necessarily unique, such that the following diagram commutes:
\begin{equation}
  \label{eq:weakly-universal-ddagger}
  \xymatrix{
    {T} \ar[rr]^f \ar[rd] \ar@/_4ex/[rdd]_j & & {U} \\
    & {T^\dagger} \ar[ur]_{f^\dagger} \ar[d]^h & \\
    & {V} \ar@/_4ex/[ruu]_g &
  }
\end{equation}

\begin{proof}[Proof of \cref{lem:map-factor-conservative}]
  We shall construct $T^\dagger$ by adjoining new symbols to~$T$, so that whenever $e$ realises $\act{f} R$ in~$U$, the corresponding new symbol realises~$R$ in~$T$. However, the new symbols generate new rule-boundaries, so the process needs to be iterated, and we have to adjoin equations as well. The construction of~$f^\dagger : T^\dagger \to U$ thus proceeds in an inductive fashion, as follows.

  Initially the signature $\Sigma_{T^\dagger}$ is just $\Sigma_T$, the specific rules of $T^\dagger$ are those of~$T$, and~$f^\dagger$ acts like~$f$.
  We inductively extend $\Sigma_{T^\dagger}$ with new symbols, $T^\dagger$ with new specific rules, and~$f^\dagger$ with new values, as follows:
  \begin{enumerate}

  \item
    If $\RuleBoundary{P}{B}$ is an object rule-boundary in $T^\dagger$ of arity~$\alpha$ and $e$ realises $\RuleBoundary{\act{f^\dagger} P}{\act{f^\dagger} B}$ in $U$, then we extend $\Sigma_{T^\dagger}$ with a \defemph{new symbol} $\symb{c}_{(\RuleBoundary{P}{B},e)}$ of arity $\alpha$, and $T^\dagger$ with the associated symbol rule
    $
     \SequentialRule
       {P}
       {\plug{B}{\genapp{\symb{c}}_{(\RuleBoundary{P}{B},e)}}}
    $.
    We extend $f^\dagger$ by letting it map~$\symb{c}_{(\RuleBoundary{P}{B},e)}$ to~$e$.

  \item
    If $R$ is an equational rule in $T^\dagger$ such that $\act{f^\dagger} R$ is derivable in~$U$, we extend~$T^\dagger$ with~$R$ as a specific rule.
  \end{enumerate}
  Because we assumed all contexts and premise-families to be sequential, the inductive definition is complete after countably many repetitions of the above process. Alternatively, $T^\dagger$ could be constructed as a suitable colimit.

  The map $f$ obviously factors as $f = f^\dagger \circ i$ where $i : T \to T^\dagger$ is induced by the inclusion $\Sigma_T \to \Sigma_{T^\dagger}$.

  The map $f^\dagger$ is conservative by construction. It obviously reflects equations, while an object rule-boundary~$R$
  in~$T^\dagger$, such that $e$ realises $\act{f^\dagger} R$ in~$U$, is realised by $\symb{c}_{(R, e)}$.

  It remains to be shown that the factorization is weakly universal. Consider a factorization $f = g \circ j$ through a conservative map $g : V \to U$, as in~\eqref{eq:weakly-universal-ddagger}. There exists a map~$r$, not necessarily unique, that witnesses conservativity of~$g$. The desired factor $h : T^\dagger \to V$ is defined inductively: it acts like~$j$ on symbols of~$\Sigma_T$, and takes $\symb{c}_{(R,e)}$ to $r(\act{h} R, e)$.
\end{proof}

\begin{corollary}
  \label{cor:wf-replacement}%
  A raw theory $T$ has a weakly universal conservative map $t : \wfreplace{T} \to T$
  where $\wfreplace{T}$ is well-founded.
\end{corollary}

\begin{proof}
  We take $\wfreplace{T} = \emptyfam^\dagger$ and $t = o^\dagger$,
  as in \cref{lem:map-factor-conservative},
  where $o : \emptyfam \to T$ is the unique map from the empty theory~$\emptyfam$.
  Weak universality is immediate, and well-foundedness of $\emptyfam^\dagger$ is witnessed by the inductive nature of its construction.
\end{proof}

\begin{definition}
  \label{def:wf-replacement}%
  The map $t : \wfreplace{T} \to T$ from \cref{cor:wf-replacement} is called the \defemph{well-founded replacement of~$T$}.
\end{definition}

Here, finally is the main theorem of this section.

\begin{theorem}
  \label{thm:wf-replacement-equivalence}%
  If $T$ is acceptable, the map $t : \wfreplace{T} \to T$ has a section.
\end{theorem}

The theorem establishes a form of equivalence of $\wfreplace{T}$ and $T$, because
any conservative map $r : U \to V$ with a section $s : V \to U$ is an isomorphism up to judgmental equality. Indeed, $r \circ s = \idmap_V$ because~$s$ is a section of~$r$, while $s \circ r$ is identity up to judgemental equality: given any derivable rule $\SequentialRule{P}{A \type}$ in~$U$, the rule $\SequentialRule{\act{r} P}{\act{r}(\act{s}(\act{r} A)) \judgeq \act{r} A}$ is derivable in~$V$ by reflexivity, and hence $\SequentialRule{P}{\act{s}(\act{r} A) \judgeq A}$ in~$U$ by conservativity of~$r$. An analogous argument works for term judgements.

\begin{proof}[Proof of \cref{thm:wf-replacement-equivalence}]

  The theory~$\wfreplace{T}$ has symbols of the form $\symb{c}_{(\RuleBoundary{P}{B},e)}$, where~$B$ is a closed boundary, but in the proof we will have to deal with boundaries that refer to variables.

  For this purpose, define the \defemph{(variable-to-metavariable) promotion} of an expression $e \in \Expr{c}{\Sigma}{\gamma}$ to be the expression $\mvpromote{e} \in \Expr{c}{\mvextend{\Sigma}{\simplearity \gamma}}{\emptyscope}$, cf.\ \cref{def:simple-arity}, which is~$e$ with the variables replaced by metavariables,
  \begin{equation*}
    \mvpromote{\synvar{i}} \defeq \synmeta{i},
    \qquad\text{and}\qquad
    \mvpromote{S(e)} \defeq S(\fammap{\mvpromote{e_i}}{i \in \args{S}}).
  \end{equation*}
  The associated \defemph{demotion} is the instantiation $D_\gamma \in \Inst{\Sigma}{\gamma}{\simplearity \gamma}$ which takes the metavariables back to variables, $D_\gamma(\synmeta{i}) = \synvar{i}$. Thus we have $e = \act{{D_\gamma}}{\mvpromote{e}}$.

  One level up, given a premise-family~$P$ and a context $\Gamma$ over $\mvextend{\Sigma}{\arity{P}}$, the \defemph{promotion} of $\Gamma$ is the extension $P ; \mvpromote{\Gamma}$ of~$P$ in which the variables of~$\Gamma$ are promoted to metavariables of suitable types,
  \begin{equation*}
    (P ; \mvpromote{\emptycxt}) \defeq P
    \qquad\text{and}\qquad
    (P ; \mvpromote{\ctxextend{\Gamma}{A}}) \defeq
    (P ; \mvpromote{\Gamma}) ; (\istype{}{\mvpromote{A}}).
  \end{equation*}

  We begin the construction of a section of~$t$ by defining a map $d$ which maps sequential rules and rule-boundaries from~$T$ to~$\wfreplace{T}$ by replacing compound expressions $e$ with suitable symbols~$\symb{c}_{(R,e)}$ from~$\wfreplace{T}$.
  When acting on contexts and judgements, $d$ takes a sequential rule-family~$P$ from $T$ as an additional parameter, in which case we write~$d_P$.

  The map~$d$ recurses over a premise-family in~$T$ to give a premise-family in~$\wfreplace{T}$:
  \begin{align*}
    d (\emptyfam) \defeq \emptyfam
    \qquad\text{and}\qquad
    d (P ; (\Gamma \types J)) \defeq (d(P) ; d_{P}(\Gamma \types J)).
  \end{align*}
  Similarly, it takes a sequential context $\Gamma$ over $T + P$ to one over $\wfreplace{T} + d(P)$:
  \begin{align*}
    d_P(\emptycxt) &\defeq \emptycxt, \\
    d_P(\ctxextend{\Gamma}{A}) &\defeq
    \ctxextend
      {d_P(\Gamma)}
      {(\act{{D_{\position{\Gamma}}}}
        \genapp{\symb{c}}_{((\RuleBoundary{d(P) ; \mvpromote{d_P(\Gamma)}}{\bdryhead \type}), \mvpromote{A})}
      )}.
  \end{align*}
  In the second clause, $d_P$ recurses into~$\Gamma$ and extends it with the rather intimidating
  \begin{equation*}
    \act{{D_{\position{\Gamma}}}}
        \genapp{\symb{c}}_{((\RuleBoundary{d(P) ; \mvpromote{d_P(\Gamma)}}{\bdryhead \type}), \mvpromote{A})},
  \end{equation*}
  which is just the generic application of the $\symb{c}$-symbol for the promoted~$\mvpromote{A}$, demoted back to~$\Gamma$.
  Note that $(d(P) ; \mvpromote{d_P(\Gamma)}) = d(P ; \mvpromote{\Gamma})$ and hence $\act{t} (d_P(\ctxextend{\Gamma}{A})) = \ctxextend{\Gamma}{A}$.

  It remains to explain how~$d_P$ maps a judgement $\Gamma \types J$ over $T + P$ to one over $\wfreplace{T} + d(P)$. Here too we use the same method of demoting a generic application of a symbol for a promoted expression:
  \begin{align*}
    d_P (\istype{\Gamma}{A}) &\defeq (\istype{\Gamma'}{A'}),
    \\
    d_P (\isterm{\Gamma}{t}{A}) &\defeq (\isterm{\Gamma'}{t'}{A'}),
    \\
    d_P (\eqtype{\Gamma}{A}{B}) &\defeq (\eqtype{\Gamma'}{A'}{B'}),
    \\
    d_P (\eqterm{\Gamma}{s}{t}{A}) &\defeq (\eqterm{\Gamma'}{s'}{t'}{A'}),
    \\
    \shortintertext{where}
    \Gamma' &\defeq d_P(\Gamma), \\
    A' &\defeq \act{{D_{\position{\Gamma}}}}
          \genapp{\symb{c}}_{((\RuleBoundary{d(P); \mvpromote{\Gamma'}}{\bdryhead \type}), \mvpromote{A})},
    \\
    B' &\defeq \act{{D_{\position{\Gamma}}}}
          \genapp{\symb{c}}_{((\RuleBoundary{d(P); \mvpromote{\Gamma'}}{\bdryhead \type}), \mvpromote{B})},
    \\
    t' &\defeq \act{{D_{\position{\Gamma}}}}
          \genapp{\symb{c}}_{((\RuleBoundary{d(P); \mvpromote{\Gamma'}}{\bdryhead : \mvpromote{A'}}), \mvpromote{t})},
    \\
    s' &\defeq \act{{D_{\position{\Gamma}}}}
          \genapp{\symb{c}}_{((\RuleBoundary{d(P); \mvpromote{\Gamma'}}{\bdryhead : \mvpromote{A'}}), \mvpromote{s})}.
  \end{align*}
  By having $d$ map $\bdryhead$ to $\bdryhead$ the above clauses also provide the action of~$d$ on boundaries.
  Finally, let $d$ map a sequential rule $\SequentialRule{P}{J}$ in~$T$ to the sequential rule $\SequentialRule{d(P)}{J'}$ where
  $d_P(\types J) = (\types J')$, and similarly for rule-boundaries.

  We have arranged~$d$ in such a way that $\act{t} (d(R)) = R$ for any sequential rule~$R$ over $T$.
  Moreover, if $R$ is derivable in~$T$, then $d(R)$ is derivable in~$\wfreplace{T}$, by an appeal to suitable symbol rules in~$\wfreplace{T}$. For instance, $d$ maps the rule $\SequentialRule{P}{A \type}$ to the rule $\SequentialRule{d(P)}{\genapp{\symb{c}}_{((\RuleBoundary{d(P)}{\bdryhead \type}), A)} \type}$. If the former is derivable in~$T$ then the latter is a symbol rule of~$\wfreplace{T}$.

  At last, let us define the section~$s$ of~$t$. Consider first a type symbol $S \in \Sigma_T$. Because~$T$ is acceptable, it has a unique symbol rule $\SequentialRule{P}{\genapp{S} \type}$. When we map it with $d$ we get
  \begin{equation*}
    \SequentialRule{d(P)}{\genapp{\symb{c}}_{((\RuleBoundary{d(P)}{\bdryhead \type}), \genapp{S})} \type},
  \end{equation*}
  which is a symbol rule in~$\wfreplace{T}$. We may therefore take $s_S \defeq \genapp{\symb{c}}_{((\RuleBoundary{d(P)}{\bdryhead \type}), \genapp{S})}$.

  A term symbol $S \in \Sigma_T$ is dealt with analogously. Its symbol rule takes the form $\SequentialRule{P}{\genapp{S} : A}$, which is mapped by $d$ to
  \begin{equation*}
    \SequentialRule{d(P)}{\genapp{\symb{c}}_{((\RuleBoundary{d(P)}{\bdryhead : A'}), \genapp{S})} : A'},
  \end{equation*}
  where $A' = \genapp{\symb{c}}_{((\RuleBoundary{d(P)}{\bdryhead \type}), A)}$,
  Again, this is a symbol rule in~$\wfreplace{T}$, so we may define $s_S \defeq \genapp{\symb{c}}_{((\RuleBoundary{d(P)}{\bdryhead : A'}), \genapp{S})}$.
\end{proof}

\begin{example}
  We revisit \cref{ex:type-in-type}, the type theory expressing type-in-type in a cyclic fashion as~$\isterm{}{\symb{u}}{\symb{El}(\symb{u})}$.
  Earlier we pointed out that the theory can be made well-founded by using the defined type constant $\eqtype{}{\symb{U}}{\symb{El}(\symb{u})}$.
  The well-founded replacement works much the same way, except that it is replete with many more defined symbols. The analogue of $\symb{U}$ appears already at the first stage of the construction. Indeed, the rule-boundary $\RuleBoundary{\emptyfam}{\bdryhead \type}$ is realised by $\symb{El}(\symb{u})$, hence the well-founded replacement contains the type constant $U = \symb{c}_{(\RuleBoundary{\emptyfam}{\bdryhead \type}, \symb{El}(\symb{u}))}$.
  We also have the new symbols
  \begin{equation*}
    \mathit{El} = \symb{c}_{(\RuleBoundary{(\isterm{}{\symb{a}}{U})}{\bdryhead \type}, \genapp{\symb{El}})}
    \qquad\text{and}\qquad
    \mathit{u} = \symb{c}_{(\RuleBoundary{}{\bdryhead : U}, \genapp{\symb{u}})},
  \end{equation*}
  and these suffice to express the type equation $\eqtype{}{U}{\mathit{El}(\mathit{u})}$, which is a specific rule of the well-founded replacement because it is mapped to a valid equation in the original type theory.
\end{example}



\section{Discussion and related work}
\label{sec:discussion-and-related-work}


We set out to give a detailed and general mathematical definition of dependent type theories, accomplished by an analysis of their traditional accounts.
Having completed the task, let us take stock of what has been accomplished.

We calibrated abstraction at the level that keeps a connection with concrete syntax, but also clearly identifies the category-theoretic structure underlying the abstract syntax.
As such, our work may serve as a theoretical grounding and a guideline for practical implementations of type theories on one hand, and on the other as a ladder to be climbed and discarded by those who wish to ascend to higher, more abstract viewpoints of dependent type theories.

There is not much to remark on our treatment of raw syntax, as the topic has been studied before and is well understood, except to remark that scope systems have served us well as a general approach to scoping and binding of variables.

More interesting are the subsequent stages of our definition.
Giving the definition in generality forces us to isolate and precisely define various notions that are traditionally treated only informally, and often only passed on in folklore rather than in writing. 
For instance, we articulate precisely the distinction between the syntactic specification of an inference rule, and the scheme of closure conditions that it begets --- a distinction which, once seen, was implicitly present all along, but which is not generally appreciated or consciously articulated.

We initially considered raw theories as just a stepping stone towards the definition of well-presented type theories, but have come to feel that they are of significant intrinsic interest.
Their simplicity makes them easy to work with, and even though they permit deviations from the orthodox teachings, they boast with a surprisingly rich collection of meta-theorems.

The well-behavedness properties of raw rules and raw type theories, such as presuppositivity, tightness, and congruity, took some effort to define and explain, but quickly paid off.
On a technical level, they allowed us to fine-tune the requirements that enable the various meta-theorems.
More importantly, once we incorporated them into our type-theoretic vocabulary they streamlined communication and invigorated the mind where there used to be just an uneasy adherence to heuristic techniques.

We hope that our selection of meta-theorems is illustrative enough to inspire further generally applicable meta-theorems, and comprehensive enough to relieve future designers of type theories from having to redo the work.
We have intentionally restrained any category-theoretic analysis of the landscape we explore, but it will be visible in the background to many readers, demanding future exploration of the categorical structure of the syntactic notions, both for its own sake and to connect with a general categorical semantics.
This, of course, we hope to return to in future work.

How widely does our definition of type theories cast its net? It takes little effort to enumerate many examples of interest, such as variants of Martin-Löf type theory~\citep{martin-lof:introduction}, in both its intensional and extensional incarnations, homotopy type theory~\citep{hott-book}, simple type theory~\citep{church32:_set_postul_found_logic}, some presentations of the higher-order logic of toposes~\citep{lambek-scott-book}, etc. But it is equally easy to list counter-examples: System F~\citep{system-F,Reynolds74} and related type systems that directly quantify over all types, pure type systems~\citep{pure-type-systems}, cubical type theories~\citep{cohen15:_cubic_type_theor}, cohesive type theories~\citep{cohesive-tt}, etc. Can such a diverse collection of formalisms be unified under the umbrella of an even more general definition of type theories? Doubtlessly, our work can be pushed and stretched in various directions, and we hope it will.  But we also state again that we do not intend our definitions to be definitive or prescriptive, nor consider our methods to be superior to others. After all, type theory is an open-ended idea.

\medskip


Several years ago Vladimir Voevodsky's relentless inquiries into the precise mathematical underpinnings of type theories motivated us to undertake the study of general type theories. We were hardly alone to do so. Voevodsky himself initially worked to develop the framework of B-systems and C-systems~\citep{B-systems,C-systems}, but these will remain tragically unfinished. There is by now a spectrum of various approaches to the meta-theory of type theories, which cannot be justly reviewed in the remaining space. We only mention a selection of contemporaneous developments and their relation to our work.


\paragraph*{Logical framework approaches}

When discussing and presenting this work, we have often been asked why we bothered, when logical frameworks (LF)~\citep{harper-honsell-plotkin:framework,pfenning:logical-frameworks} already give a satisfactory definition of type theories.

The main answer is that most work with logical frameworks does not give a general definition in the same sense that we are looking for.
It sets up a framework within which many type theories may be defined, but that is not quite the same thing.
Indeed, at the point when we were first embarking on the present project, no definition in a generality close to our aims had been proposed in the LF literature (to our knowledge), nor could we see how to do so using LF methods.
Since then, Uemura has succeeded in giving a very clean general definition using the LF; we discuss that work in detail below.

A significant secondary motivation for the present approach, though, was to directly recover the standard “naïve” presentation of particular type theories, in specific examples.
This problem is, by design, something that LF-based approaches bypass entirely.
One may argue --- as some have --- that this desire is misguided: that since LF-based approaches are so much cleaner, naïve syntax should be discarded as obsolete, and LF-embedded presentations of theories preferred as primitive.
We however find that view somewhat unsatisfactory, for several reasons.

Firstly, even if we \emph{should} be always reading type theories as LF presentations, we \emph{have not} been.  At least within the literature on constructive type theory in the tradition of Martin-Löf, most work still uses the naïve reading, including the work of Martin Hofmann~\citep{hofmann:syntax-and-semantics} and others~\citep{martin-lof:introduction,streicher91:_seman,hott-book}.
Or rather, most of the literature stays silent about the issue, but where the intended reading is made clear, it tends to be the naïve one.
Secondly, most work using LF approaches explicitly comments on the setup, and often gives or cites adequacy theorems.
This seems to suggest that writers agree that the correctness of LF presentations of type theories rests, in part, on their connection to the naïve presentations.
Finally, it seems very difficult to adopt a position of \emph{completely} discarding the naïve readings, and reading all presentations of type theories always as shorthands for their LF embedding.
This is because the framework is itself a type theory, whose presentation must sooner or later be given the naïve reading, rather than as embedded in a further framework --- it cannot be “turtles all the way”.

We therefore believe it is important to have both the naïve and LF-based definitions of syntax defined and developed in as wide a generality as possible (as well as other approaches, such as those of categorical logic).
The naïve approach is most natural and conceptually basic.
The LF approach is cleaner and simpler to set up, and easier to analyse and apply for some purposes.
Both should be available, and connected by an adequacy theorem, not just in special cases but in generality.


\paragraph*{Uemura’s general type theories}

Another general definition of type theories has recently been given by Taichi Uemura.
Indeed, \citep{uemura19:_gener_framew_seman_type_theor} provides two definitions, one semantic and one syntactic, and shows their correspondence with a general initiality theorem.

In terms of generality, Uemura’s definition essentially subsumes ours and indeed generalises much further, encompassing type theories with different judgement forms, while still retaining enough structure to allow proofs of important type-theoretic meta\-theorems.
It therefore quite satisfactorily solves one of the major goals we set out to solve with the present project.

On inspection, however, Uemura’s approach is sufficiently different that it is complementary with our approach, rather than subsuming it.
His syntactic definition is given via a particularly ingenious use of a logical framework; as with other LF-based approaches, this keeps the setup very clean, but means that in specific examples, it does not so closely recover the standard naïve reading of the theory in question.

It does not, therefore, address our secondary goal of taking seriously the naïve reading of syntax, and directly recovering it in examples.
We therefore hope that it should be possible in the future to connect our syntactic definition with Uemura’s by means of a generalised adequacy theorem, and show that for theories with Martin-Löf’s original four judgment forms, the two approaches are equivalent.


\paragraph*{Other general definitions of dependent type theories}
\label{sec:other-related-work}

Independently of our work, Guillaume Brunerie has proposed a general syntactic definition of dependent type theories~\citep{brunerie20}, and is formalising it in Agda~\citep{brunerie:_agda}.
His approach is very similar to ours, which we see as a welcome convergence of ideas.

Valery Isaev has also proposed a definition of dependent type theories, in \citep{isaev17:_algeb}.
His approach is semantic, avoiding syntax with binding, and defining dependent type theories as certain essentially algebraic theories extending the theory of categories with families.
The generality of his definition seems to be roughly similar to ours, but a precise comparison seems slightly subtle to state, and is beyond the scope of the present paper.


\paragraph*{Meta-theory of type theory in type theory}

When one takes type theory seriously as a foundation of mathematics, it is natural and even imperative, to develop the meta-theory of type theories in type theory itself. Whereas in the LF approach the expressivity of the ambient formalism is curbed to ensure an adequacy theorem, here we gladly trade adequacy for working in a full-fledged dependent type theory.

Such a project has been undertaken by Thorsten Altenkirch and Ambrus Kaposi~\citep{qiit-tt}. Broadly speaking, a \emph{specific} object type theory is constructed in one fell swoop as a quotient-inductive-inductive type (QIIT) that incorporates all judgement forms, the structural and the specific rules. The inductive character of the definition automatically provides the correct notion of derivation, while the ambient type theory guarantees that only derivable judgements can be constructed --- the ``raw'' stage is completely side-stepped. The quotienting capabilities favourably relate the ambient propositional equality with the object-level judgmental equality. From a semantic point of view, the construction is the type-theoretic analogue of an initial-model construction. The ingenuity of the definition allows one to prove many meta-theorems quite effortlessly, especially with the aid of a proof assistant.

Our bottom-up approach can add little to the setup in terms of abstraction, but can possibly provide useful clues on how to pass from the case-by-case presentations of object type theories to a single type whose inhabitants are (presentations of) general type theories. For instance, the type of well-presented type theories would have to improve on our staged definitions by joining them into a single mutually recursive inductive definition that would incorporate the above QIIT construction of a single object-level theory, suitably adapted, as the realisation of a well-presented theory.

\bibliographystyle{plainnat}
\bibliography{general-type-theories.bib}

\appendix

\section{Formalisation in Coq}
\label{sec:formalisation-coq}

We have partially formalised our work in the Coq proof assistant~\citep{coq} on top of the HoTT library~\citep{bauer17:_hott}. The formalisation is publicly available at~\citep{lumsdaine:_formal}, wherein further instructions are given on how to compile and use the formalisation.
The formalisation will continue to evolve in future; the description here refers to the version tagged as \texttt{\small arXiv}.

The formalisation broadly follows the structure of the paper.
\Cref{tab:dict-paper-coq} lists selected major definitions and theorems from the paper, along with the names of the corresponding items in the formalisation, if any.
Almost all material of \cref{sec:preliminaries,sec:raw-syntax,sec:typing} has been formalised, as has some but not all of \cref{sec:well-behavedness}.
Versions of the main definitions of \cref{sec:well-founded-type-theories} are also formalised, but at time of writing, their treatment in the formalisation has non-trivial differences from the definitions here; such items are marked with an asterisk.
 
\begin{table}[htbp]
  \centering
  \footnotesize
  \begin{tabular}{ll}
    \toprule
    Paper & Formalisation \\ \midrule
    Family
    (\cref{def:family})
    & \coqident{Auxiliary.Family.family}
    \\
    Closure rule
    (\cref{def:closure-rule})
    & \coqident{Auxiliary.Closure.rule}
    \\
    Closure system
    (\cref{def:closure-system})
    & \coqident{Auxiliary.Closure.system}
    \\
    Derivation
    (\cref{def:closure-system-derivation})
    & \coqident{Auxiliary.Closure.derivation}
    \\
    Scope system
    (\cref{def:scope-system})
    & \coqident{Syntax.ScopeSystem.system}
    \\
    De Bruijn scope system
    (\cref{ex:de-bruijn-scope-systems})
    & \coqident{Examples.ScopeSystemExamples.DeBruijn}
    \\
    Syntactic class
    (\cref{def:syntactic-class})
    & \coqident{Syntax.SyntacticClass.class}
    \\
    Arity
    (\cref{def:arity})
    & \coqident{Syntax.Arity.arity}
    \\
    Signature
    (\cref{def:signature})
    & \coqident{Syntax.Signature.signature}
    \\
    Signature map
    (\cref{def:signature-map})
    & \coqident{Syntax.Signature.map}
    \\
    Raw expressions
    (\cref{def:raw-syntax})
    & \coqident{Syntax.Expression.expression}
    \\
    Raw substitution
    (\cref{def:raw-substitution})
    & \coqident{Syntax.Substitution.raw\_substitution}
    \\
    Metavariable extension
    (\cref{def:metavariable-extensions})
    & \coqident{Syntax.Metavariable.extend}
    \\
    Instantiation of syntax
    (\cref{def:instantiation})
    & \coqident{Syntax.Metavariable.instantiate\_expression}
    \\
    Raw context
    (\cref{def:raw-context})
    & \coqident{Typing.Context.raw\_context}
    \\
    Raw rule
    (\cref{def:raw-rule})
    & \coqident{Typing.RawRule.raw\_rule}
    \\
    Instantiation of derivations
    (\cref{cor:instantiation-of-derivations})
    & \coqident{Typing.RawTypeTheory.instantiate\_derivation}
    \\
    Associated closure system
    (\cref{def:associated-closure-system})
    & \coqident{Typing.RawRule.closure\_system}
    \\
    Structural rules
    (\cref{def:structural-rules})
    & \coqident{Typing.StructuralRule.structural\_rule}
    \\
    Congruence rule
    (\cref{def:congruence-rule})
    & \coqident{Typing.RawRule.raw\_congruence\_rule}
    \\
    Raw type theory
    (\cref{def:raw-type-theory})
    & \coqident{Typing.RawTypeTheory.raw\_type\_theory}
    \\
    Acceptable rule
    (\cref{def:acceptable-rule})
    & (not formalised)
    \\
    Acceptable type theory
    (\cref{def:theory-good-properties})
    & \coqident{Metatheorem.Acceptability.acceptable}
    \\
    Presuppositions theorem
    (\cref{thm:presuppositions})
    & \coqident{Metatheorem.Presuppositions.presupposition}
    \\
    Admissibility of renaming
    (\cref{lem:admissibility-renaming})
    & \coqident{Metatheorem.Elimination.rename\_derivation}
    \\
    Admissibility of substitution
    (\cref{lem:admissibility-substitution})
    & \coqident{Metatheorem.Elimination.substitute\_derivation}
    \\
    Admissibility of equality substitution
    (\cref{lem:admissibility-equality-substitution})
    & \coqident{Metatheorem.Elimination.substitute\_equal\_derivation}
    \\
    Elimination of substitution
    (\cref{thm:elimination-substitution})
    & \coqident{Metatheorem.Elimination.elimination}
    \\
    Uniqueness of typing
    (\cref{thm:tight-uniqueness-of-typing})
    & (not formalised)
    \\
    Inversion principle
    (\cref{thm:inversion-principle})
    & (not formalised)
    \\
    Sequential context
    (\cref{def:seq-cxt-by-rules})
    & \coqident{ContextVariants.wf\_context\_derivation}$^{(*)}$
    \\
    Sequential rule
    (\cref{def:sequential-rule})
    & (not formalised)
    \\
    Well-presented rule
    (\cref{def:realisation-well-presented-rule})
    & (not formalised)
    \\
    Well-presented type theory
    (\cref{def:well-presented-type-theory})
    & \coqident{Presented.TypeTheory.type\_theory}$^{(*)}$
    \\
    Well-founded replacement
    (\cref{thm:wf-replacement-equivalence})
    & (not formalised)
    \\ \bottomrule
  \end{tabular}
  \caption{The correspondence between the paper and the formalisation~\citep{lumsdaine:_formal}. Items marked with~$(*)$ differ non-trivially from their counterparts in the paper.}
  \label{tab:dict-paper-coq}
\end{table}

Throughout the paper we worked rigorously but informally, and without discussing which mathematical foundation might be sufficient to carry out the constructions and proofs.
On this topic we may consult the formalisation.

Our formalisation is built on top of a homotopy type theory library with an eye towards future formalisation of the categorical semantics of type theories, but is so far agnostic with respect to commitments such as the Univalence axiom or the Uniqueness of identity proofs. The only axiom that we use is function extensionality. In other words, the code can be read in plain Coq.

The formalisation confirms that our development is constructive, there are no uses of excluded middle or the axiom of choice.

It is a bit harder to tell how many universes we have used, because Coq relieves the user from explicit handling of universes. Two seem to be enough, one to serve as a base and another to work with families over the base. The base universe can be very small, say consisting of the decidable finite types, if we limit attention to finitary syntax only.

We rely in many places on the ability to perform inductive constructions and carry out proofs by induction, and so we require some meta-theoretic support for these. Of course, there is no shortage of induction in Coq, and even a fairly weak set theory will have the capability to construct the necessary inductive structures, whereas the higher-order logic of toposes would have to be extended with $W$-types. Alternatively, we could restrict to finitary syntax, contexts and rules throughout to allow Gödelization of 
syntax and reliance on induction supplied by arithmetic.



\end{document}